\documentclass{article}
\usepackage{hyperref}
\usepackage{xcolor}
\usepackage{amsmath}
\usepackage{amsthm}
\usepackage{amssymb}
\usepackage{enumitem}
\usepackage{indentfirst}
\numberwithin{equation}{section}
\newtheorem{theorem}{Theorem}[section]
\newtheorem{definition}[theorem]{Definition}
\newtheorem{corollary}[theorem]{Corollary}

\newtheorem{lemma}[theorem]{Lemma}
\newtheorem{proposition}[theorem]{Proposition}

\newtheorem{remark}{Remark}[section]
\newenvironment{proof-th}[1][Proof]{\medskip\noindent\textbf{\textit{#1}}}{\hfill$\Box$\medskip}
\newcommand{\R}{\mathbb{R}}

\begin{document}

\title{Fractional Schr\"odinger-Poisson-Slater equations in Coulomb-Sobolev spaces}

\author{\textbf{Elisandra Gloss}\\
\small Departamento de Matem\'atica\\
\small Universidade Federal da Para\'iba\\
\small Cidade Universit\'aria, 58051-900, João Pessoa, Brazil\\
\small \textit{elisandra.gloss@academico.ufpb.br}\medskip\\\textbf{Carlo Mercuri}\\
\small Dipartimento di Scienze Fisiche, Informatiche e Matematiche\\\small 
Universit\`a di Modena e Reggio Emilia\\\small
Via Campi 213/b, 41125 Modena, Italy\\\small
\textit{carlo.mercuri@unimore.it}\medskip\\\textbf{Kanishka Perera}\\
\small Department of Mathematics\\\small Florida Institute of Technology\\\small
150 W University Blvd, Melbourne, FL 32901-6975, USA\\\small
\textit{kperera@fit.edu}\medskip\\\textbf{Bruno Ribeiro}\\
\small Departamento de Matem\'atica\\
\small Universidade Federal da Para\'iba\\
\small Cidade Universit\'aria, 58051-900, João Pessoa, Brazil\\
\small \textit{bhcr@academico.ufpb.br}
}

\date{}

\maketitle

\begin{abstract}
\noindent We prove existence and multiplicity results for the fractional Schr\"odinger–Poisson–Slater equation
\[
(-\Delta)^{s}u+\bigl(I_{\alpha}*u^{2}\bigr)u=f(|x|,u)
\quad\text{in }\mathbb{R}^{N},
\]
with $0<s<1$, $\alpha\in(1,N).$
We seek solutions in a fractional Coulomb–Sobolev space by employing new tools in critical point theory, which connect the  behavior of the nonlinearity $f$ at zero and at infinity with the scaling properties of the left hand side of the equation. In various regimes for the nonlinearity $f,$
we establish compactness results for an associated action
functional and find multiple solutions as critical points, whose number is sensitive to the interaction of $f$ with a sequence of eigenvalues $\{\lambda_k\}$ defined
via the $\mathbb{Z}_2$–cohomological index of Fadell and Rabinowitz. The use of this index, instead of the classical Krasnoselskii genus, is essential for us to use new critical group estimates and scaling-based linking geometries. In this fractional setting, new regularity results and necessary conditions for solutions to exist are also proved.  
\end{abstract}

\newpage

\tableofcontents

\newpage

\section{Introduction}

In this paper we prove existence and multiplicity results for the fractional Schr\"odinger-Poisson-Slater equation
\begin{equation}\label{P}
(-\Delta)^su+\left(I_\alpha*u^2\right)u= f(|x|,u), \quad\textrm{in}\quad\mathbb{R}^N
\tag{$\mathcal{P}$}
\end{equation}
where $s\in(0,1)$, $N\geq2,$   $\alpha\in(1,N)$ and
$f:[0,\infty)\times\R\to\mathbb{R}$ is a Caratheodory function subject to new growth conditions, as we shall describe together with our main results. Here $I_\alpha$ is the Riesz kernel of order $\alpha$, defined as
$$
I_\alpha(x)=\frac{C_\alpha}{|x|^{N-\alpha}},\forall x\in\mathbb{R}^N\backslash\{0\}, \quad\textrm{with}\quad C_\alpha = \frac{\Gamma(\frac{N-\alpha}{2})}{2^\alpha\pi^{\frac{N}{2}}\Gamma(\frac{\alpha}{2})}.
$$
We define the fractional Laplacian $(-\Delta)^s$ via the Fourier transform as
$$
\widehat{((-\Delta)^su)}(\xi) = |\xi|^{2s}\widehat{u}\,\,\,\mbox{for all} \,\,s>0.
$$
For any $s>0$, $\alpha\in(0,N)$ and $q\geq1,$ as pointed out in \cite{BE-GHI-MER-MO-SC-2018},
in the framework of the so-called fractional Coulomb-Sobolev space
\begin{equation*}\label{space definition}
    \mathcal{E}^{s,\alpha,q}(\mathbb{R}^N)=\left\{u:\mathbb{R}^N\to\R:\iint_{\mathbb{R}^N\times\mathbb{R}^N}\frac{
    |u(x)|^q|u(y)|^q}{|x-y|^{N-\alpha}}dxdy <\infty\quad\mathrm{and}\quad\int_{\mathbb{R}^N}|\xi|^{2s}|\hat{u}(\xi)|^2d\xi<\infty \right\},
\end{equation*}
the expression $|\xi|^{s}\widehat{u}$ is a well-defined tempered distribution on $\mathbb R^N\setminus\{0\}$ and can be represented by an $L^2$ function. As noted in \cite{BE-GHI-MER-MO-SC-2018, Mercuri-Moroz-VS-2016} this space is a natural functional setting for a variational formulation of interesting variants of our PDE. We recall that $\mathcal{E}^{s,\alpha,q}(\mathbb{R}^N)$ is
a Banach space equipped with the norm
$$
\|u\|=\left[\|(-\Delta)^\frac{s}{2}u\|_{L^2}^2+\left(\iint_{\mathbb{R}^N\times \mathbb{R}^N}\frac{|u(x)|^q|u(y)|^q}{|x-y|^{N-\alpha}}dxdy\right)^\frac{1}{q}\right]^\frac{1}{2}.
$$
In fact, the Coulomb space 
$$
Q^{\alpha,q}(\mathbb{R}^N)=\left\{u:\mathbb{R}^N\to\R:\iint_{\mathbb{R}^N\times\mathbb{R}^N}\frac{
    |u(x)|^q|u(y)|^q}{|x-y|^{N-\alpha}}dxdy <\infty\right\}
$$
 is also a Banach space for all $\alpha\in(0,N)$ and $q\geq1$, with norm 
$$\|u\|_{Q^{\alpha,q}(\mathbb R^N)}=\left(\iint_{\mathbb{R}^N\times\mathbb{R}^N}\frac{|u(x)|^q|u(y)|^q}{|x-y|^{N-\alpha}}dxdy\right)^{\frac{1}{2q}},
$$
and it turns out to be uniformly convex when $q>1$ and hence reflexive. 
Since the operator $T:\mathcal{E}^{s,\alpha,q}(\mathbb{R}^N)\to L^2(\mathbb{R}^N)\times Q^{\alpha,q}(\mathbb{R}^N)$, given by $T(u)=((-\Delta)^\frac{s}{2}u,u)$, is an isometry, we observe that $\mathcal{E}^{s,\alpha,q}(\mathbb{R}^N)$ is also uniformly convex. 

Hereafter we deal with the case $q=2$ and work with $E=\mathcal{E}_{rad}^{s,\alpha,2}(\mathbb{R}^N)$, the closed subspace of radial functions. 
By \cite[Theorems 1.1, 1.4 and 1.5]{BE-GHI-MER-MO-SC-2018} (see also \cite{BE-FR-VI-2014}), when $\alpha\in(1,N)$ and $4s+\alpha\neq N$, it is worth observing that a continuous embedding $E\hookrightarrow L^p(\mathbb{R}^N)$ holds in either of the following cases
\begin{eqnarray}
    \label{eq3}p\in\left(p_{\rm rad}, \frac{2N}{N-2s}\right]&\mathrm{and}&4s+\alpha>N,
   \\
    \label{eq4}p\in\left[\frac{2N}{N-2s},p_{\rm rad}\right)&\mathrm{and}&4s+\alpha<N,\nonumber
\end{eqnarray}
where 
$$
p_{\rm rad}:= 2+\frac{4s(N-\alpha)}{2s(N+\alpha-2)+N-\alpha}.
$$
Moreover, the embedding $E\hookrightarrow L^p(\mathbb{R}^N)$ is compact if $p$ is not an endpoint of the above intervals. 
We set
$$
2^*_{s}=\frac{2N}{N-2s} \quad \mbox{and} \quad 2^*_{s,\alpha}=\frac{2(4s+\alpha)}{2s+\alpha}
$$
and observe that for $\alpha>1,$ $p=2^*_{s,\alpha}$ lies in a compactness range. This property is an essential tool in our approach to existence and multiplicity (hence that $\alpha>1$), as it allows us to set up and solve a {\it scaled} nonlinear eigenvalue problem, based on which we are able to classify the possibly different asymptotic behavior of $f,$ and apply accordingly, suitable and new critical point theorems. Unless stated otherwise, we shall focus on the range of parameters
$$
s\in(0,1),\quad 1<\alpha<N\quad{and}\quad 4s+\alpha> N.
$$

We point out that the condition \(4s + \alpha \ne N\) ensures that \(2^*_{s,\alpha} \ne 2^*_s\), so that the embedding of $E$ into \(L^p(\mathbb{R}^N)\)  for \(p = 2^*_{s,\alpha}\) is compact for $\alpha>1$. We believe that the condition \(4s + \alpha > N\) is only technical and stress that it is used in this paper to ensure enough regularity for a Pohozaev-type identity to hold when the nonlinearity is \(f(u) = |u|^{2^*_{s,\alpha} - 2}u\). 
\begin{remark}
 When \(4s + \alpha < N\), we note that \(2^*_{s,\alpha}\) is supercritical with respect to the classical Sobolev exponent \(2^*_s\), bringing in additional challenges in regularity theory which we hope to overcome in future, with a unified approach. 
\end{remark}

\vspace{2mm}

Motivated by the above embedding properties, we assume that
$f:[0,\infty)\times\R\to\mathbb{R}$ is a Caratheodory function satisfying the growth conditions
\begin{equation}\label{fgrowth}
    |f(|x|,t)|\leq a_1|t|^{q_1-1}+a_2|t|^{q_2-1}+a(x),\quad\mbox{a.e}\,\,x\in\mathbb{R}^N,\,\,\mbox{all} \,\, t\in\R,
\end{equation}
for some constants $a_1,a_2>0$, $a(\cdot)\in L^{r'}(\mathbb{R}^N)$ with $r, q_1,q_2$ satisfying \eqref{eq3}, where $1/r+1/r'=1$. 

A variational formulation to problem \eqref{P} is provided introducing the action functional $\Phi: E\to\mathbb{R}$ defined as
\begin{equation*}\label{Phi}
\Phi(u)=\frac{1}{2}\int_{\mathbb{R}^N}|(-\Delta)^\frac{s}{2}u|^2dx+\frac{C_\alpha}{4}\iint_{\mathbb{R}^N\times \mathbb{R}^N}\frac{u^2(x)u^2(y)}{|x-y|^{N-\alpha}}dxdy - \int_{\mathbb{R}^N}F(|x|,u)dx,
\end{equation*}
which turns out to be $C^1.$ It is standard to check that 
\begin{equation*}
\Phi'(u)v=\int_{\mathbb{R}^N}(-\Delta)^\frac{s}{2}u(-\Delta)^\frac{s}{2}vdx+C_\alpha\iint_{\mathbb{R}^N\times\mathbb{R}^N}\frac{u^2(x)u(y)v(y)}{|x-y|^{N-\alpha}}dxdy - \int_{\mathbb{R}^N}f(|x|,u)vdx.
\end{equation*}
and that weak solutions to \eqref{P} coincide with the critical points of $\Phi$.

\vspace{2mm}

For a few decades there has been a growing interest in the study of nonlinear and nonlocal partial differential equations that arise in mathematical physics from models related to the quantum many-body, mainly dealing with the case $s=1$ see e.g. \cite{Ambrosetti, Bokanowski et al., Catto et al., Le Bris Lions, Lieb, Liebb, Mauser}, which are related to Thomas–Fermi–Dirac–von Weizs\"acker
(TFDW) models of density functional theory.  When $s = 1,$ \eqref{P} reduces to a classical Schr\"odinger–Poisson–Slater equation, on which a great deal of results is available on existence, non-existence, multiplicity, and qualitative properties of the solutions.  See, for instance \cite{Ianni-Ruiz-2012,Mercuri-Moroz-VS-2016,Mercuri-Perera,Ruiz-2006,Ruiz-ARMA-2010} and references therein.

\vspace{2mm}

When $s\neq 1$ much less is known; models in this case, mainly dealing with $s=1/2,$ have mostly been introduced to describe relativistic matter, see e.g. \cite{Lieb-Seiringer}. From a variational and a PDE perspective, the interplay between the fractional Laplacian operator and the Coulomb-Riesz potential requires a careful analysis due to the presence of various critical numbers, such as the classical Sobolev exponent $2^*_s = (2N)/(N-2s)$ related to well-known embedding properties of fractional Sobolev spaces, as well as of a new Coulomb–Sobolev exponent $2^*_{s,\alpha} = 2(4s+\alpha)/(2s+\alpha).$  As highlighted in \cite{BE-GHI-MER-MO-SC-2018, Mercuri-Moroz-VS-2016}, their critical role is mostly due to loss of local and global compactness properties of embeddings into $L^p$ spaces,  which generate a rich variational scenario not yet explored in full. Recent papers on systems related to \eqref{P} which are worth mentioning are \cite{Zh-doO-Sq-2016} and  \cite{Hu-Li-Zhao-2021}. In \cite{Zh-doO-Sq-2016}  subcritical and critical systems are considered in $\mathbb{R}^3$, with $s\in(0,1)$, $\alpha\in(0,2]$ such that $4s+\alpha\geq3$ and $f\in C^1(\mathbb{R},\mathbb{R})$ satisfying a Berestycki-Lions type condition, which allows to show existence of solutions in $H^s(\mathbb{R}^3).$ In this work, unlike in our problem \eqref{P}, the Riesz term takes the form  $\mu(I_\alpha*u^2)$,  where $\mu>0$ is a small perturbation parameter. 
In \cite{Hu-Li-Zhao-2021} existence results of groundstate solutions to \eqref{P} are obtained in the space 
$\mathcal{E}^{s,2,2}(\mathbb R^3),$ with $s\in (3/4,1),$  $f(u)=|u|^{2^*_s-2}u+\mu|u|^{p-2}u,$ $\mu>0$ and $p\in(3,2^*_s).$ In a recent paper \cite{Feng-Su-2024}, \eqref{P} is considered in a more general range of parameters, with $N \geq 2$, $s\in(0,1)$, $\alpha\in(0,1]$ such that $4s+\alpha \neq N,$ and $f(u)=\lambda|u|^{2^*_{s,\alpha}-2}u+\mu|u|^{p-2}u+|u|^{2^*_s-2}u$ where $p$ lies between $2^*_{s,\alpha}$ and $2^*_{s}.$ In that context, existence of a groundstate solution is found by constrained minimisation in $\mathcal{E}_{rad}^{s,\alpha,2}$ and on a manifold defined combining Pohozaev and Nehari type conditions, provided $\lambda\in(0,C(s,\alpha))$ and $\mu>0$ is sufficiently large. \\
Finally, we refer the reader to other recent and interesting works set on classical Sobolev spaces, such as \cite[Chapter~13]{book-Ambrosio} and \cite{Jia-Feng-Xiao-Qing,Liu-Zhang,Murcia-Siciliano,Teng-2016,Wang-Wang}.

\vspace{2mm}

Our main contribution is to extend to the fractional Laplacian operator, some recent existence and multiplicity results by Mercuri and Perera \cite{Mercuri-Perera}, which deal with the case $s=1$. Adapting their variational approach, we extend many of their results to the case $0<s<1,$ allowing in problem \eqref{P} local nonlinearities $f$ of sublinear, linear, superlinear, or critical power like behavior. Although this terminology may still be appropriate in our context, we borrow from \cite{Mercuri-Perera} terms like ``subscaled'', ``scaled'', and ``superscaled'' to reflect and highlight more effectively some novel variational feautures of our PDE in different regimes, depending on the behavior of our local nonlinearity $f$; see\eqref{classification} below.  At the same time, the word ``critical'' when referred to $f,$ is used in a fairly broad sense in relation to certain growth regimes which may allow loss of compactness  phenomena, and which are not necessarily covered by the literature on the classical critical Sobolev exponent; see e.g. \cite{Mercuri-Moroz-VS-2016}. 

\vspace{2mm}

Unlike for the scaled and subscaled cases, we stress that in the superscaled regime we will assume an Ambrosetti-Rabinowitz type condition. Because of the Coulomb-Riesz term, this condition requires that $2^*_s>4$, namely \( 4s > N \). Since we focus \( 0 < s < 1 \), this restricts us to  \( N = 2 \) and \( N = 3 \), which however are physical dimensions of major significance. Removing this condition seems an interesting issue to be addressed in future.

\vspace{2mm}

Some of the novelties of this paper are new tools which reflect the fractional nature of our PDE, such as the density of compactly supported smooth functions in fractional Coulomb–Sobolev spaces. Additional new devices we provide take the form of some nontrivial regularity estimates for weak solutions, which in turn allow us to prove a novel Pohozaev-type identity for the fractional Schr\"odinger–Poisson–Slater equation, in the spirit of some recent works; see e.g. \cite{Hu-Li-Zhao-2021,Teng-2019}.

\subsection*{Main results}
\addcontentsline{toc}{subsection}{Main results}
\noindent To state and prove our results, we introduce a \emph{scaled problem} associated with our PDE, namely
\begin{equation}\label{NE}
(-\Delta)^su+\left(I_\alpha*u^2\right)u= \lambda|u|^{2^*_{s,\alpha}-2}u, \quad\textrm{in}\quad\mathbb{R}^N.
\end{equation}
It is easy to check that $u\in E$ satisfies this PDE if and only if for all $t>0,$ $u_t:=t^\frac{2s+\alpha}{2} u(t\cdot)$ does, too. Whenever \eqref{NE} possesses a nontrivial solution, we refer to $\lambda$ as an {\it eigenvalue} to \eqref{NE}. In Section \ref{section SOp} we show that the map $(u,t)\to u_t$ is a \textit{scaling} on $E$ (see Definition \ref{def scaling}) 
 and list some relevant properties for scaled operators related to this nonlinear eigenvalue problem, which ensure the existence of a sequence $\lambda_k\nearrow \infty$ of eigenvalues for \eqref{NE} (see Theorem \ref{lambdak}). \newline
Problem \eqref{NE} yields a natural classification of possibly different growth regimes for the local nonlinearity $f$ appearing in \eqref{P}. We set
\begin{equation}\label{classification}
l_\infty=\lim_{|t|\to\infty}\frac{f(|x|,t)}{|t|^{2^*_{s,\alpha}-2}t}\quad\mbox{uniformly}\,\,a.e.\,\, x\in\mathbb{R}^N.
\end{equation}
We say that $f$ is:
\begin{enumerate}
   \item[(i)] \textit{Subscaled} if $l_\infty=0$;
  \item[(ii)] \textit{Asymptotically scaled} if $l_\infty\in\mathbb{R}\backslash\{0\}$;
 \item[(iii)] \textit{Superscaled} if $l_\infty=\pm\infty$.
\end{enumerate}

For convenience, let us write
\begin{equation}\label{f=B+g}
f(|x|,t)=\lambda|t|^{2^*_{s,\alpha}-2}t+g(|x|,t)
\end{equation}
where \( g \) satisfies either 
\begin{equation}\label{ggrowth}
|g(|x|,t)|\leq a_3|t|^{q_3-1}+a_4|t|^{q_4-1},\quad a.e.\,\, x\in\mathbb{R}^N,\,\, \mbox{all}\,\, t\in\mathbb{R},
\end{equation}
for some constants $a_3,a_4>0$ and 
\begin{equation}\label{ggrowthc}
       2^*_{s,\alpha}<q_3<q_4<2^*_{s},
\end{equation}
or
\begin{equation}\label{ggrowth2}
|g(|x|,t)|\leq a_5|t|^{q_5-1}+h(x),\quad a.e.\,\, x\in\mathbb{R}^N,\,\, \mbox{all}\,\, t\in\mathbb{R},
\end{equation}
for some constant $a_5>0$, $h\in L^{r'}(\mathbb{R}^N)$ with
 \begin{equation}\label{ggrowth2c}
        p_{\rm rad} <q_5<2^*_{s,\alpha}\quad\mbox{and}\quad p_{\rm rad}<r\leq 2^*_{s}.
 \end{equation}
We also set
$$
F(|x|,t)=\int_0^tf(|x|,\tau)d\tau\quad\mbox{and}\quad G(|x|,t)=\int_0^tg(|x|,\tau)d\tau.
$$
We are now in the position to state our results for subcritical nonlinearities, whose proofs are postponed to Section~\ref{sec subcritical}.  
We begin with the subscaled case.


\begin{theorem}\label{THsubscaled}
Assume that $f$ satisfies \eqref{fgrowth} with $p_{\rm rad} <q_1,q_2<2^*_{s,\alpha}$, \eqref{f=B+g}-\eqref{ggrowthc} hold and that $\lambda$ is not an eigenvalue of problem \eqref{NE}.
\begin{enumerate}
   \item[(i)]  If $\lambda>\lambda_1,$ then \eqref{P} has a nontrivial solution.
   \item[(ii)] If $\lambda>\lambda_2,$ then \eqref{P} has two nontrivial solutions.
\end{enumerate}
\end{theorem}

Together with the proof of this theorem (see Section~\ref{sec subcritical}), we will provide an example of a nonlinearity $f$ satisfying all the above conditions. 

The next result is about an ``asymptotically scaled problem''.
\begin{theorem}\label{th subcri AS}   
Assume that \eqref{f=B+g} and \eqref{ggrowth2}-\eqref{ggrowth2c} hold and that $\lambda$ is not an eigenvalue of
problem \eqref{NE}. Then, \eqref{P} has a solution.
\end{theorem}

A consequence of this result is the following nonlinear Fredholm-type alternative.
\begin{corollary}\label{Cor alternative}
    Consider $\lambda\in\mathbb{R}$. Then, only one of the following cases may occur:
    \begin{enumerate}
        \item[(i)] either $\lambda$ is an eigenvalue, which means that \eqref{NE} has a nontrivial solution $u\in E,$ or
        \item[(ii)] $\lambda$ is not an eigenvalue for \eqref{NE}, hence for any $h\in L^{r'}(\mathbb{R}^N)$ with $p_{\rm rad}<r\leq 2^*_s$, a solution exists for
        \begin{equation*}
(-\Delta)^su+\left(I_\alpha*u^2\right)u= \lambda|u|^{2^*_{s,\alpha}}u+h(x), \quad\textrm{in}\quad\mathbb{R}^N.
\end{equation*}

    \end{enumerate}
\end{corollary}

To deal with the superscaled case we consider $N/4<s<1$ and assume the existence of $c_0>0$ and $q\in (4,2^*_s)$ satisfying
\begin{equation}\label{Gsuper}
    c_0|t|^q\leq G(|x|,t)\leq \frac{1}{q}g(|x|,t)t\quad\mbox{for a.e.}\,\,x\in\mathbb{R}^N\,\,\mbox{and all}\,\,t\in\mathbb{R}.
\end{equation}

\begin{theorem}\label{Th-superscaled-not Eigenvalue}
Consider $N/4<s<1$. Assume that \eqref{f=B+g}-\eqref{ggrowthc} and \eqref{Gsuper} hold. If $\lambda\in\mathbb{R}$ is not an eigenvalue for equation \eqref{NE}, then problem \eqref{P} has a nontrivial solution whose energy level is positive.
\end{theorem}

In the next set of results we remove the requirement that $\lambda$ is not an eigenvalue to equation \ref{NE}, by using an abstract result based on a notion of local linking, given in \cite[Theorem 2.30]{Mercuri-Perera}, and by assuming, in the subscaled case, that  $\tilde{G}$ has a negative sign.

\begin{theorem}\label{THsubscaledG<0}
Assume that $f$ satisfies \eqref{fgrowth} with $p_{\rm rad} <q_1,q_2<2^*_{s,\alpha}$, \eqref{f=B+g}-\eqref{ggrowthc} hold and $G(|x|,t)<0$ a.e. $x\in\mathbb{R}^N$ and all $t\in\mathbb{R}\backslash\{0\}$. 
\begin{enumerate}
   \item[(i)]  If $\lambda>\lambda_1$ then \eqref{P} has a nontrivial solution.
   \item[(ii)] If $\lambda>\lambda_2$ then \eqref{P} has two nontrivial solutions.
\end{enumerate}
\end{theorem}

For subcritical problems in the superscaled case where $\lambda\geq\lambda_1$ can be an eigenvalue of equation \eqref{NE}, we assume that an additional assumption is satisfied, namely that
$G$ is continuously differentiable on $[0,\infty)\times\mathbb{R}$ and 
\begin{equation}\label{Gtauu}
  \frac{\partial G}{\partial \tau}(\tau,t)\leq0,\quad \left|\tau \frac{\partial G}{\partial \tau}(\tau,t)\right|\leq C(|t|^{\tilde{q}_3}+|t|^{\tilde{q}_4}),\quad\forall (\tau,t)\in [0,\infty)\times\mathbb{R},
\end{equation}
for some $\tilde{q}_3,\tilde{q}_4\in(p_{\rm rad}, 2^*_s]$ and $C>0$.
\begin{theorem}\label{Th-superscaled-general}
Consider $N/4<s<1$ and assume that \eqref{f=B+g}-\eqref{ggrowthc} and \eqref{Gsuper} hold. If either of the following cases occurs
\begin{enumerate} 
\item[(i)] $\lambda<\lambda_1,$

\item[(ii)] $\lambda\geq\lambda_1$ and \eqref{Gtauu} holds,
\end{enumerate}
then, \eqref{P} has a nontrivial solution.
\end{theorem}


The following results, proved in Section~\ref{sec critical problems}, concern critical nonlinearities.
We assume that $N/4<s<1$ and   $f:\mathbb{R}\to\mathbb{R}$ is a critical growth nonlinearity of the form
\begin{equation}\label{fgrowthcritical2}
  f(t) = \lambda |t|^{2^*_{s,\alpha}-2}t+ \mu|t|^{q_6-2}t+|t|^{2^*_s-2}t,\quad\mbox{for all} \,\, t\in\mathbb{R},
\end{equation}
for $\lambda\in\mathbb{R}$, $\mu\geq0$ and $ q_6\in (2^*_{s,\alpha},2^*_{s})$.

\vspace{2mm}

We stress that the results below extend those in \cite[Section 1.3.2]{Mercuri-Perera}, not only to the fractional setting, but also to more general nonlinearities $f$, which are not restricted to be only scaled or a superscaled perturbations  of a critical nonlinearity. We also refer the reader to \cite{Feng-Su-2024} for a related result, where  existence of a solution is provided in the case $0 < \alpha < 1$ and $\lambda > 0$ sufficiently small.

\begin{theorem}\label{th cri-super-mult}
Let $N/4 < s < 1$ and let $f$ satisfy \eqref{fgrowthcritical2}.  
Suppose the nonlinear eigenvalues satisfy
\(\lambda_k = \lambda_{k+1} = \dots = \lambda_{k+m-1} < \lambda_{k+m}
\quad\text{for some } k,m \ge 1.\)
Then, for every $\mu \ge 0$, there is a number $\delta_k > 0$ such that
problem~\eqref{P} possesses at least $m$ distinct pairs of nontrivial
solutions with positive energy whenever
\(\lambda \in \bigl(\lambda_k - \delta_k,\; \lambda_k\bigr).\)
\end{theorem}

As an immediate consequence of this theorem we get the following 
\begin{corollary}
Suppose $N/4 < s < 1$ and that $f$ is of the form \eqref{fgrowthcritical2}.
If $\lambda_k<\lambda_{k+1}$ for some $k\geq1$, then, for each $\mu\geq0$, there is $\delta_k>0$ such that \eqref{P} has a pair of nontrivial solutions at a positive energy level for all $\lambda\in(\lambda_k-\delta_k,\lambda_k)$.
\end{corollary}
In Theorem~\ref{th cri-super-mult}, the assumption on \(\lambda\) being close to \(\lambda_k\) can be dropped, provided \(\mu\) is sufficiently large.  
In fact, we have the following
\begin{theorem}\label{th cri l<l1}
Let $N/4 < s < 1$ and suppose $f$ satisfies \eqref{fgrowthcritical2} with  $\lambda\in\mathbb{R}$.  
For every integer $m \ge 1$ there exists a number $\mu_m>0$ such that, for all $\mu>\mu_m$,  
problem~\eqref{P} admits at least $m$ distinct pairs of nontrivial solutions with positive energy.  
In particular, the number of these solutions tends to infinity as $\mu\to\infty$.
\end{theorem}


Section~\ref{section sub near 0} is devoted to the proof of the results we now state on nonlinearities which are subscaled near the origin. Here we write $f:[0,\infty)\times\mathbb{R}\to\mathbb{R}$  as
\begin{equation}\label{f sub near 0}
  f(|x|,t) = |t|^{\beta-2}t+ g(|x|,t),
\end{equation}
where $  p_{\rm rad}<\beta< 2^*_{s,\alpha}$ and $g:[0,\infty)\times\mathbb{R}\to\mathbb{R}$ is a Caratheodory function such that
\begin{equation}\label{g sub near 0}
  |g(|x|,t)| \leq a_7|t|^{q_7-1}+ a_8|t|^{q_8-1}, \quad\mbox{for a.e.}\,\,x\in\mathbb{R}^N ,\,\, \forall t\in\mathbb{R},
\end{equation}
with 
\begin{equation}\label{q7 q8 sub near 0}
        \beta<q_7,q_8 \leq 2^*_{s}.
    \end{equation}
We also assume that $g$ is odd in $t$ and its primitive $G$ vanishing at zero satisfies
\begin{equation}\label{G sub near 0}
  G(|x|,t) \leq \tilde{C}|t|^{\tilde{\beta}}\quad\mbox{for a.e.}\,\,x\in\mathbb{R}^N \,\, \forall t\in\mathbb{R},
\end{equation}
for some  $  p_{\rm rad}<\tilde\beta< 2^*_{s,\alpha}$ and a positive constant $\tilde C$. 
Assuming that the associated functional $\Phi$ satisfies the (PS) condition, we can prove the following 
\begin{theorem}\label{th sub near 0}
 Assume that \eqref{f sub near 0}-\eqref{G sub near 0} hold, with $p_{\rm rad}<\beta,\tilde{\beta}< 2^*_{s,\alpha}$, and that $\Phi$ satisfies the $(PS)_c$ condition for all $c<0$. Then, there exists a sequence of solutions $(u_k)$ for \eqref{P} satisfying $\Phi(u_k)\nearrow0$.
\end{theorem}
As a consequence, we obtain two corollaries on problems which in fact turn out to satisfy the (PS) condition.
\begin{corollary}\label{cor1 sub near 0}
    For any $p_{\rm rad}<\beta< 2^*_{s,\alpha}$ the equation
   \begin{equation*}
(-\Delta)^su+\left(I_\alpha*u^2\right)u=  |u|^{\beta-2}u \quad\textrm{in}\quad\mathbb{R}^N
\end{equation*}
    has infinitely many solutions at negative levels.
\end{corollary}

\begin{corollary}\label{cor2 sub near 0}
  The equation
   \begin{equation*}
(-\Delta)^su+\left(I_\alpha*u^2\right)u=  |u|^{\beta-2}u+\lambda |u|^{2^*_{s,\alpha}-2}u-|u|^{q-2}u \quad\textrm{in}\quad\mathbb{R}^N
\end{equation*}
    has infinitely many solutions at negative levels for any $\lambda\in\mathbb{R}$ whenever $ p_{\rm rad}<\beta<2^*_{s,\alpha}$ and $q$ is such that
   $$
       \beta<q \leq 2^*_{s} \quad\mbox{for}\,\,\lambda\leq0\quad\mbox{or}\quad
        2^*_{s,\alpha}<q \leq 2^*_{s}\quad\mbox{for}\,\,\lambda>0.
   $$
\end{corollary}

When $\lambda\in\mathbb{R}$ is not an eigenvalue of \eqref{NE}  we can drop condition \eqref{G sub near 0} and prove a similar result even in some instances where the functional $\Phi$ is unbounded from below, as in the following
\begin{theorem}\label{subscaled-NE}
   If $\lambda\in\mathbb{R}$ is not an eigenvalue to equation \eqref{NE} and $p_{\rm rad}<\beta< 2^*_{s,\alpha}$ then equation
   \begin{equation*}
(-\Delta)^su+\left(I_\alpha*u^2\right)u=  |u|^{\beta-2}u+ \lambda|u|^{2^*_{s,\alpha}-2}u  \quad\textrm{in}\quad\mathbb{R}^N
\end{equation*}
has a sequence of solutions $(u_k)$ such that $\Phi(u_k)\nearrow0$.
\end{theorem}
In a critical regime, we have the following multiplicity result, obtained by a suitable truncation of the action functional
\begin{theorem}\label{th sub near 0 critical}
Consider $N/4<s<1$ and $\beta\in (p_{\rm rad}, 2^*_{s,\alpha})$. Then, there exists $\mu^*>0$ such that
   \begin{equation*}
(-\Delta)^su+\left(I_\alpha*u^2\right)u=  \mu|u|^{\beta-2}u + |u|^{2^*_{s}-2}u  \quad\textrm{in}\quad\mathbb{R}^N
\end{equation*}
has a sequence of solutions at negative levels, for all $\mu\in(0,\mu^*)$.
\end{theorem}

Finally, we deal with nonlinearities involving a subscaled term at the origin with a negative sign, and in a context which we may regard as asymptotically scaled at infinity.
\begin{theorem}\label{sub near 0, asympt, subcritical}
Assume that  $ p_{\rm rad}<\beta <2^*_{s,\alpha}$ and $\lambda>\lambda_k$ is not an eigenvalue of \eqref{NE}. Then the equation
   \begin{equation*}
(-\Delta)^su+\left(I_\alpha*u^2\right)u=   \lambda|u|^{2^*_{s,\alpha}-2}u -|u|^{\beta-2}u  \quad\textrm{in}\quad\mathbb{R}^N
\end{equation*}
has $k$ pairs of nontrivial solutions at positive energy levels.
\end{theorem}
In the critical case, we have the following
\begin{theorem}\label{sub near 0, asympt, critical}
Consider $N/4<s<1$, $\lambda>\lambda_k$ and $\beta\in (p_{\rm rad}, 2^*_{s,\alpha})$. Then, there exists $\mu^*>0$ such that
   \begin{equation*}
(-\Delta)^su+\left(I_\alpha*u^2\right)u=  \lambda|u|^{2^*_{s,\alpha}-2}u-\mu|u|^{\beta-2}u + |u|^{2^*_{s}-2}u  \quad\textrm{in}\quad\mathbb{R}^N
\end{equation*}
has $k$ pairs of nontrivial solutions at positive energy levels, for all $\mu\in(0,\mu^*)$.
\end{theorem}

 To conclude this introduction we finally highlight some auxiliary results which are collected in the appendix to ease the reading of the paper, and which we believe are of independent interest in the study of Coulomb-Sobolev spaces and their applications to PDEs. Most of these results require $s\in (0,1),$ a restriction which we hope to relax in future.

\begin{itemize}
  \item \textbf{Appendix~A} contains a proof of the density of smooth functions in the fractional Coulomb–Sobolev space $\mathcal{E}^{s,\alpha,q}(\mathbb{R}^{N})$.  
  \item \textbf{Appendix~B} deals with regularity of weak solutions found in the main theorems.  Beside other potential applications, with these we are able to prove a Pohozaev-type identity.
  \item \textbf{Appendix~C} provides the aforementioned Pohozaev-type identity associated with our fractional Schr\"odinger–Poisson–Slater equation.  Results of similar flavor and for functions in $L^{2}$ are available (see e.g.  \cite[Proposition 2.9]{Teng-2019} and \cite[Lemma 2.3]{Hu-Li-Zhao-2021}). However, our results here extend these to a more general setting.
\end{itemize}


\section{Scaled operators on fractional Coulomb-Sobolev spaces}\label{section SOp}

We recall the definition of a scaling on a Banach space, as given in \cite[Definition 1.1]{Mercuri-Perera}.

\begin{definition}\label{def scaling}
    Let $W$ be a reflexive Banach space. A scaling on $W$ is a continuous mapping
$$W\times [0,\infty) \to W,\quad (u,t) \mapsto u_t$$ satisfying
\begin{enumerate}[label={($H_\arabic*$)},  
                  ref=$H_\arabic*$]               
\item\label{H1} $(u_{t_1})_{t_2}=u_{t_1t_2}$ for all $u\in W$ and $t_1,t_2\geq0$,
\item\label{H2} $(\tau u)_{t}=\tau u_{t}$ for all $u\in W$, $\tau\in\mathbb{R}$ and $t\geq0$,
\item\label{H3} $u_{0}=0$ and $u_{1}=u$ for all $u\in W$, 
    \item\label{H4} $u_{t}$ is bounded on bounded sets in $W\times [0,\infty)$, 
  \item\label{H5} $\exists \sigma >0$ such that $\|u_t\|=O(t^\sigma)$ as $t\to\infty$,  uniformly in $u$ on bounded sets.
\end{enumerate}
\end{definition}
Hereafter we set
\begin{equation}\label{def theta}
\theta=\frac{(2s+\alpha)}{2}.
\end{equation}
\begin{lemma}
    Let $0<s<1$. The map $E\times [0,\infty)\to E$ given by $(u,t)\mapsto u_t:=t^\theta u(t\cdot)$
is a scaling on $E$.
\end{lemma}
 \begin{proof} From the definition of $u_t$ it is clear that \eqref{H1}-\eqref{H3} hold. We also observe that
$$
\|(-\Delta)^\frac{s}{2}u_t\|_{L^2}^2=\int_{\mathbb{R}^N}|\xi|^{2s}|\widehat{u_t}|^2d\xi=t^{4s+\alpha-N}\int_{\mathbb{R}^N}|\xi|^{2s}|\widehat{u}|^2d\xi
$$
and
$$
\iint_{\mathbb{R}^N\times\mathbb{R}^N}\frac{u_t^2(x)u_t^2(y)}{|x-y|^{N-\alpha}}dxdy=t^{4s+\alpha-N}\iint_{\mathbb{R}^N\times\mathbb{R}^N}\frac{u^2(x)u^2(y)}{|x-y|^{N-\alpha}}dxdy
$$
for all $u\in E$ and $t\geq0$. Hence
$$
\|u_t\|\leq\max\{t^{(4s+\alpha-N)/2},t^{(4s+\alpha-N)/4}\}\|u\|,\quad\forall (u,t)\in E\times [0,\infty),
$$
which implies \eqref{H4} and \eqref{H5} are satisfied with $\sigma=4s+\alpha-N$. Continuity follows then as in \cite[Lemma 3.1]{Mercuri-Perera} as a result of the density of $C^\infty_c(\mathbb{R}^N)$ in $E$ \textcolor{red} (see Proposition \ref{densityE} in the Appendix). 
\end{proof}

Recall that $\tilde{h}\in C(E,E^*)$ is a potential operator if there is a functional $\tilde{H}\in C^1(E,\mathbb{R})$, called a potential for $\tilde{h}$, such that $\tilde{H}'=\tilde{h}$. We can assume that $\tilde{H}(0)=0$, by replacing $\tilde{H}$ with $\tilde{H}-\tilde{H}(0)$ if necessary.
An odd potential operator $\tilde{h}\in C(E,E^*)$ is called a scaled operator if it maps bounded sets into bounded sets and satisfies
\begin{equation}\label{scaledoperator}
    \tilde{h}(u_t)v_t = t^\sigma \tilde{h}(u)v\quad \forall u,v\in E, \,t\geq0.
\end{equation}
As shown in \cite[Proposition 2.2]{Mercuri-Perera}, if $\tilde{h}$ is a potential operator, then its potential $\tilde{H}$ with $\tilde{H}(0)=0$ is such that
\begin{equation}\label{potential integral}
\tilde{H}(u)=\int_0^1\tilde{h}(\tau u)ud\tau,\quad\forall u\in E.
\end{equation}
In particular, $\tilde{H}$ is even if $\tilde{h}$ is odd and maps bounded sets to bounded sets whenever $\tilde{h}$ does. Moreover, if $\tilde{h}$ is a scaled operator satisfying \eqref{scaledoperator}, then
$$ 
\tilde{H}(u_t) = t^\sigma \tilde{H}(u)\quad\forall u\in E,\,\,t\geq0.
$$

With these preliminaries in place, let us now consider the odd potential operators $\mathcal{A}, \mathcal{B}:E\to E^*$, defined as $\mathcal{A}(u):=(-\Delta)^su+\left(I_\alpha*u^2\right)u$, that means,
\begin{equation}
    \mathcal{A}(u)v= \int_{\mathbb{R}^N}(-\Delta)^\frac{s}{2}u(-\Delta)^\frac{s}{2}vdx+C_\alpha\iint_{\mathbb{R}^N\times\mathbb{R}^N}\frac{u^2(x)u(y)v(y)}{|x-y|^{N-\alpha}}dxdy
\end{equation}
and  $\mathcal{B}(u):=|u|^{2^*_{s,\alpha}-2}u$, where
\begin{equation}
    \mathcal{B}(u)v= \int_{\mathbb{R}^N}|u|^{2^*_{s,\alpha}-2}uvdx.
\end{equation}
We note that the potential operators $\mathcal{A}$ and $\mathcal{B}$ satisfy \eqref{scaledoperator} with the same $\sigma=4s+\alpha-N$ such that \eqref{H5} holds, namely
$$
\mathcal{A}(u_t)v_t=t^{4s+\alpha-N}\mathcal{A}(u)v\quad{\rm and}\quad \mathcal{B}(u_t)v_t=t^{4s+\alpha-N}\mathcal{B}(u)v,
$$
for all $u,v\in E$ and $t\geq 0$. Moreover, the embedding $E\hookrightarrow L^{2^*_{s,\alpha}}(\mathbb{R}^N)$ ensures that $\mathcal{B}$ is bounded on bounded sets, and this clearly holds for $\mathcal{A}$. It follows then that $\mathcal{A}$ and $\mathcal{B}$ are scaled operators. \\ Let us now define the potentials $I,J:E\to\mathbb{R}$ for $\mathcal{A}$ and $\mathcal{B}$, respectively, as  
\begin{equation*}
I(u)=\frac{1}{2}\|(-\Delta)^\frac{s}{2}u\|_{L^2}^2+\frac{C_\alpha}{4}\iint_{\mathbb{R}^N\times\mathbb{R}^N}\frac{u^2(x)u^2(y)}{|x-y|^{N-\alpha}}dxdy,\quad J(u)=\frac{1}{2^*_{s,\alpha}}\int_{\mathbb{R}^N}|u|^{2^*_{s,\alpha}}dx.
\end{equation*}

Observe that, if $u\in E$ is a solution for $\mathcal{A}(u)=\lambda \mathcal{B}(u)$, i.e. a solution to \eqref{NE},
then $u_t$ is also a solution, for all $t\geq0$.  

\vspace{0.5cm}

We may check now that the following properties are enjoyed by $\mathcal A,\mathcal B, I$ and $J$.
\begin{enumerate}[label={($H_{\arabic*}$)},  
                  ref=$H_{\arabic*}$,start=6]               
    \item\label{H6} $\mathcal{A}(u)u>0$ for all $u\in E\backslash\{0\}.$
    \item\label{H7}  If $u_n\rightharpoonup u$ in $E$ and $\mathcal{A}(u_n)(u_n-u)\to0$ then $(u_n)$ has a subsequence that converges strongly to u.
    \item\label{H8} $\mathcal{B}(u)u>0$ for all $u\in E\backslash\{0\}.$
    \item\label{H9} If $u_n\rightharpoonup u$ in $E$ then $\mathcal{B}(u_n)\to \mathcal{B}(u)$ in $E^*$.
    \item\label{H10} $I$ is coercive, i.e., $I(u)\to\infty$ as $\|u\|\to\infty.$ 
    \item\label{H11} For each $u\in E\backslash\{0\}$ there is an unique $t>0$ such that $I(tu) = 1.$
    \item\label{H12} Every solution $u$ of equation \eqref{NE} satisfies $I(u)=\lambda J(u).$
\end{enumerate}
Note that \eqref{H7} is satisfied by Lemma \ref{verify H7} and  \eqref{H12} by Lemma \ref{verify H12} below. \eqref{H6}, \eqref{H8} and \eqref{H10} are obvious. As far as \eqref{H9}, note that since $\alpha\in(1,N)$ and $4s+\alpha>N$, we have $p_{\rm rad}<2^*_{s,\alpha}<2^*_s$ and so, by compactness of the embedding $E\hookrightarrow L^{2^*_{s,\alpha}}(\mathbb{R}^N)$ we have $\mathcal{B}(u_n ) \to \mathcal{B}(u)$ in $E^*$ if $u_n \rightharpoonup u$ in $E$. Finally, we may check that \eqref{H11} is also satisfied, as
$$
I(tu)=\frac{t^2}{2}\|(-\Delta)^\frac{s}{2}u\|_{L^2}^2+\frac{C_\alpha t^4}{4}\iint_{\mathbb{R}^N\times \mathbb{R}^N}\frac{u^2(x)u^2(y)}{|x-y|^{N-\alpha}}dxdy,
$$
is strictly increasing in $t\geq0$, for any fixed $u\in E\backslash\{0\}.$

\begin{lemma}\label{verify H7}
    Assume that $u_n\rightharpoonup u$ in $E$ and $\mathcal{A}(u_n)(u_n-u)\to0$. Then $(u_n)$ has a convergent subsequence.
\end{lemma}
\begin{proof}
   Arguing as in \cite[Lemma 3.2]{Mercuri-Perera}, note that

   $$
  \mathcal{A}(u_n)(u_n-u)= \int_{\mathbb{R}^N}(-\Delta)^\frac{s}{2}u_n(-\Delta)^\frac{s}{2}(u_n-u)dx +C_\alpha\iint_{\mathbb{R}^N\times\mathbb{R}^N}\frac{u_n^2(x)u_n(y)(u_n-u)(y)}{|x-y|^{N-\alpha}}dxdy.
   $$
   Since $u_n\rightharpoonup u$ in $E$, we also have $(-\Delta)^\frac{s}{2}u_n\rightharpoonup (-\Delta)^\frac{s}{2}u$ weakly in $L^2(\mathbb{R}^N)$. This yields
   $$
   \int_{\mathbb{R}^N}(-\Delta)^\frac{s}{2}u(-\Delta)^\frac{s}{2}u_ndx \to  \int_{\mathbb{R}^N}|(-\Delta)^\frac{s}{2}u|^2dx,\quad\mbox{as}\,\,n\to\infty,
   $$
    and hence
    \begin{equation*}\label{Hs control}
 \int_{\mathbb{R}^N}(-\Delta)^\frac{s}{2}u_n(-\Delta)^\frac{s}{2}(u_n-u)dx=  \|(-\Delta)^\frac{s}{2}u_n\|_{L^2}^2-\|(-\Delta)^\frac{s}{2}u\|_{L^2}^2+o_n(1)= \|(-\Delta)^\frac{s}{2}(u_n-u)\|_{L^2}^2+o_n(1).
   \end{equation*}
 
Setting
$$
T(u,v,w,z) = C_\alpha\iint_{\mathbb{R}^N\times\mathbb{R}^N}\frac{u(x)v(x)w(y)z(y)}{|x-y|^{N-\alpha}}dxdy,
   $$
as in \cite[Lemma 2.3]{Ianni-Ruiz-2012} (see also \cite[Lemma 2.4]{Hu-Li-Zhao-2021}) one sees that $T(u_n,v_n,w_n,z) \to T(u,v,w,z)$ whenever $u_n\rightharpoonup u$, $v_n\rightharpoonup v$ and $w_n\rightharpoonup w$ in $E$ and $z\in E$ is fixed. Hence,
$$
\iint_{\mathbb{R}^N\times\mathbb{R}^N}\frac{u_n^2(x)u_n(y)(u_n-u)(y)}{|x-y|^{N-\alpha}}dxdy=
\iint_{\mathbb{R}^N\times\mathbb{R}^N}\frac{u_n^2(x)u_n^2(y)}{|x-y|^{N-\alpha}}dxdy
-\iint_{\mathbb{R}^N\times\mathbb{R}^N}\frac{u^2(x)u^2(y)}{|x-y|^{N-\alpha}}dxdy+o_n(1).
$$
Passing if necessary to a subsequence, since $u_n(x)\to u(x)$ almost everywhere, by the nonlocal Brezis-Lieb lemma (see \cite[Proposition 4.1]{Mercuri-Moroz-VS-2016}) we obtain
$$
\iint_{\mathbb{R}^N\times\mathbb{R}^N}\frac{u_n^2(x)u_n^2(y)}{|x-y|^{N-\alpha}}dxdy
-\iint_{\mathbb{R}^N\times\mathbb{R}^N}\frac{u^2(x)u^2(y)}{|x-y|^{N-\alpha}}dxdy
\geq \iint_{\mathbb{R}^N\times\mathbb{R}^N}\frac{(u_n-u)^2(x)(u_n-u)^2(y)}{|x-y|^{N-\alpha}}dxdy+o_n(1).
$$
Thus,
$$
\mathcal{A}(u_n)(u_n-u)\geq \|(-\Delta)^\frac{s}{2}(u_n-u)\|_{L^2}^2+ C_\alpha\iint_{\mathbb{R}^N\times\mathbb{R}^N}\frac{(u_n-u)^2(x)(u_n-u)^2(y)}{|x-y|^{N-\alpha}}dxdy+o_n(1),
$$
which is enough to conclude.
\end{proof}

The following lemma is based on a Pohozaev-type identity which is stated and proved in Lemma \ref{pohozaev} in the Appendix.
\begin{lemma}\label{verify H12}
Suppose that $\mathcal{A}(u)=\lambda \mathcal{B}(u)$ for some $u\in E\backslash\{0\}$ and $\lambda>0$. Then, $I(u)=\lambda J(u),$ namely \eqref{H12} is satisfied.
\end{lemma}
\begin{proof}
Since $u\in E\backslash\{0\}$ is a solution for \eqref{NE}, it holds that $\mathcal{A}(u)u=\lambda \mathcal{B}(u)u,$ and so by \eqref{H6} and \eqref{H8} the condition $\lambda>0$ is necessary.
By Lemma \ref{pohozaev} it holds that 
\begin{equation}\label{Poh}
\frac{N-2s}{2}\|(-\Delta)^{\frac{s}{2}}u\|_2^2+\frac{C_\alpha(N+\alpha)}{4}\iint_{\mathbb{R}^N\times\mathbb{R}^N}\frac{u^2(x)u^2(y)}{|x-y|^{N-\alpha}}dxdy
=
\frac{\lambda N}{2^*_{s,\alpha}}\int_{\mathbb{R}^N} |u|^{2^*_{s,\alpha}}dx.
\end{equation}
 Moreover, using $u$ as a test function we have
\begin{equation}\label{NES}
\|(-\Delta)^\frac{s}{2}u\|_2^2+{C_\alpha}\iint_{\mathbb{R}^N\times\mathbb{R}^N}\frac{u^2(x)u^2(y)}{|x-y|^{N-\alpha}}dxdy= \lambda\int_{\mathbb{R}^N}|u|^{2^*_{s,\alpha}}dx.
\end{equation}
 Multiplying equation \eqref{Poh} by $1/(4s+\alpha-N)$, equation \eqref{NES} by $(2s+\alpha)/[2(4s+\alpha-N)]$, and subtracting, we obtain
\begin{equation*}\label{I=J}
\frac{1}{2}\|(-\Delta)^{\frac{s}{2}}u\|_2^2+\frac{C_\alpha}{4}\iint_{\mathbb{R}^N\times\mathbb{R}^N}\frac{u^2(x)u^2(y)}{|x-y|^{N-\alpha}}dxdy
=
\frac{\lambda }{2^*_{s,\alpha}}\int_{\mathbb{R}^N} |u|^{2^*_{s,\alpha}}dx.
\end{equation*}
which yields the conclusion.
\end{proof}

Now, let us to consider the potential operators $\tilde{f},\tilde{g}:E\to E^*$ associated with the local nonlinearities $f$ and $g$, defined as
$$
\tilde{f}(u)v=\int_{\mathbb{R}^N}f(|x|,u)vdx\quad\mbox{and}\quad\tilde{g}(u)v=\int_{\mathbb{R}^N}g(|x|,u)vdx,
$$
and whose potentials $\tilde{F},\tilde{G}:E\to \mathbb{R}$ such that $\tilde{F}(0)=0$ and $\tilde{G}(0)=0$ are
$$
\tilde{F}(u)=\int_{\mathbb{R}^N}F(|x|,u)dx\quad\mbox{and}\quad\tilde{G}(u)=\int_{\mathbb{R}^N}G(|x|,u)dx.
$$

\begin{lemma}\label{o(t)}The following asymptotics hold. \\
\begin{enumerate}[label={(${\roman*}$)},  
                  ref=${\roman*}$]               
\item\label{o(t)(i)} \textit{($f$ subscaled)}. Assume that $f$ satisfies \eqref{fgrowth} with 
    \begin{equation}\label{fsubscaled}
        p_{\rm rad} <q_1<q_2<2^*_{s,\alpha}\quad\mbox{and}\quad p_{\rm rad}<r\leq 2^*_{s}
    \end{equation}
    Then, $\tilde f(u_t)v_t=o(t^{4s+\alpha-N})\|v\|$ as $t\to\infty$, uniformly in $u$ on bounded sets for all $v\in E$.
\item\label{o(t)(ii)}  Suppose that $f$ satisfies \eqref{f=B+g}. If $g$ satisfies \eqref{ggrowth}-\eqref{ggrowthc}, then 
       $$
       \tilde{g}(u_t)v_t=o(t^{4s+\alpha-N})\|v\| \quad\mbox{as}\quad t\to0,
       $$
       uniformly in $u$ on bounded sets, and for all $v\in E$.
 \item\label{o(t)(iii)} \textit{($f$ asymptotically scaled)}. Suppose that $f$ satisfies \eqref{f=B+g}. If $g$ satisfies \eqref{ggrowth2}-\eqref{ggrowth2c}, then 
 $$
 \tilde{g}(u_t)v_t=o(t^{4s+\alpha-N})\|v\|\quad\mbox{as}\quad t\to\infty,
 $$
 uniformly in $u$ on bounded sets, and for all $v\in E$.
\end{enumerate}
\end{lemma}
\begin{proof}
We break the proof analysing each case separately. \\
\eqref{o(t)(i)} We have $u_t(x)=t^\theta u(tx)$, where $\theta=(2s+\alpha)/2$, as in \eqref{def theta}. 
Since $f$ satisfies \eqref{fgrowth}, for any $t>0$ we have
\begin{eqnarray*}
 |\tilde f(u_t)v_t|&\leq& 
 a_1\int_{\mathbb{R}^N}|u_t|^{q_1-1}|v_t|dx+a_2\int_{\mathbb{R}^N}|u_t|^{q_2-1}|v_t|dx
+\int_{\mathbb{R}^N}|a(x)||v_t|dx\\
&\leq& a_1t^{\theta q_1-N}\int_{\mathbb{R}^N}|u|^{q_1-1}|v|dx+a_2t^{\theta q_2-N}\int_{\mathbb{R}^N}|u|^{q_2-1}|v|dx
+t^{\theta-N}\int_{\mathbb{R}^N}|a(x/t)||v|dx\\
&\leq& a_1t^{\theta q_1-N}\|u\|_{L^{q_1}}^{q_1-1}\|v\|_{L^{q_1}}+a_2t^{\theta q_2-N}\|u\|_{L^{q_2}}^{q_2-1}\|v\|_{L^{q_2}}
+t^{\theta-N/r}\|a\|_{L^{r'}}\|v\|_{L^{r}}.
\end{eqnarray*}
Recalling the embedding $E\hookrightarrow L^{p}(\mathbb{R}^N)$ for $p=q_1,q_2$ and $r$, we obtain
$$
 |\tilde f(u_t)v_t|\leq
 c_1 t^{\theta q_1-N}\|u\|^{q_1-1}\|v\|+c_2 t^{\theta q_2-N}\|u\|^{q_2-1}\|v\|
+c_3t^{\theta-N/r}\|a\|_{L^{r'}}\|v\|.
$$
Since $q_1<q_2<2^*_{s,\alpha}$ and $\theta 2^*_{s,\alpha}=4s+\alpha$ we have $\theta q_1-N<\theta q_2-N<(4s+\alpha-N)$. Also,  $r\leq 2^*_s$ gives us $\theta-N/r\leq(4s+\alpha-N)/2<(4s+\alpha-N)$. Thus, \eqref{o(t)(i)} holds.\\

\noindent \eqref{o(t)(ii)}  By \eqref{ggrowth}-\eqref{ggrowthc}, for any $t>0$ we have
\begin{eqnarray*}
 |\tilde g(u_t)v_t|&\leq& 
 a_3\int_{\mathbb{R}^N}|u_t|^{q_3-1}|v_t|dx+a_4\int_{\mathbb{R}^N}|u_t|^{q_4-1}|v_t|dx
\\
&\leq& a_3t^{\theta q_3-N}\int_{\mathbb{R}^N}|u|^{q_3-1}|v|dx+a_4t^{\theta q_4-N}\int_{\mathbb{R}^N}|u|^{q_4-1}|v|dx
\\
&\leq& a_3t^{\theta q_3-N}\|u\|_{L^{q_3}}^{q_3-1}\|v\|_{L^{q_3}}+a_4t^{\theta q_4-N}\|u\|_{L^{q_4}}^{q_4-1}\|v\|_{L^{q_4}}
\\
&\leq& \tilde c_3t^{\theta q_3-N}\|u\|^{q_3-1}\|v\|+\tilde c_4t^{\theta q_4-N}\|u\|^{q_4-1}\|v\|.
\end{eqnarray*}
Since $q_4>q_3>2^*_{s,\alpha}$, we have $\theta q_3-N>4s+\alpha-N$ and so \eqref{o(t)(ii)}  holds. \\

\noindent \eqref{o(t)(iii)} Since the argument in this case is similar to \eqref{o(t)(i)}, we leave out the details; and this concludes the proof.
\end{proof}

The following lemma is an immediate consequence of the compactness of the embedding $E\hookrightarrow L^q(\mathbb{R}^N)$, for $p_{\rm rad} <q<2^*_{s}$. 

\begin{lemma}\label{f~compact}
    Suppose that $f$ satisfies \eqref{fgrowth} with $p_{\rm rad}<q_1,q_2< 2^*_s$. Then $\tilde{f}$ is a compact potential operator.
\end{lemma}

In the spirit of \cite[Theorem 2.16]{Mercuri-Perera} we have the following existence result by minimisation. 
\begin{proposition}\label{minimization}
Suppose that $f$ satisfies \eqref{fgrowth} with exponents satisfying \eqref{fsubscaled}. Then,\eqref{P} has a solution.
\end{proposition}
\begin{proof}
    By the preceding lemma, the assumption on $f$ implies that $\tilde{f}$ is a compact operator, and by Lemma \ref{o(t)}\eqref{o(t)(i)} that
    $$
    \tilde f(u_t)v_t=o(t^{4s+\alpha-N})\|v\| \quad\mbox{as}\quad t\to\infty,
    $$ 
    uniformly in $u$ on bounded sets, for all $v\in E$. This implies that the functional $\Phi$ is coercive and, since it is bounded on bounded sets, we conclude that $\Phi$ is bounded from below. Moreover, $\Phi$ satisfies the (PS) condition in $E$ (see also \cite[Proposition 2.3]{Mercuri-Perera}): 
    in fact, since any (PS) sequence $(u_n)$ must be bounded in $E$ and $E$ is a reflexive space, it follows that $u_n\rightharpoonup u$ in $E$, up to a subsequence, and then $\tilde{f}(u_n)\to \tilde{f}(u)$ in $E^*$. So 
    $$
    \mathcal{A}(u_n)(u_n-u)=(\Phi'(u_n)(u_n-u)+\tilde{f}(u_n)(u_n-u))\to0,
    $$
    and by Lemma \ref{verify H7} we conclude that $u_n\to u$ in $E$, up to a subsequence. Therefore, $\Phi$ has a minimizer in $E$, hence a solution to \eqref{P}.
\end{proof}


\subsection{Scaled eigenvalue problem}\label{subsection SEP}

Let us concentrate on the scaled eigenvalue problem $\mathcal{A}=\lambda \mathcal{B}$, i.e. problem \eqref{NE}. 
The associated functional for this equation is $\Phi_\lambda: E\to\mathbb{R}$, given by
\begin{equation*}\label{Philambda}
\Phi_\lambda(u)=\frac{1}{2}\int_{\mathbb{R}^N}|(-\Delta)^\frac{s}{2}u|^2dx+\frac{C_\alpha}{4}\int_{\mathbb{R}^N}\int_{\mathbb{R}^N}\frac{u^2(x)u^2(y)}{|x-y|^{N-\alpha}}dxdy - \frac{\lambda}{2^*_{s,\alpha}}\int_{\mathbb{R}^N}|u|^{2^*_{s,\alpha}}dx.
\end{equation*}
When \eqref{NE} has a nontrivial solution $u\in E$ we say that $\lambda$ is an eigenvalue and $u$ is an eigenfuction associated with $\lambda$.
We have seen that conditions \eqref{H6}-\eqref{H12} are satisfied by  the scaled operators $\mathcal{A}$ and $\mathcal{B}$ and their potentials $I$ and $J$, respectively.  We will call the set $\sigma(\mathcal{A},\mathcal{B})$ of all eigenvalues of \eqref{NE} the spectrum of the pair of
scaled operators $(\mathcal{A},\mathcal{B})$. By \eqref{H6} and \eqref{H8}we know that $\sigma(\mathcal{A},\mathcal{B})\subset (0,\infty)$.

We set
$$
\mathcal{M}=\{u\in E: I(u)=1\},
$$
and define the projection
$
\pi:E\backslash\{0\}\to\mathcal{M}
$
as
\begin{equation}\label{def pi}
    \pi(u)=u_{t_u}
\end{equation}
where $t_u=(I(u))^{-1/(4s+\alpha-N)}$. 
We also define
$$
\Psi(u)=\frac{1}{J(u)},\quad u\in E\backslash\{0\},\quad\mbox{and}\quad \Tilde{\Psi}=\Psi_{|_\mathcal{M}}.
$$
We note that $\mathcal{M}$ is a complete, symmetric, and bounded $C^1$-Finsler manifold,
and eigenvalues of problem \eqref{NE} coincide with critical values of $\Tilde{\Psi}$, as proved in \cite[Proposition 2.5]{Mercuri-Perera}. 
Let $\mathcal{F}$ denote the class of symmetric subsets of $\mathcal{M}$, and let $i(Y)$ be the cohomological index of $Y \in\mathcal F$, of Fadell and Rabinowitz (see \cite{Fadell-Rabinowitz-1978}). For $k\geq1$, let
$$
\mathcal{F}_k = \{Y \in \mathcal{F}: i(Y) \geq k\}
$$
and set
$$
\lambda_k := \inf_{Y\in\mathcal{F}_k}\sup_{u\in Y}\Psi(u).
$$
The next theorem is an immediate consequence of an abstract result proved in \cite{Mercuri-Perera}. It ensures that $(\lambda_k)$ is indeed a sequence of \emph{nonlinear eigenvalues} for \eqref{NE} and have some additional information, which are stated below.
\begin{theorem}[\cite{Mercuri-Perera}, Theorem 1.3]\label{lambdak}
Assume $s\in(0,1)$, $\alpha\in(1,N)$ and $4s+\alpha> N$. Then $\lambda_k\nearrow \infty$ is a sequence of eigenvalues of \eqref{NE}. Moreover,
\begin{enumerate}[label={(\roman*)},  
                  ref=(\roman*)]               
\item\label{lambdak(i)} The first eigenvalue is given by
$$
\lambda_1=\min_{u\in\mathcal{M}}\tilde{\Psi}(u)>0.
$$
\item\label{lambdak(ii)}  If $\lambda_k = \cdots = \lambda_{k+m-1} = \lambda$, then $i(E_\lambda)\geq m$, where $E_\lambda$ is the set of eigenfunctions associated with $\lambda$ that lie on $\mathcal{M}$.
\item\label{lambdak(iii)} If $\lambda_k<\lambda<\lambda_{k+1}$, then
$$i(\tilde{\Psi}^{\lambda_k} ) = i(\mathcal{M}\backslash\tilde{\Psi}_{\lambda}) = i(\tilde{\Psi}^{\lambda} ) 
= i(\mathcal{M}\backslash \tilde{\Psi}_{\lambda_{k+1}}) = k,$$
where 
$\tilde{\Psi}^a=\{ u\in \mathcal{M} : \tilde{\Psi}(u) \leq a\}$ and 
$\tilde{\Psi}_a = \{ u\in \mathcal{M} : \tilde{\Psi}(u) \geq a\}$ for $a\in\mathbb{R}$.
\end{enumerate}
\end{theorem}


\section{Solutions to subcritical problems}\label{sec subcritical}

Recall that if $u_0$ is an isolated critical point for $\Phi$, with $\Phi(u_0)=c$, the critical groups are
defined by
$$
C^l(\Phi,u_0) = H^l(\Phi^c\cap U,\Phi^c\cap (U\backslash\{u_0\})),\quad l\geq0,
$$
where  $\Phi^c = \{u\in E : \Phi(u) \leq c\}$, $U\subset E$ is an open set such that $u_0$ is the unique critical point of $\Phi$ in  $U$ and $H^*$ denotes cohomology with $\mathbb{Z}_2$ coeﬃcients. 

As it is customary, in order to find multiple solutions to our PDE we use critical groups to distinguish them, and in particular from the trivial one. We focus first on the proofs in the case of subcritical nonlinearities $f$.

\begin{proof-th} \textbf{\textit{of Theorem \ref{THsubscaled}.}} 
\phantomsection \addcontentsline{toc}{subsection}{Proof of Theorem~\ref{THsubscaled}}
    Assumptions \eqref{f=B+g}-\eqref{ggrowthc} allows us to use Lemma \ref{f~compact}, by which $\tilde{f}$ and $\tilde{g}$ are compact.  By \eqref{fsubscaled} and Lemma \ref{o(t)}\eqref{o(t)(i)}-\eqref{o(t)(ii)}, it holds that
    $$
    \tilde{f}(u_t)v_t=o(t^{4s+\alpha-N})\|v\| \quad\mbox{as}\quad t\to\infty\quad\mbox{and}\quad \tilde{g}(u_t)v_t=o(t^{4s+\alpha-N})\|v\| \quad\mbox{as}\quad t\to0,
    $$
    uniformly in $u$ on bounded sets for all $v\in E$. The conclusion follows then by \cite[Theorem 2.17]{Mercuri-Perera}.
\end{proof-th}

To see an application of this result, let   $\lambda\in \mathbb{R}\backslash\sigma(\mathcal{A},\mathcal{B}),$  $0<\gamma<\min\{2^*_{s}-2^*_{s,\alpha},2^*_{s,\alpha}(2^*_{s,\alpha}-p_{\rm rad})\}$ and $f:\mathbb{R}\to\mathbb{R}$ defined as
$$
f(t)=\lambda \frac{|t|^{2^*_{s,\alpha}-2}t}{1+|t|^\gamma}.
$$
Since $|t|^\frac{\gamma}{2^*_{s,\alpha}}\leq\frac{1}{(2^*_{s,\alpha})'}+\frac{1}{2^*_{s,\alpha}}|t|^{\gamma}\leq 1+|t|^{\gamma}$  we see that
$
|f(t)|\leq |\lambda||t|^{q_1-1},
$
where $q_1=2^*_{s,\alpha}-\gamma/2^*_{s,\alpha}\in(p_{\rm{rad}},2^*_{s,\alpha})$. Also, 
$$
|g(t)|=|f(t)-\lambda |t|^{2^*_{s,\alpha}-2}t|=|\lambda|.|t|^{2^*_{s,\alpha}-1}\frac{|t|^\gamma}{1+|t|^\gamma}\leq |\lambda|.|t|^{q_3-1},\,\,\forall t\in\mathbb{R},
$$
where $q_3=2^*_{s,\alpha}+\gamma\in(2^*_{s,\alpha},2^*_{s})$. Thus, if $\lambda$ is not an eigenvalue of  \eqref{NE}, Theorem \ref{THsubscaled} provides the existence of at least one or, respectively, two solutions to \eqref{P}, when $\lambda>\lambda_1$ or $\lambda>\lambda_2.$\\

The next proof deals with the asymptotically scaled case.

\begin{proof-th} \textbf{\textit{of Theorem \ref{th subcri AS}.}}\addcontentsline{toc}{subsection}{Proof of Theorem~\ref{th subcri AS}}
Note that by Lemma \ref{f~compact}, $\tilde{g}$ is a compact operator, and Lemma \ref{o(t)}\eqref{o(t)(iii)} yields
    $$
    \tilde{g}(u_t)v_t=o(t^{4s+\alpha-N})\|v\| \quad\mbox{as}\quad t\to\infty,
    $$
    uniformly in $u$ on bounded sets for all $v\in E$. The conclusion then follows by  \cite[Theorem 2.25]{Mercuri-Perera}. 
\end{proof-th}


As for the superscaled case, we recall that our approach requires that $g$ satisfies in addition the Ambrosetti-Rabinowitz condition \eqref{Gsuper}.
We need the following compactness result, which says that the functional $\Phi$ satisfies the Palais-Smale condition.

\begin{lemma}\label{PS-superscaled}
 Consider $N/4<s<1$ and assume that \eqref{f=B+g}-\eqref{ggrowthc} and \eqref{Gsuper} hold. Then $\Phi$ satisfies the $(PS)$ condition. 
\end{lemma}
\begin{proof}
Let $c\in\mathbb{R}$ and $(u_n)$ in $E$ a $(PS)_c$ sequence for $\Phi,$ namely 
$$
\Phi(u_n)=c+o(1)\quad\mbox{and}\quad \|\Phi'(u_n)\|_{E^*}=o(1),\quad\mbox{as}\,\,n\to\infty.
$$
To show that it contains a strongly convergent subsequence, we first show that it is bounded in $E$. In fact, suppose, arguing by contradiction, that $\|u_n\|\to\infty$ for a subsequence. We use the notation \eqref{def pi} and set
$$
t_n=t_{u_n}=\frac{1}{I(u_n)^{1/(4s+\alpha-N)}},\quad\tilde{u}_n=\pi(u_n)=(u_n)_{t_n}\quad\mbox{and}\quad \tilde{t}_n=\frac{1}{t_n}=I(u_n)^{1/(4s+\alpha-N)}.
$$
Since  $I$ is coercive by \eqref{H10}, we have $\tilde{t}_n\to\infty$.
Since $ \mathcal{M}$ is bounded,  $(\tilde{u}_n)$ is, too. Since $(\tilde{u}_n)_{\tilde{t}_n}=u_n$ by \eqref{H1}, we may estimate
\begin{equation}\label{unOt}
    \|u_n\|=\|(\tilde{u}_n)_{\tilde{t}_n}\|\leq\max\{(\tilde{t}_n)^{(4s+\alpha-N)/2},(\tilde{t}_n)^{(4s+\alpha-N)/4}\}I(\tilde{u}_n)=O((\tilde{t}_n)^{(4s+\alpha-N)/2})
\end{equation}
as $n\to\infty$. Observe that, since $\Phi(u_n)=c+o(1)$, we may write $\Phi_\lambda(u_n)=\tilde{G}(u_n)+c+o(1)$. Therefore, recalling that
\begin{eqnarray*}
 \Phi_\lambda(u_t)&=&\frac{1}{2}\int_{\mathbb{R}^N}|(-\Delta)^\frac{s}{2}(u_t)|^2dx+\frac{C_\alpha}{4}\iint_{\mathbb{R}^N\times\mathbb{R}^N}\frac{(u_t)^2(x)(u_t)^2(y)}{|x-y|^{N-\alpha}}dxdy -\frac{\lambda}{2^*_{s,\alpha}} \int_{\mathbb{R}^N}|u_t|^{2^*_{s,\alpha}}dx\\
 &=&I(u_t)-\lambda J(u_t)= t^{4s+\alpha-N}\Phi_\lambda(u)
\end{eqnarray*}
for all $u\in E$ and $t\geq0$, we also have, using \eqref{Gsuper}, that
\begin{eqnarray}
(\tilde{t}_n)^{4s+\alpha-N}\Phi_\lambda(\tilde u_n)&=&\int_{\mathbb{R}^N}G(|x|,u_n)dx+c+o(1)\label{eqPhiLn}\\
&\geq &c_0\int_{\mathbb{R}^N}|u_n
|^qdx+c+o(1)= c_0\tilde{t}_n^{\theta q-N}\int_{\mathbb{R}^N}|\tilde{u}_n|^qdx+c+o(1).\nonumber
\end{eqnarray}
As $q>2^*_{s,\alpha}$, we note that $\theta q-N>4s+\alpha-N>0$ and hence, by the boundedness of $(\Phi_\lambda(\tilde{u}_n))$ we get $\|\tilde{u}_n\|_{L^q}\to0$. Picking some $p\in (p_{\rm rad},2^*_{s,\alpha})$, by the embedding $E\hookrightarrow L^p(\mathbb{R}^N)$,  $\|(\tilde{u}_n)\|_{L^p}$ is bounded in $n$ and therefore, by interpolation, $\|\tilde{u}_n\|_{L^{2^*_{s,\alpha}}}\to0$.
On the other hand, by
$$
\Phi_\lambda'(u_t)u_t=\mathcal{A}(u_t)u_t-\lambda \mathcal{B}(u_t)u_t=t^{4s+\alpha-N}\Phi_\lambda'(u)u,\quad \forall u\in E, \forall t\geq0,
$$ 
$\Phi'(u_n)u_n=o(\|u_n\|)$ and \eqref{unOt}, we get
\begin{eqnarray}\label{eqPhi'Ln}
(\tilde{t}_n)^{4s+\alpha-N}\Phi'_\lambda(\tilde u_n)\tilde{u}_n=\int_{\mathbb{R}^N}g(|x|,u_n)u_ndx+o(\|u_n\|)
=\int_{\mathbb{R}^N}g(|x|,u_n)u_ndx+o((\tilde{t}_n)^{(4s+\alpha-N)/2}).
\end{eqnarray}
Finally, multiplying \eqref{eqPhiLn} by $q$ and subtracting \eqref{eqPhi'Ln} we conclude that
\begin{eqnarray*}
  & &\left(\frac{q}{2}-1\right)\int_{\mathbb{R}^N}|(-\Delta)^\frac{s}{2}\tilde u_n|^2dx+C_\alpha\left(\frac{q}{4}-1\right) \iint_{\mathbb{R}^N\times\mathbb{R}^N}\frac{\tilde u_n^2(x)\tilde u_n^2(y)}{|x-y|^{N-\alpha}}dxdy \\
  &&\qquad\qquad=\lambda\left(\frac{q}{2^*_{s,\alpha}} -1\right)\int_{\mathbb{R}^N}|\tilde u_n|^{2^*_{s,\alpha}}dx 
  + (\tilde{t}_n)^{-(4s+\alpha-N)}\int_{\mathbb{R}^N}\left[qG(|x|,u_n)-g(|x|,u_n)u_n\right]dx+o(1)\\
  && \qquad\qquad\leq\,\, o(1),
\end{eqnarray*}
which yields $\|\tilde u_n\|=o(1)$, as $q>4>2$. Since $\tilde u_n\in\mathcal{M},$ this is a contradiction, and therefore the (PS) sequence $(u_n)$ is bounded in $E$. Since by Lemma \ref{f~compact} $\tilde{f}$ is a compact operator, we may use \cite[Proposition 2.3]{Mercuri-Perera} by which $(u_n)$ has a convergent subsequence, and this concludes the proof.
\end{proof}

We now deal with the proof in the superscaled case, by applying \cite[Theorem 2.27]{Mercuri-Perera}, which follows from a new linking theorem; see \cite[Theorem 2.24]{Mercuri-Perera}.

\begin{proof-th} \textbf{\textit{of Theorem \ref{Th-superscaled-not Eigenvalue}.}}\addcontentsline{toc}{subsection}{Proof of Theorem~\ref{Th-superscaled-not Eigenvalue}}
We claim that 
\begin{equation}\label{F~infty}
    \lim_{t\to\infty}\dfrac{\tilde{F}(u_t)}{t^{4s+\alpha-N}}=\infty,
\end{equation}
uniformly in $u$ on compact subsets of $\mathcal{M},$
 where $\tilde{F}:E\to\mathbb{R}$ is the potential of $\tilde{f},$ defined as
$$
\tilde{F}(u)=\int_{\mathbb{R}^N}F(|x|,u)dx=\frac{\lambda}{2^*_{s,\alpha}}\int_{\mathbb{R}^N}|u|^{2^*_{s,\alpha}}dx+\int_{\mathbb{R}^N}G(|x|,u)dx.
$$
In fact, \eqref{Gsuper} yields
$$
\tilde{F}(u_t)\geq\frac{\lambda}{2^*_{s,\alpha}}\int_{\mathbb{R}^N}|u_t|^{2^*_{s,\alpha}}dx+c_0\int_{\mathbb{R}^N}|u_t|^qdx
=\frac{\lambda t^{4s+\alpha-N}}{2^*_{s,\alpha}}\int_{\mathbb{R}^N}|u|^{2^*_{s,\alpha}}dx+c_0t^{q\theta-N}\int_{\mathbb{R}^N}|u|^qdx,
$$
and the claim now follows as $q>2^*_{s,\alpha},$ namely $q\theta-N>4s+\alpha-N>0$. Since by Lemma \ref{PS-superscaled}, $\Phi$ satisfies the (PS) condition, we are now in the position to apply \cite[Theorem 2.27]{Mercuri-Perera}, and this concludes the proof.
\end{proof-th}

In the next proofs, the additional challenge consists in removing the requirement that $\lambda\notin \sigma(\mathcal{A},\mathcal{B}).$ We deal with this case, assuming  $\tilde{G}$ to be negative in the subscaled case, as this allows us to apply a recent abstract result based on a notion of local linking; see \cite[Theorem 2.30]{Mercuri-Perera}.

We say that $\Phi$ has a scaled local linking near the origin in dimension $k\geq1$ if
there exist two nonempty, closed, symmetric and disjoint subsets $A_0,B_0\subset\mathcal{M}$ such that
$$
i(A_0)=i(\mathcal{M}\backslash B_0)=k
$$
and $\rho>0$ satisfying
$$
\begin{cases}
    \Phi(u_t)\leq0,&\forall u\in A_0\,\,\mbox{and}\,\, 0\leq t\leq\rho,\\
      \Phi(u_t)>0,&\forall u\in B_0\,\,\mbox{and}\,\, 0< t\leq\rho.
\end{cases}
$$
\begin{lemma}\label{local-linking}
  Assume \eqref{f=B+g}-\eqref{ggrowthc}. Then $\Phi$ has a scaled local linking near the origin in dimension
$k$ in either of the following cases:
\begin{enumerate}
   \item[(i)]  $\lambda_k<\lambda\leq\lambda_{k+1}$ and $G(|x|,t)<0$ a.e. $x\in\mathbb{R}^N$ and all $t\in\mathbb{R}\backslash\{0\}$. 
   \item[(ii)] $\lambda_k\leq\lambda <\lambda_{k+1}$ and $G(|x|,t)\geq0$ a.e. $x\in\mathbb{R}^N$ and all $t\in\mathbb{R}$.
\end{enumerate}
\end{lemma}
\begin{proof}
Assumptions \eqref{f=B+g}-\eqref{ggrowthc} allows us to use Lemma \ref{o(t)}\eqref{o(t)(ii)}, which states that $\tilde{g}(u_t)v_t=o(t^{4s+\alpha-N})\|v\|$ as $t\to0^+$ uniformly in $u$ on bounded sets, for all $v\in E$. Hence, by \eqref{potential integral}, we obtain
$$
\tilde{G}(u_t)=o(t^{4s+\alpha-N}),\quad\mbox{as}\,\,t\to0^+,
$$
uniformly in $u$ on bounded sets. Since $\tilde{\Psi}(u)=1/J(u)$ for $u\in\mathcal{M}$,  we may write 
\begin{eqnarray}
    \Phi(u_t)&=&I(u_t)-\lambda J(u_t)-\tilde{G}(u_t) = t^{4s+\alpha-N}\left(I(u)-\lambda J(u)\right)-\tilde{G}(u_t)\nonumber\\
    &=&t^{4s+\alpha-N}\left(1-\frac{\lambda}{\tilde{\Psi}(u)}\right)-\tilde{G}(u_t) \label{Gsign} \\
    &=&t^{4s+\alpha-N}\left(1-\frac{\lambda}{\tilde{\Psi}(u)}+o(1)\right) \label{G=o(1)}
\end{eqnarray}
as $t\to0^+$, for all $u\in\mathcal{M}$.  We observe now that if $\lambda_k<\lambda_{k+1}$, by Theorem \ref{lambdak}\ref{lambdak(iii)} we have
  $$
i(\tilde{\Psi}^{\lambda_k})=i(\mathcal{M}\backslash\tilde{\Psi}_{\lambda_{k+1}})=k.
  $$
Consider $A_0=\tilde{\Psi}^{\lambda_k}$ and $B_0=\tilde{\Psi}_{\lambda_{k+1}}$, and note that these subsets of $\mathcal{M}$ are nonempty and disjoint, closed as $\tilde{\Psi}$ is continuous. Moreover, these are symmetric as $\tilde{\Psi}$ is even.

\noindent (i) If $\lambda_k<\lambda\leq\lambda_{k+1}$, by \eqref{G=o(1)} we see that there exists some $\rho>0$ such that
$$
 \Phi(u_t)\leq t^{4s+\alpha-N}\left(1-\frac{\lambda}{\lambda_k}+o(1)\right) \leq0,\quad \forall (u,t)\in A_0\times[0,\rho].
$$
 On the other hand, since $\tilde{G}(u_t)<0$ for all $t>0$, by \eqref{Gsign} we get
$$
\Phi(u_t)\geq t^{4s+\alpha-N}\left(1-\frac{\lambda}{\lambda_{k+1}}\right) -\tilde{G}(u_t)> 0, \quad \forall (u,t)\in B_0\times(0,\rho].
$$
Thus, we may conclude that $\Phi$ has a scaled local linking near the origin in dimension $k.$ 

\noindent (ii) This case is similar to (i), we therefore leave out the details. We conclude the proof simply observing that by \eqref{Gsign} and the fact of $\tilde{G}\geq0,$ we get  $\Phi(u_t)\leq 0$ for all $u\in A_0$ and $t\geq0$. Whereas, by \eqref{G=o(1)} we get $\rho>0$ such that $\Phi(u_t)>0$ for all $u\in B_0$ and $t\in(0,\rho]$.
\end{proof}

We are now ready to prove our existence result in the subscaled case and allowing $\lambda$ to be an eigenvalue,  under the assumption that $G<0.$

\begin{proof-th} \textbf{\textit{of Theorem \ref{THsubscaledG<0}.}}\addcontentsline{toc}{subsection}{Proof of Theorem~\ref{THsubscaledG<0}}
 Since $f$ satisfies \eqref{fgrowth} and \eqref{fsubscaled}, by Proposition \ref{minimization} we know that $\Phi$ has a minimizer $u_0$ in $E$. If this critical point is not isolated, we would have infinite solutions for \eqref{P}. So, we may assume that $u_0$ is an isolated critical point for $\Phi$, and so the critical groups are given by
 $$
 C^l(\Phi,u_0)=\delta_{l0}\mathbb{Z}_2,\,\, \forall l\geq0.
 $$
By \eqref{f=B+g}-\eqref{ggrowthc}, we have $\tilde{g}(0)=0$ and so $u\equiv0$ is a critical point for $\Phi$, which we may assume is also isolated.  We distinguish two cases. 
 
\noindent (i) In the case $\lambda>\lambda_1$, let us suppose that $\lambda_k<\lambda\leq \lambda_{k+1}$ for some $k\geq1$. Then, Lemma \ref{local-linking}(i) ensures that $\Phi$ has a local linking near origin in dimension $k$. Based on this, by \cite[Theorem 2.29]{Mercuri-Perera} we know that the critical group $C^k(\Phi,0)$ is nontrivial. Since for the minimizer $u_0$ it holds that $C^k(\Phi,u_0)=0$, we may conclude that $u_0\neq0$.

\noindent (ii) If $\lambda>\lambda_2$, we see that $C^k(\Phi,0)\neq 0$ for some $k\geq2$. Then, \cite[Proposition 2.18]{Mercuri-Perera} ensures the existence of a critical point $u_1\neq0$ for $\Phi$ such that either
$\Phi(u_1)<0$ and $C^{k-1}(\Phi,u_1)\neq 0,$ or $\Phi(u_1)>0$ and $C^{k+1}(\Phi,u_1)\neq 0$. We may therefore conclude that, for some $l_0\geq1$,
$$
C^{l_0}(\Phi,u_1)\neq 0\quad\mbox{and}\quad C^{l_0}(\Phi,u_0)= 0.
$$
Thus, $u_0\neq u_1$ are nontrivial solutions for \eqref{P}.
\end{proof-th}

We now deal with subcritical problems in a superscaled regime, allowing again $\lambda$ to be an eigenvalue to \eqref{NE}.

\begin{proof-th} \textbf{\textit{of Theorem \ref{Th-superscaled-general}.}}\addcontentsline{toc}{subsection}{Proof of Theorem~\ref{Th-superscaled-general}}
Since \eqref{Gsuper} holds, by Lemma \ref{PS-superscaled}  $\Phi$ satisfies the (PS) condition. We distinguish two cases.

\noindent (i)  $\lambda<\lambda_1.$ 
As in Lemma \ref{local-linking} we see that \eqref{G=o(1)} holds, namely
\begin{eqnarray*}
    \Phi(u_t)&=&t^{4s+\alpha-N}\left(1-\frac{\lambda}{\tilde{\Psi}(u)}+o(1)\right) \quad\mbox{as}\,\,t\to0^+,
\end{eqnarray*}
uniformly in $u\in\mathcal{M}$. Setting $\lambda^+=\max\{\lambda,0\}$, by Theorem \ref{lambdak}\ref{lambdak(i)} we obtain
\begin{eqnarray*}
    \Phi(u_t)&\geq&t^{4s+\alpha-N}\left(1-\frac{\lambda^+}{\lambda_1}+o(1)\right) \quad\mbox{as}\,\,t\to0^+,
\end{eqnarray*}
uniformly in $u\in\mathcal{M}$. On the other hand, by \eqref{F~infty} (proof of Theorem \ref{Th-superscaled-not Eigenvalue}) we have
$$
\lim_{t\to\infty}\frac{\tilde F(u_t)}{t^{4s+\alpha-N}}=\infty,
$$
for any $u\in\mathcal{M},$ hence
$$
\Phi(u_t)=t^{4s+\alpha-N}\left(1-\frac{\tilde F(u_t)}{t^{4s+\alpha-N}}\right)\to-\infty\quad\mbox{as}\,\,t\to\infty.
$$

Thus, $\Phi$ has the mountain pass geometry. Since it also satisfies the (PS) condition, by the classical mountain-pass theorem $\Phi$ has a nontrivial critical point corresponding to a minimax energy level $c>0$.\\

\noindent (ii) $\lambda\geq\lambda_1.$
We may assume $\lambda_k\leq\lambda<\lambda_{k+1}$ for some $k\geq1$ and $G(|x|,t)>0$ almost everywhere in $\mathbb{R}^N$, for all $t\neq0$. Note that Lemma \ref{local-linking} ensures that $\Phi$ has a local linking near the origin in dimension $k$. In order to apply \cite[Corollary 2.31]{Mercuri-Perera} and find a nontrivial critical point for $\Phi,$ we ovbserve that $\Phi^a$ is contractible for some $a<0$. In fact, for any
$$
a<\inf_{u\in\mathcal{M}, 0\leq t\leq1}\Phi(u_t),
$$
we have
$$
\Phi^a=\{u\in E:\Phi(u)\leq a\}=\{u_t:u\in\mathcal{M},\,t>1\,\,\mbox{and}\,\,\varphi_u(t)\leq a\}
$$
where $\varphi_u(t)=\Phi(u_t)$. By \eqref{Gsuper}, for any $u\in E$ and $t>0$ we get
$$
\tilde{G}(u_t)=t^{-N}\int_{\mathbb{R}^N}G(|x|/t,t^\theta u(x))dx\geq c_0 t^{\theta q-N}\int_{\mathbb{R}^N}|u(x)|^qdx.
$$
Note that by $q>2^*_{s,\alpha}$ and $4s+\alpha>N$, we have $\theta q-N>4s+\alpha-N>0$.
Hence, since by \eqref{Gsign} we can estimate
\begin{eqnarray*}
\varphi_u(t)&=&t^{4s+\alpha-N}\left(1-\frac{\lambda}{\tilde{\Psi}(u)}\right)-t^{-N}\int_{\mathbb{R}^N}G(|x|/t,t^\theta u(x))dx\\
&\leq& t^{4s+\alpha-N}\left(1-\frac{\lambda}{\tilde{\Psi}(u)}\right)-c_0 t^{\theta q-N}\int_{\mathbb{R}^N}|u(x)|^qdx,
\end{eqnarray*}
it follows that $\varphi_u(t)\to-\infty$ as $t\to\infty$. We now observe that by \eqref{Gtauu},  $\varphi_u$ is a $C^1$ function and 
\begin{eqnarray*}
\varphi'_u(t)&=&(4s+\alpha-N)t^{4s+\alpha-N-1}\left(1-\frac{\lambda}{\tilde{\Psi}(u)}\right)+Nt^{-N-1}\int_{\mathbb{R}^N}G(|x|/t,t^\theta u(x))dx\\
&&+t^{-N-2}\int_{\mathbb{R}^N}\frac{\partial G}{\partial \tau}(|x|/t,t^\theta u(x))|x|dx-t^{-N}\int_{\mathbb{R}^N}\theta  g(|x|/t,t^\theta u(x))t^{\theta-1}u(x)dx\\
&\leq& \frac{(4s+\alpha-N)}{t}\varphi_u(t)+(4s+\alpha)t^{-N-1}\int_{\mathbb{R}^N}G(|x|/t,t^\theta u(x))dx\\
&&-t^{-N-1}\int_{\mathbb{R}^N}\theta  g(|x|/t,t^\theta u(x))t^{\theta}u(x)dx.
\end{eqnarray*}
Recalling that $4s+\alpha=\theta 2^*_{s,\alpha}<\theta q$, by \eqref{Gsuper} we get
$$
\varphi'_u(t)\leq \frac{(4s+\alpha-N)}{t}\varphi_u(t), \quad\mbox{for all}\,\,t>0.
$$
As a consequence, $\varphi_u(t)<0$ yields $\varphi'_u(t)<0$. The implicit function theorem now provides us with a $C^1-$map $\psi:\mathcal{M}\to(1,\infty)$ such that
$$
\varphi_u(t)>a\,\,\mbox{for}\,\, 0\leq t<\psi(u),\quad\varphi_u(\psi(u))=a\quad\mbox{and}\,\,\varphi_u(t)<a\,\,\forall t>\psi(u),
$$
hence
$$
\Phi^a=\{u_t:u\in\mathcal{M},\,t\geq \psi(u)\}.
$$
Finally, since $T:(E\backslash\{0\})\times [0,1] \to E\backslash\{0\}$ defined as
$$
T(u,\sigma)=
\begin{cases}
\tilde{u}_{\sigma\psi(\tilde{u})+(1-\sigma)t^{-1}_u}& \mbox{if}\,\,u\notin\Phi^a\\
u& \mbox{if}\,\,u\in\Phi^a,
\end{cases}
$$
is a deformation retraction onto $\Phi^a$ and $E\backslash\{0\}$ is a contractible set, we conclude that $\Phi^a$ is also contractible, and this concludes the proof.
\end{proof-th}


\section{Solutions to critical problems}\label{sec critical problems}

In this section we study \eqref{P}, focusing on the case $N/4<s<1$ and with a local nonlinearity $f:\mathbb{R}\to\mathbb{R}$ of critical growth, defined as
\begin{equation*}
  f(t) = \lambda |t|^{2^*_{s,\alpha}-2}t+ \mu|t|^{q_6-2}t+|t|^{2^*_s-2}t,\quad\mbox{for all} \,\, t\in\mathbb{R},
\end{equation*}
for $\lambda\in\mathbb{R}$, $\mu\geq0$
and $ q_6\in (2^*_{s,\alpha},2^*_{s})$. Namely, our PDE takes the form
\begin{equation}\label{Pcri}
(-\Delta)^su+\left(I_\alpha*u^2\right)u= \lambda |u|^{2^*_{s,\alpha}-2}u+ \mu|u|^{q_6-2}u+|u|^{2^*_s-2}u, \quad\textrm{in}\quad\mathbb{R}^N.
\end{equation}
Note that the associated action functional $\Phi$ is even, and reads as
\begin{eqnarray}
   \Phi(u)&=&\frac{1}{2}\int_{\mathbb{R}^N}|(-\Delta)^\frac{s}{2}u|^2dx+\frac{C_\alpha}{4}\iint_{\mathbb{R}^N\times \mathbb{R}^N}\frac{u^2(x)u^2(y)}{|x-y|^{N-\alpha}}dxdy - \frac{\lambda}{2^*_{s,\alpha}}\int_{\mathbb{R}^N}|u|^{2^*_{s,\alpha}}dx\nonumber\\
&&- \frac{\mu}{q_6}\int_{\mathbb{R}^N}|u|^{q_6}dx- \frac{1}{2^*_{s}}\int_{\mathbb{R}^N}|u|^{2^*_{s}}dx\nonumber\\
&=&\Phi_\lambda(u)- \frac{\mu}{q_6}\int_{\mathbb{R}^N}|u|^{q_6}dx- \frac{1}{2^*_{s}}\int_{\mathbb{R}^N}|u|^{2^*_{s}}dx.
\label{Phi-critical}
\end{eqnarray}

In this section, our existence and multiplicity results to~\eqref{Pcri} will be derived by critical point theorems from \cite{Mercuri-Perera} which require only a local Palais–Smale condition, as it is natural when dealing with equations involving the critical Sobolev exponent. For reader's convenience we recall the definition of pseudo-index.

For any $\rho > 0$, we define
\begin{equation}\label{Mrho}
\mathcal{M}_\rho = \{ u \in E : I(u) = \rho^{4s+\alpha-N} \} = \{ u_\rho : u \in \mathcal{M} \}.
\end{equation}
Let $\Gamma$ be the group of odd homeomorphisms of $E$ that act as the identity outside the set $\Phi^{-1}(0, c^*)$, for some $c^*>0$ such that $\Phi$ satisfies the $(\mathrm{PS})_c$ condition for all $c \in (0, c^*)$. Let $\mathcal{A}^*$ be the family of symmetric subsets of $E$.
For any $M \in \mathcal{A}^*$, the pseudo-index $i^*(M)$ with respect to $i$, $\mathcal{M}_\rho$, and $\Gamma$ (see Benci \cite{V. Benci}) is defined 
\[
i^*(M) = \min_{\gamma \in \Gamma} i(\gamma(M) \cap \mathcal{M}_\rho).
\]

\begin{proposition}[\cite{Mercuri-Perera}, Theorem 2.33]\label{Theorem 2.33 MP}
Let $A_0$ and $B_0$ be symmetric subsets of $\mathcal{M}$ such that $A_0$ is compact, $B_0$ is closed, and
\begin{equation*}
i(A_0) \geq k + m - 1, \qquad i(\mathcal{M} \setminus B_0) \leq k - 1
\end{equation*}
for some $k, m \geq 1$. Let $R > \rho > 0$ and let
\[
X = \{u_t : u \in A_0,\, 0 \leq t \leq R \}, \quad
A = \{u_R : u \in A_0 \}, \quad
B = \{u_\rho : u \in B_0 \}.
\]
Assume that
\begin{equation*}
\sup_{u \in A} \Phi(u) \leq 0 < \inf_{u \in B} \Phi(u), \qquad
\sup_{u \in X} \Phi(u) < c^*.
\end{equation*}
For $j = k, \ldots, k + m - 1$, let
\[
\mathcal{A}^*_j = \{M \in \mathcal{A}^* : M \text{ is compact and } i^*(M) \geq j \}
\]
and set
\[
c^*_j := \inf_{M \in \mathcal{A}^*_j} \max_{u \in M} \Phi(u).
\]
Then $0 < c^*_k \leq \cdots \leq c^*_{k+m-1} < c^*$, each $c^*_j$ is a critical value of $\Phi$, and $\Phi$ has $m$ distinct pairs of associated critical points.
\end{proposition}
The proof of Theorem \ref{sub near 0, asympt, subcritical} and Theorem \ref{sub near 0, asympt, critical} will be based on the following
\begin{corollary}[\cite{Mercuri-Perera}, Corollary 2.34]\label{Corollary 2.34 MP}
Let $A_0$ be a compact symmetric subset of $\mathcal{M}$ with $i(A_0) = m \geq 1$, let $R > \rho > 0$, and let
\[
A = \{u_R : u \in A_0\}, \qquad X = \{u_t : u \in A_0,\, 0 \leq t \leq R\}.
\]
Assume that
\[
\sup_{u \in A} \Phi(u) \leq 0 < \inf_{u \in \mathcal{M}_\rho} \Phi(u), \qquad \sup_{u \in X} \Phi(u) < c^*.
\]
Then $\Phi$ has $m$ distinct pairs of critical points at levels in $(0, c^*)$.
\end{corollary}

Before dealing with the main proofs of this section, we establish a compactness result for (PS) sequences within a certain energy range.  We denote by $\mathbb{S}$ the classical best Sobolev constant, such that
\begin{equation}\label{best S}
\|u\|_{L^{2^*_s}}^2\leq\frac{1}{\mathbb S}\|(-\Delta)^\frac{s}{2}u\|_{L^2}^2,\quad\forall u\in E.
\end{equation}
In particular, we show that the functional $\Phi$ defined in \eqref{Phi-critical} satisfies the Palais-Smale condition for $c\in(0,c^*)$ where $c^*=s\mathbb{S}^\frac{N}{2s}/N.$

\begin{lemma}\label{PSlocal}
Consider $N/4<s<1$. If $\lambda\in\mathbb{R}$, $\mu\geq0$ and $q_6\in(2^*_{s,\alpha},2^*_{s})$, then the functional $\Phi$, associated to problem \eqref{Pcri}, satisfies the $(PS)_c$ condition for $0<c<s\mathbb{S}^\frac{N}{2s}/N$.
\end{lemma}
\begin{proof}
Let $(u_n)$ be a $(PS)_c$ sequence for $\Phi$ in $E,$ namely such that
$$
\Phi(u_n)=c+o(1)\quad\mbox{and}\quad \Phi'(u_n)=o(1)\quad\mbox{as}\,\,n\to\infty.
$$
We claim that $(u_n)$ is bounded in $E$. In fact, arguing by contradiction, suppose that $\|u_n\|\to\infty$, for a subsequence. As in Lemma \ref{PS-superscaled}, we set
$$
t_n=\frac{1}{I(u_n)^{1/(4s+\alpha-N)}},\quad\tilde{u}_n=(u_n)_{t_n}\quad\mbox{and}\quad \tilde{t}_n=\frac{1}{t_n}=I(u_n)^{1/(4s+\alpha-N)}.
$$
As $I$ is coercive, it holds that $\tilde{t}_n\to\infty$.
Since by scaling we have $I(u_t)=t^{4s+\alpha-N}I(u)$ for all $u\in E$ and $t\geq 0$, it holds that $I(\tilde{u}_n)=1$, namely $\tilde{u}_n\in \mathcal{M}$, and therefore $(\tilde{u}_n)$ is bounded. Moreover, note that $(\tilde{u}_n)_{\tilde{t}_n}=u_n$ and hence
\begin{equation}\label{unOtcritical}
    \|u_n\|=\|(\tilde{u}_n)_{\tilde{t}_n}\|\leq\max\{(\tilde{t}_n)^{(4s+\alpha-N)/2},(\tilde{t}_n)^{(4s+\alpha-N)/4}\}I(\tilde{u}_n)=O((\tilde{t}_n)^{(4s+\alpha-N)/2})
\end{equation}
as $n\to\infty$. As
\begin{eqnarray*}
 \Phi_\lambda(u_t) &=&I(u_t)-\lambda J(u_t)= t^{4s+\alpha-N}\Phi_\lambda(u)\quad\mbox{for all}\,\, u\in E\,\,\mbox{and}\,\, t\geq0, 
\end{eqnarray*}
and $\Phi(u_n)=c+o(1)$, namely
\begin{eqnarray}
\frac{1}{2}\int_{\mathbb{R}^N}|(-\Delta)^\frac{s}{2}u_n|^2dx+\frac{C_\alpha}{4}\iint_{\mathbb{R}^N\times \mathbb{R}^N}\frac{u_n^2(x)u_n^2(y)}{|x-y|^{N-\alpha}}dxdy - \frac{\lambda}{2^*_{s,\alpha}}\int_{\mathbb{R}^N}|u_n|^{2^*_{s,\alpha}}dx\nonumber\\
- \frac{\mu}{q_6}\int_{\mathbb{R}^N}|u_n|^{q_6}dx- \frac{1}{2^*_{s}}\int_{\mathbb{R}^N}|u_n|^{2^*_{s}}dx=c+o(1),
\label{Phi(un)c}
\end{eqnarray}
  we may conclude that
\begin{eqnarray}
(\tilde{t}_n)^{4s+\alpha-N}\Phi_\lambda(\tilde u_n)&
=&\frac{\mu}{q_6}\int_{\mathbb{R}^N}|u_n|^{q_6}dx+ \frac{1}{2^*_{s}}\int_{\mathbb{R}^N}|u_n|^{2^*_{s}}dx+c+o(1)\nonumber\\
&=&\frac{\mu(\tilde{t}_n)^{\theta q_6-N}}{q_6}\int_{\mathbb{R}^N}|\tilde u_n|^{q_6}dx
+ \frac{(\tilde{t}_n)^{\theta 2^*_{s}-N}}{2^*_{s}}\int_{\mathbb{R}^N}|\tilde u_n|^{2^*_{s}}dx+c+o(1). \label{eqPhiLn3}
\end{eqnarray}
We stress that $q_6\in(2^*_{s,\alpha},2^*_{s})$ yields
$$
2^*_{s}\theta -N>q_6\theta -N>2^*_{s,\alpha}\theta -N= 4s+\alpha-N>0. 
$$
As a consequence, since $\mu\geq0$, the boundedness of $(\Phi_\lambda(\tilde{u}_n))$ gives us 
\begin{equation}\label{un2*b}
\int_{\mathbb{R}^N}|\tilde u_n|^{2^*_{s}}dx=O((\tilde{t}_n)^{-b}),\quad\mbox{as}\,\,n\to\infty.
\end{equation}
with $b=\theta 2^*_{s}-4s-\alpha>0,$ hence $\tilde{u}_n\to0$ in $L^{2^*_{s}}(\mathbb{R}^N)$.
 Picking $r\in (p_{\rm rad},2^*_{s,\alpha})$, the embedding $E\hookrightarrow L^r(\mathbb{R}^N)$ forces $(\tilde{u}_n)$ to be bounded in $L^r(\mathbb{R}^N)$. Pick now $\sigma\in(0,1)$ such that
 $
 \frac{1}{2^*_{s,\alpha}}= \frac{\sigma}{2^*_{s}} +\frac{1-\sigma}{r}.
 $
 By interpolation, we obtain
 $$
 \|\tilde{u}_n\|_{L^{2^*_{s,\alpha}}}\leq \|\tilde{u}_n\|_{L^{2^*_{s}}}^{\sigma}\|\tilde{u}_n\|_{L^r}^{1-\sigma}\leq C\|\tilde{u}_n\|_{L^{2^*_{s}}}^{\sigma}.
 $$
Hence $\tilde{u}_n\to0$ in $L^{2^*_{s,\alpha}}(\mathbb{R}^N)$. In the same fashion, for some $\tilde\sigma\in(0,1)$ such that
 $
 \frac{1}{q_6}= \frac{\tilde\sigma}{2^*_{s}} +\frac{1-\tilde\sigma}{2^*_{s,\alpha}},
 $
 we also have
 $$
 \|\tilde{u}_n\|_{L^{q_6}}\leq \|\tilde{u}_n\|_{L^{2^*_{s}}}^{\tilde\sigma}\|\tilde{u}_n\|_{L^{2^*_{s,\alpha}}}^{1-\tilde\sigma}.
 $$
Since $(\theta q_6-4s-\alpha)=b(\tilde\sigma q_6/2^*_{s})$, combining the two inequalities above with \eqref{un2*b}, we get
$$
(\tilde{t}_n)^{\theta q_6-4s-\alpha}\int_{\mathbb{R}^N}|\tilde{u}_n|^{q_6}dx
\leq C(\tilde{t}_n)^{\theta q_6-4s-\alpha}(\tilde{t}_n)^{-b\tilde\sigma q_6/2^*_{s}}\|\tilde{u}_n\|_{L^{2^*_{s,\alpha}}}^{1-\tilde\sigma}=C\|\tilde{u}_n\|_{L^{2^*_{s,\alpha}}}^{1-\tilde\sigma},
 $$
which yields
\begin{equation*}\label{un2*q6}
    \int_{\mathbb{R}^N}|\tilde{u}_n|^{q_6}dx=o(\tilde{t}_n^{-(\theta q_6-4s-\alpha)})
\quad\mbox{as}\,\,n \to\infty.
\end{equation*}
On the other hand, since $\Phi'(u_n)u_n=o(\|u_n\|),$ we have
\begin{eqnarray}
   \int_{\mathbb{R}^N}|(-\Delta)^\frac{s}{2}u_n|^2dx+{C_\alpha}\iint_{\mathbb{R}^N\times \mathbb{R}^N}\frac{u_n^2(x)u_n^2(y)}{|x-y|^{N-\alpha}}dxdy - {\lambda}\int_{\mathbb{R}^N}|u_n|^{2^*_{s,\alpha}}dx \nonumber\\
   =\mu\int_{\mathbb{R}^N}|u_n|^{q_6}dx+\int_{\mathbb{R}^N}|u_n|^{2^*_s}dx+o(\|u_n\|), \label{Phi'(un)o(1)}
\end{eqnarray}
that is
$$
\Phi'_\lambda(u_n){u}_n=\mu\int_{\mathbb{R}^N}|u_n|^{q_6}dx+\int_{\mathbb{R}^N}|u_n|^{2^*_s}dx+o(\|u_n\|)
$$
as $n\to\infty$. It is worth recalling now that
$$
\Phi'_\lambda(u_t)u_t=\mathcal{A}(u_t)u_t-\lambda \mathcal{B}(u_t)u_t=t^{(4s+\alpha-N)}\Phi_\lambda'(u)u,\quad \forall u\in E, \forall t\geq0,
$$ 
and $u_n=(\tilde{u}_n)_{\tilde t_n}.$ By \eqref{unOtcritical}, it then follows that 
\begin{eqnarray}\label{eqPhi'Lncri}
(\tilde{t}_n)^{(4s+\alpha-N)}\Phi'_\lambda(\tilde u_n)\tilde{u}_n=
\mu(\tilde{t}_n)^{\theta q_6-N}\int_{\mathbb{R}^N}|\tilde u_n|^{q_6}dx+(\tilde{t}_n)^{\theta 2^*_s-N}\int_{\mathbb{R}^N}|\tilde u_n|^{2^*_s}dx+o((\tilde{t}_n)^{(4s+\alpha-N)/2}).
\end{eqnarray}
Multiplying \eqref{eqPhiLn3} by $2^*_s$ and subtracting \eqref{eqPhi'Lncri} we finally obtain
\begin{eqnarray*}
  & &\left(\frac{2^*_s}{2}-1\right)\int_{\mathbb{R}^N}|(-\Delta)^\frac{s}{2}\tilde u_n|^2dx+C_\alpha\left(\frac{2^*_s}{4}-1\right) \iint_{\mathbb{R}^N\times\mathbb{R}^N}\frac{\tilde u_n^2(x)\tilde u_n^2(y)}{|x-y|^{N-\alpha}}dxdy \\
  &&\qquad\qquad=\lambda\left(\frac{2^*_s}{2^*_{s,\alpha}} -1\right)\int_{\mathbb{R}^N}|\tilde u_n|^{2^*_{s,\alpha}}dx
  +\mu\left(\frac{2^*_s}{q_6} -1\right)(\tilde{t}_n)^{\theta q_6-4s-\alpha}\int_{\mathbb{R}^N}|\tilde u_n|^{q_6}dx+o(1)\\
  && \qquad\qquad=\,\, o(1),\quad\mbox{as}\,\,n\to\infty.
\end{eqnarray*}
As $2^*_s>4$ for $s>N/4$, we get $\tilde u_n\to0$, which contradicts the fact that $\tilde u_n\in\mathcal{M},$
and this is enough to conclude that $(u_n)$ is bounded in $E$. \\
Since $E$ is reflexive Banach space, up to a subsequence, $u_n\rightharpoonup u$ for some $u\in E$. By the compactness of the embedding $E\hookrightarrow L^r(\mathbb{R}^N)$ for $r\in (p_{\rm rad},2^*_s)$, we can assume that $u_n\rightarrow u$ in $L^r(\mathbb{R}^N)$ and $u_n(x)\rightarrow u(x)$ almost everywhere. In particular, $u_n\rightarrow u$  in $L^{2^*_{s,\alpha}}(\mathbb{R}^N)$ and in $L^{q_6}(\mathbb{R}^N)$, which implies that
$$
\int_{\mathbb{R}^N}| u_n|^{2^*_{s,\alpha}-2}u_nvdx\to \int_{\mathbb{R}^N}| u|^{2^*_{s,\alpha}-2}uvdx\quad\mbox{and}
\quad\int_{\mathbb{R}^N}| u_n|^{q_6-2}u_nvdx\to \int_{\mathbb{R}^N}| u|^{q_6-2}uvdx,
$$
for any $v\in E$. As in the proof of Lemma \ref{verify H7}, by the weak convergence  $u_n\rightharpoonup u$ in $E$ we see that
\begin{equation}\label{unuHs}
   \int_{\mathbb{R}^N}(-\Delta)^\frac{s}{2}u_n(-\Delta)^\frac{s}{2}vdx \to  \int_{\mathbb{R}^N}(-\Delta)^\frac{s}{2}u(-\Delta)^\frac{s}{2}vdx,
   \end{equation}
and
\begin{equation*}\label{unuQaq}
\iint_{\mathbb{R}^N\times\mathbb{R}^N}\frac{u^2_n(x)u_n(y)v(y)}{|x-y|^{N-\alpha}}dxdy\to\iint_{\mathbb{R}^N\times\mathbb{R}^N}\frac{u^2(x)u(y)v(y)}{|x-y|^{N-\alpha}}dxdy.
\end{equation*}
Moreover, by a standard weak convergence argument we also have
$$
\int_{\mathbb{R}^N}| u_n|^{2^*_{s}-2}u_nvdx\to \int_{\mathbb{R}^N}| u|^{2^*_{s}-2}uvdx.
$$
Thus, for each $v\in E$, it holds that
$$
0=\lim_{n\to\infty}\Phi'(u_n)v=\Phi'(u)v
$$
and therefore $u$ is a critical point for $\Phi$. Testing with $v=u$, the Nehari identity $\Phi'(u)u=0$ yields
\begin{equation}\label{phi'(u)u}
  \int_{\mathbb{R}^N}|(-\Delta)^\frac{s}{2}u|^2dx+C_\alpha\iint_{\mathbb{R}^N\times \mathbb{R}^N}\frac{u^2(x)u^2(y)}{|x-y|^{N-\alpha}}dxdy - {\lambda}\int_{\mathbb{R}^N}|u|^{2^*_{s,\alpha}}dx
- {\mu}\int_{\mathbb{R}^N}|u|^{q_6}dx- \int_{\mathbb{R}^N}|u|^{2^*_{s}}dx=0.
\end{equation}
Moreover, by the Pohozaev identity provided in Lemma \ref{pohozaev}, we also get
\begin{eqnarray}\label{pohozaev-critical}
  \frac{N-2s}{2}\int_{\mathbb{R}^N}|(-\Delta)^\frac{s}{2}u|^2dx+\frac{C_\alpha(N+\alpha)}{4}\iint_{\mathbb{R}^N\times \mathbb{R}^N}\frac{u^2(x)u^2(y)}{|x-y|^{N-\alpha}}dxdy\nonumber\\
 \quad - \frac{\lambda N}{2^*_{s,\alpha}}\int_{\mathbb{R}^N}|u|^{2^*_{s,\alpha}}dx
- \frac{\mu N}{q_6}\int_{\mathbb{R}^N}|u|^{q_6}dx- \frac{N}{2^*_{s}}\int_{\mathbb{R}^N}|u|^{2^*_{s}}dx=0.
\end{eqnarray}
With these preliminaries in place we can now show strong convergence by a Brezis-Lieb type argument adapted to the present fractional and nonlocal context. To this aim, observe now that by \eqref{unuHs} it is standard to see that
\begin{equation}\label{unuHso(1)}
   \int_{\mathbb{R}^N}|(-\Delta)^\frac{s}{2}(u_n-u)|^2dx =  \int_{\mathbb{R}^N}|(-\Delta)^\frac{s}{2}u_n|^2dx-\int_{\mathbb{R}^N}|(-\Delta)^\frac{s}{2}u|^2dx+o(1).
   \end{equation}
On the other hand, since by boundedness and the fact that $u_n(x)\to u(x)$ almost everywhere, the classical Brezis-Lieb lemma yields
\begin{equation}\label{BLLemma}
    \int_{\mathbb{R}^N}|u_n-u|^{2^*_{s}}dx=\int_{\mathbb{R}^N}|u_n|^{2^*_{s}}dx-\int_{\mathbb{R}^N}|u|^{2^*_{s}}dx+o(1).
\end{equation}
On the other hand, the nonlocal Brezis-Lieb lemma (see \cite[Proposition 4.1]{Mercuri-Moroz-VS-2016}) also yields
\begin{eqnarray}
\iint_{\mathbb{R}^N\times\mathbb{R}^N}\frac{u_n^2(x)u_n^2(y)}{|x-y|^{N-\alpha}}dxdy
\!\!\!&-&\!\!\!\iint_{\mathbb{R}^N\times\mathbb{R}^N}\frac{u^2(x)u^2(y)}{|x-y|^{N-\alpha}}dxdy\nonumber\\
&\geq&\iint_{\mathbb{R}^N\times\mathbb{R}^N}\frac{(u_n-u)^2(x)(u_n-u)^2(y)}{|x-y|^{N-\alpha}}dxdy+o(1).\label{nonlocalBLL}
\end{eqnarray}
Hence, subtracting \eqref{phi'(u)u} by \eqref{Phi'(un)o(1)}, putting together \eqref{unuHso(1)}-\eqref{nonlocalBLL} and the fact that $u_n\rightarrow u$  in $L^{2^*_{s,\alpha}}(\mathbb{R}^N)$ and in $L^{q_6}(\mathbb{R}^N)$, we obtain
\begin{eqnarray*}
\int_{\mathbb{R}^N}|(-\Delta)^\frac{s}{2}(u_n-u)|^2dx+C_\alpha\iint_{\mathbb{R}^N\times\mathbb{R}^N}\frac{(u_n-u)^2(x)(u_n-u)^2(y)}{|x-y|^{N-\alpha}}dxdy\leq \int_{\mathbb{R}^N}|u_n-u|^{2^*_{s}}dx+o(1).
\end{eqnarray*}
Observe now that the classical Sobolev inequality \eqref{best S} gives us
\begin{eqnarray}\label{convS}
\int_{\mathbb{R}^N}|(-\Delta)^\frac{s}{2}(u_n-u)|^2dx
\leq \mathbb{S}^{-\frac{2^*_s}{2}}\left(\int_{\mathbb{R}^N}|(-\Delta)^\frac{s}{2}(u_n-u)|^2dx\right)^\frac{2^*_s}{2}+o(1).
\end{eqnarray}
Let us assume now that there is no subsequence of $(u_n)$ that converges to $u$ and prove that $c\geq \frac{s}{N}\mathbb{S}^\frac{N}{2s},$ this will be enough to complete the proof. By \eqref{convS} and \eqref{unuHso(1)} we  have
\begin{equation}\label{no-convergence}
 \mathbb{S}^\frac{2^*_s}{2^*_s-2}\leq\int_{\mathbb{R}^N}|(-\Delta)^\frac{s}{2}(u_n-u)|^2dx+ o_n(1)=\int_{\mathbb{R}^N}|(-\Delta)^\frac{s}{2}u_n|^2dx-\int_{\mathbb{R}^N}|(-\Delta)^\frac{s}{2}u|^2dx+ o(1).
\end{equation}
Hence, multiplying \eqref{Phi'(un)o(1)} by $1/2^*_s$ and subtracting by \eqref{Phi(un)c}, we obtain
\begin{eqnarray*}
 c & =&\left(\frac{1}{2}-\frac{1}{2^*_s}\right)\int_{\mathbb{R}^N}|(-\Delta)^\frac{s}{2}u_n|^2dx+C_\alpha\left(\frac{1}{4}-\frac{1}{2^*_s}\right) \iint_{\mathbb{R}^N\times\mathbb{R}^N}\frac{ u_n^2(x) u_n^2(y)}{|x-y|^{N-\alpha}}dxdy \\
  &&-\lambda\left(\frac{1}{2^*_{s,\alpha}} -\frac{1}{2^*_s}\right)\int_{\mathbb{R}^N}|u_n|^{2^*_{s,\alpha}}dx
  -\mu\left(\frac{1}{q_6} -\frac{1}{2^*_s}\right)\int_{\mathbb{R}^N}|u_n|^{q_6}dx+o(1).
\end{eqnarray*}
On the other hand, by \eqref{no-convergence} and \eqref{nonlocalBLL} as $n\to\infty$ we get
\begin{eqnarray*}
 c & \geq&\frac{s}{N}\mathbb{S}^\frac{N}{2s}+\frac{s}{N}\int_{\mathbb{R}^N}|(-\Delta)^\frac{s}{2}u|^2dx+C_\alpha\left(\frac{1}{4}-\frac{1}{2^*_s}\right) \iint_{\mathbb{R}^N\times\mathbb{R}^N}\frac{ u^2(x) u^2(y)}{|x-y|^{N-\alpha}}dxdy \\
  &&-\lambda\left(\frac{1}{2^*_{s,\alpha}} -\frac{1}{2^*_s}\right)\int_{\mathbb{R}^N}|u|^{2^*_{s,\alpha}}dx
  -\mu\left(\frac{1}{q_6} -\frac{1}{2^*_s}\right)\int_{\mathbb{R}^N}|u|^{q_6}dx.
\end{eqnarray*}
Now, multiplying \eqref{pohozaev-critical} by any $a$ and \eqref{phi'(u)u} by $\tilde{a}$, both to be defined later, the inequality above yields
\begin{eqnarray*}
 c & \geq&\frac{s}{N}\mathbb{S}^\frac{N}{2s}+ \left(\frac{aN}{2^*_s}+\tilde{a}+\frac{s}{N}\right)\int_{\mathbb{R}^N}|(-\Delta)^\frac{s}{2}u|^2dx\\
 &&+ C_\alpha\left(\frac{a(N+\alpha)}{4}+\tilde{a}+\frac{1}{4}-\frac{1}{2^*_s}\right) \iint_{\mathbb{R}^N\times\mathbb{R}^N}\frac{ u^2(x) u^2(y)}{|x-y|^{N-\alpha}}dxdy \\
  &&- \lambda\left(\frac{aN}{2^*_{s,\alpha}}+\tilde{a}+\frac{1}{2^*_{s,\alpha}} -\frac{1}{2^*_s}\right)\int_{\mathbb{R}^N}|u|^{2^*_{s,\alpha}}dx
  -\mu\left(\frac{aN}{q_6}+\tilde{a}+\frac{1}{q_6} -\frac{1}{2^*_s}\right)\int_{\mathbb{R}^N}|u|^{q_6}dx\\
  &&-\left(\frac{aN}{2^*_{s}}+\tilde{a}\right)\int_{\mathbb{R}^N}|u|^{2^*_{s}}dx.
\end{eqnarray*}
Setting in particular $a=(4s+\alpha-N)^{-1}>0$ and $\tilde{a}=\frac{-s}{N}-\frac{aN}{2^*_{s}}$ we have that the first three terms in the inequality above cancel out. 
As $q_6>2^*_{s,\alpha},$ we get
$$
-d:=\left(\frac{aN}{q_6}+\tilde{a}+\frac{1}{q_6} -\frac{1}{2^*_s}\right)<\left(\frac{aN}{2^*_{s,\alpha}}+\tilde{a}+\frac{1}{2^*_{s,\alpha}} -\frac{1}{2^*_s}\right)=0.
$$
Hence, as $\mu\geq0$, we finally get
\begin{eqnarray}\label{PS level estimate}
 c & \geq&\frac{s}{N}\mathbb{S}^\frac{N}{2s}
   +\mu d\int_{\mathbb{R}^N}|u|^{q_6}dx
  +\frac{s}{N}\int_{\mathbb{R}^N}|u|^{2^*_{s}}dx\geq\frac{s}{N}\mathbb{S}^\frac{N}{2s},
\end{eqnarray}
and this completes the proof.
\end{proof} 

We are now ready to deal with the main proofs of this section, on the critical equation \eqref{Pcri}.

\begin{proof-th} \textbf{\textit{of Theorem \ref{th cri-super-mult}.}}\addcontentsline{toc}{subsection}{Proof of Theorem~\ref{th cri-super-mult}}
 Our aim is to apply Proposition \ref{Theorem 2.33 MP}. Let us then consider $\lambda\in(0,\lambda_k)$. By Lemma \ref{PSlocal} we know that $\Phi$ satisfies the $(PS)_c$ for all $0<c<c^*:=(s/N)\mathbb{S}^\frac{N}{2s}$. Pick $\varepsilon\in(0,\lambda_{k+m}-\lambda_{k+m-1})$. Then, by Theorem \ref{lambdak}\ref{lambdak(iii)} we have
$$
i(\mathcal{M}\backslash\tilde{\Psi}_{\lambda_{k+m-1}+\varepsilon})=k+m-1.
$$
Since $\mathcal{M}\backslash\tilde{\Psi}_{\lambda_{k+m-1}+\varepsilon}$ is an open symmetric subset of $\mathcal{M}$,  there is a compact symmetric $C\subset\mathcal{M}\backslash\tilde{\Psi}_{\lambda_{k+m-1}+\varepsilon}$ with $i(C)=k+m-1$  (see the proof of Proposition 3.1 in Degiovanni and Lancelotti \cite{De-Lan-2007}).
We set $A_0=C$ and $B_0=\tilde{\Psi}_{\lambda_k}$. The properties of $C$ yield
$$
i(A_0)=k+m-1.
$$
As far as $B_0,$ we distinguish the cases $\lambda_k=\lambda_1$ and $\lambda_k>\lambda_1$. If $\lambda_1=\cdots=\lambda_k$, we have
$$
B_0=\{u\in\mathcal{M}:\tilde{\Psi}(u)\geq\lambda_1\}=\mathcal{M},
$$
by Theorem \ref{lambdak}\ref{lambdak(i)}. Then, $i(\mathcal{M}\backslash B_0)=i(\emptyset)=0\leq k-1$. In the other case, it holds $\lambda_{l-1}<\lambda_l=\cdots=\lambda_k$ for some $2\leq l\leq k$. Then, Theorem \ref{lambdak}\ref{lambdak(iii)} implies that
$$
i(\mathcal{M}\backslash B_0)=
i(\mathcal{M}\backslash\tilde{\Psi}_{\lambda_{l}})=l-1\leq k-1.
$$
Now, for some $R>\rho>0$, we set
\begin{eqnarray*}
    X=\{u_t:u\in A_0,\,0\leq t\leq R\},\quad A=\{u_R:u\in A_0\}\quad\mbox{and}\quad B=\{u_{\rho}:u\in B_0\},
\end{eqnarray*}
and show that it is possible to choose suitable $R>\rho>0$ and $\delta_k>0$ in such a way that, for all $\lambda\in(\lambda_k-\delta_k,\lambda_k)$, it holds that
$$
\sup_{v\in A}\Phi(v)\leq 0<\inf_{v\in B}\Phi(v)\quad\mbox{and}
\quad\sup_{v\in X}\Phi(v)<\frac{s}{N}\mathbb{S}^\frac{N}{2s}.
$$
In fact, for all $\lambda\in\mathbb{R}$, $u\in\mathcal{M}$ and $t\geq0$ we have
\begin{equation}\label{Phi cri in M}
    \Phi(u_t)=t^{4s+\alpha-N}\left(1-\frac{\lambda}{\tilde{\Psi}(u)}\right)-\frac{\mu t^{\theta q_6-N}}{q_6}\int_{\mathbb{R}^N}|u|^{q_6}dx-\frac{t^{\theta 2^*_s-N}}{2^*_s}\int_{\mathbb{R}^N}|u|^{2^*_s}dx.
\end{equation}
We recall that $\theta=(2s+\alpha)/2$ and $\theta 2^*_s>\theta q_6>\theta 2^*_{s,\alpha}=4s+\alpha$.
Since $\mathcal{M}$ is bounded, there exist constants $c_1,c_2>0$ such that
$$
\Phi(u_t)\geq 
t^{4s+\alpha-N}\left(1-\frac{\lambda}{\tilde{\Psi}(u)}-c_1\mu t^{\theta q_6-(4s+\alpha)}-c_2t^{\theta 2^*_s-(4s+\alpha)}\right).
$$
 Hence, since $\tilde{\Psi}(u)\geq\lambda_k>\lambda$ for all $u\in B_0$, we get 
$$
\Phi(u_t)\geq 
t^{4s+\alpha-N}\left(1-\frac{\lambda}{\lambda_k}-c_1\mu t^{\theta q_6-(4s+\alpha)}-c_2t^{\theta 2^*_s-(4s+\alpha)}\right)>0,
$$
provided $t=\rho$ is small enough. 
This gives us that
\begin{equation}\label{inf in B}
    \inf_{v\in B}\Phi(v)>0.
\end{equation}
Now, for $u\in A_0\subset\mathcal{M}\backslash\tilde{\Psi}_{\lambda_{k}+\varepsilon}$ we have
$
\tilde{\Psi}(u)<\lambda_{k}+\varepsilon<\lambda_{k+m},
$
which implies
$$
\frac{1}{2^*_{s,\alpha}}\int_{\mathbb{R}^N}|u|^{2^*_{s,\alpha}}dx=J(u)>\frac{1}{\lambda_{k+m}}.
$$
Picking $r\in(p_{\rm rad},2^*_{s,\alpha})$ and $\sigma\in(0,1)$ such that $\frac{1}{2^*_{s,\alpha}}=\frac{1-\sigma}{r}+\frac{\sigma}{2^*_{s}}$,  the interpolation inequality and the boundedness of $\mathcal{M}$ yields
$$
\|u\|_{L^{2^*_{s,\alpha}}}\leq \|u\|_{L^{r}}^{1-\sigma}\|u\|_{L^{2^*_{s}}}^\sigma\leq  \|u\|_{L^{2^*_{s}}}^\sigma.
$$
Hence, there exists $c_3>0$, independent on $\varepsilon$, such that
$$
\frac{1}{2^*_{s}}\int_{\mathbb{R}^N}|u|^{2^*_{s}}dx\geq c_3,\quad\forall u\in A_0.
$$
Since $\lambda,\mu\geq0$, by \eqref{Phi cri in M} we get
 $$
\Phi(u_t)\leq t^{4s+\alpha-N}-\frac{t^{\theta 2^*_s-N}}{2^*_s}\int_{\mathbb{R}^N}|u|^{2^*_s}dx\leq 
t^{4s+\alpha-N}\left(1-c_3t^{\theta 2^*_s-(4s+\alpha)}\right)\leq0, \quad\forall u\in A_0,
$$
for $t=R>\rho$, large enough. Hence,
\begin{equation}\label{sup in A 2} 
    \sup_{v\in A}\Phi(v)\leq0.
\end{equation}
We now estimate the supremum of $\Phi$ in $X$. Again by  \eqref{Phi cri in M}, since $\lambda,\mu\geq0$, we obtain
\begin{equation*}
    \Phi(u_t)\leq t^{4s+\alpha-N}\left(1-\frac{\lambda}{\lambda_{k}+\varepsilon}\right)-c_3t^{\theta 2^*_s-N}, \quad\forall u\in A_0, \,t\geq0.
\end{equation*}
Thus, 
\begin{eqnarray*}
   \sup_{u\in A_0, 0\leq t\leq R}\Phi(u_t)&\leq &\max_{t\geq0}\left[t^{4s+\alpha-N}\left(1-\frac{\lambda}{\lambda_{k}+\varepsilon}\right)-c_3t^{\theta 2^*_s-N}\right]\\
    &=&\frac{s}{N}\frac{2^\frac{N}{2s}}{(c_32^*_s)^{\frac{N}{2s}-1}}\left(1-\frac{\lambda}{\lambda_{k}+\varepsilon}\right)^\frac{N}{2s}.
\end{eqnarray*}
We set
$$
\delta_k=\frac{\lambda_k}{2}\mathbb{S}\left(c_32^*_s\right)^\frac{2}{2^*_s}.
$$
For each $\lambda_k-\delta_k<\lambda<\lambda_k$, we can see that for $\varepsilon>0$ small enough it holds that
\begin{eqnarray*}
   \sup_{u\in A_0, 0\leq t\leq R}\Phi(u_t)<\frac{s}{N}\mathbb{S}^\frac{N}{2s}.
\end{eqnarray*}
Finally, combining this inequality, \eqref{inf in B} and \eqref{sup in A 2},
the conclusion now follows by Proposition \ref{Theorem 2.33 MP}, and this concludes the proof.
\end{proof-th}

We conclude this section with the proof of Theorem~\ref{th cri l<l1}, which does not require $\lambda$ to lie near an eigenvalue, provided $\mu$ is sufficiently large. In fact, $\lambda$ may be any real number.

\begin{proof-th} \textbf{\textit{of Theorem \ref{th cri l<l1}.}}\addcontentsline{toc}{subsection}{Proof of Theorem~\ref{th cri l<l1}}
 Although the argument is similar to the above proof, for sake of clarity we prefer not to leave it out. We shall conclude again by Proposition \ref{Theorem 2.33 MP}. To this aim, note that by Lemma \ref{PSlocal} $\Phi$ satisfies the $(PS)_c$ for all $0<c<c^*:=(s/N)\mathbb{S}^\frac{N}{2s}$. Pick $m\in\mathbb{N}$ and set 
 $$
 j=\min\{k\in\mathbb{N}: \lambda<\lambda_k\}\quad\mbox{and}\quad \tilde{m}=j+m-1.
 $$
Let  $l\in\mathbb{N}$ and $a\in\mathbb{R}$ such that $\lambda_{\tilde m}=\cdots=\lambda_{\tilde m+l-1}<a<\lambda_{\tilde m+l}$.   Then, by Theorem \ref{lambdak}\ref{lambdak(iii)} we have
$$
i(\mathcal{M}\backslash\tilde{\Psi}_{a})=\tilde m+l-1.
$$
Since $\mathcal{M}\backslash\tilde{\Psi}_{a}$ is an open symmetric subset of $\mathcal{M}$,  there exists a compact symmetric set $C\subset\mathcal{M}\backslash\tilde{\Psi}_{a}$ with $i(C)=\tilde m+l-1$  (see the proof of Proposition 3.1 in Degiovanni and Lancelotti \cite{De-Lan-2007}).
Set
$$
A_0=C\quad \mbox{and}\quad B_0=\tilde{\Psi}_{\lambda_j}.
$$ 
Hence
$$
i(A_0)=\tilde m+l-1=(m+l-1)+j-1.
$$
Dealing with $B_0$ we analyse the cases $j=1$ and $j>1$ separately. If $j=1$, we have
$$
B_0=\{u\in\mathcal{M}:\tilde{\Psi}(u)\geq\lambda_1\}=\mathcal{M},
$$
by Theorem \ref{lambdak}\ref{lambdak(i)}. Hence $i(\mathcal{M}\backslash B_0)=i(\emptyset)=0\leq j-1$. In the other case, it holds that $\lambda_{j-1}\leq\lambda<\lambda_j$, hence Theorem \ref{lambdak}\ref{lambdak(iii)} yields 
$$
i(\mathcal{M}\backslash B_0)=
i(\mathcal{M}\backslash\tilde{\Psi}_{\lambda_{j}})=j-1.
$$
Now, for some $R>\rho>0$ to be chosen later, we set
\begin{eqnarray*}
    X=\{u_t:u\in A_0,\,0\leq t\leq R\},\quad A=\{u_R:u\in A_0\}\quad\mbox{and}\quad B=\{u_{\rho}:u\in B_0\}.
\end{eqnarray*}
We claim that we may choose $R,\,\mu_m>0$ such that, for each $\mu\geq \mu_m$ there is $\rho\in(0,R)$ such that
$$
\sup_{v\in A}\Phi(v)\leq 0<\inf_{v\in B}\Phi(v)\quad\mbox{and}
\quad\sup_{v\in X}\Phi(v)<\frac{s}{N}\mathbb{S}^\frac{N}{2s}.
$$
Recall that 
\begin{equation}\label{Phi cri MPT}
    \Phi(u_t)=t^{4s+\alpha-N}\left(1-\frac{\lambda}{\tilde{\Psi}(u)}\right)-\frac{\mu t^{\theta q_6-N}}{q_6}\int_{\mathbb{R}^N}|u|^{q_6}dx-\frac{t^{\theta 2^*_s-N}}{2^*_s}\int_{\mathbb{R}^N}|u|^{2^*_s}dx,
\end{equation}
for all $u\in\mathcal{M}$ and $t\geq0$.
Since $A_0\subset\mathcal{M}$ is compact, there exists $c_0>0$ such that 
$$
\frac{1}{2^*_{s,\alpha}}\int_{\mathbb{R}^N}|u|^{2^*_{s,\alpha}}dx,\quad \frac{1}{2^*_{s}}\int_{\mathbb{R}^N}|u|^{2^*_{s}}dx, \quad \frac{1}{q_6}\int_{\mathbb{R}^N}|u|^{q_6}dx \geq c_0,\quad\forall u\in A_0.
$$
Setting $\lambda^-=-\min\{\lambda,0\}$, by the boundedness of $\|u\|_{L^{2^*_{s,\alpha}}}$ in $\mathcal{M}$ we get
$$
1-\frac{\lambda}{\tilde{\Psi}(u)}\leq 1+\frac{\lambda^-}{2^*_{s,\alpha}}\int_{\mathbb{R}^N}|u|^{2^*_{s,\alpha}}dx\leq \tilde{b},\quad\forall u\in A_0,
$$
for some $\tilde{b}>0$.  Hence, as $\mu>0$, by \eqref{Phi cri MPT} we obtain
 $$
\Phi(u_t)\leq \tilde{b}t^{4s+\alpha-N}-\frac{t^{\theta 2^*_s-N}}{2^*_s}\int_{\mathbb{R}^N}|u|^{2^*_s}dx\leq 
t^{4s+\alpha-N}\left(\tilde{b}-c_0t^{\theta 2^*_s-(4s+\alpha)}\right)\leq0, \quad\forall u\in A_0,
$$
for $t=R$ large enough, independent on $\mu$. Hence
\begin{equation}\label{sup in A}
    \sup_{v\in A}\Phi(v)\leq0.
\end{equation}
Let us now estimate the supremum of $\Phi$ in $X$. Using again  \eqref{Phi cri MPT} we obtain
\begin{equation*}
    \Phi(u_t)\leq \tilde{b} t^{4s+\alpha-N}-\mu c_0t^{\theta q_6-N}, \quad\forall u\in A_0, \,t\geq0.
\end{equation*}
Thus
\begin{eqnarray*}
   \sup_{u\in A_0, 0\leq t\leq R}\Phi(u_t)&\leq &\max_{t\geq0}\left[\tilde{b}t^{4s+\alpha-N}-\mu c_0t^{\theta q_6-N}\right]\\
    &=&\left[\theta q_6-(4s+\alpha)\right]\left[\frac{\tilde{b}}{\theta q_6-N}\right]^\frac{\theta q_6-N}{\theta q_6-(4s+\alpha)}\left[\frac{4s+\alpha-N}{c_0\mu}\right]^\frac{4s+\alpha-N}{\theta q_6-(4s+\alpha)}.
\end{eqnarray*}
Picking $\mu_m>0$ sufficiently large, for $\mu\geq\mu_m$ it holds that
\begin{eqnarray}\label{sup in X lambda}
   \sup_{v\in X}\Phi(v)= \sup_{u\in A_0, 0\leq t\leq R}\Phi(u_t)<\frac{s}{N}\mathbb{S}^\frac{N}{2s}.
\end{eqnarray}
To finalise the proof, let us now estimate $\Phi$ on $B$. By \eqref{Phi cri MPT}, since $\mathcal{M}$ is bounded, there exist $c_1,c_2>0$ such that
$$
\Phi(u_t)\geq 
t^{4s+\alpha-N}\left(1-\frac{\lambda^+}{\tilde{\Psi}(u)}-c_1 \mu t^{\theta q_6-(4s+\alpha)}-c_2t^{\theta 2^*_s-(4s+\alpha)}\right),
$$
where $\lambda^+=\max\{\lambda,0\}$. Hence, as $\tilde{\Psi}(u)\geq\lambda_j>\lambda^+$ for all $u\in B_0$ and $\theta 2^*_s>\theta q_6>\theta 2^*_{s,\alpha}=4s+\alpha$,   we get 
$$
\Phi(u_t)\geq 
t^{4s+\alpha-N}\left(1-\frac{\lambda_+}{\lambda_j}-c_1\mu t^{\theta q_6-(4s+\alpha)}-c_2t^{\theta 2^*_s-(4s+\alpha)}\right)>0,
$$
provided $t=\rho\in(0,R)$ is small enough. 
This gives us that
\begin{equation*}
  \inf_{v\in B}\Phi(v)>0.
\end{equation*}
Putting together the last inequality with \eqref{sup in A} and \eqref{sup in X lambda}, the conclusion now follows by Proposition \ref{Theorem 2.33 MP}.
\end{proof-th}


\section{Solutions to problems subscaled near the origin}\label{section sub near 0}

In this section we deal with local nonlinearities $f:[0,\infty)\times\mathbb{R}\to\mathbb{R}$  of the form
$$
  f(|x|,t) = |t|^{\beta-2}t+ g(|x|,t),
$$
where $ p_{\rm rad}<\beta< 2^*_{s,\alpha}$ and $g:[0,\infty)\times\mathbb{R}\to\mathbb{R}$ is an odd Caratheodory function satisfying \eqref{g sub near 0}-\eqref{G sub near 0}. Our PDE then reads as
\begin{equation}\label{P sub near 0}
(-\Delta)^su+\left(I_\alpha*u^2\right)u=  |u|^{\beta-2}u+ g(|x|,u) \quad\textrm{in}\quad\mathbb{R}^N,
\end{equation}
whose solutions are critical points of the even functional
\begin{eqnarray}
   \Phi(u)&=&\frac{1}{2}\int_{\mathbb{R}^N}|(-\Delta)^\frac{s}{2}u|^2dx+\frac{C_\alpha}{4}\iint_{\mathbb{R}^N\times \mathbb{R}^N}\frac{u^2(x)u^2(y)}{|x-y|^{N-\alpha}}dxdy - \frac{1}{\beta}\int_{\mathbb{R}^N}|u|^{\beta}dx
   - \int_{\mathbb{R}^N}G(|x|,u)dx\nonumber\\
&=&I(u)-  \frac{1}{\beta}\int_{\mathbb{R}^N}|u|^{\beta}dx - \int_{\mathbb{R}^N}G(|x|,u)dx.
\label{Phi-sub near 0}
\end{eqnarray}
   
We seek solutions ot \eqref{P sub near 0} by means of an abstract result in critical point theory (see \cite[Proposition~3.36]{Perera-book}), which we recall for reader's convenience, with the following 
 
\begin{proposition}\label{critical point result}
    Let $\Phi$ be an even $C^1-$functional on a Banach space $W$ such that $\Phi(0) = 0$
and $\Phi$ satisfies the $(PS)_c$ condition for all $c<0$. Let $\mathcal{F}$ denote the class of symmetric subsets
of $W \backslash\{0\}$. For $k\geq1$, let
$$
\mathcal{F}_k = \{M \in \mathcal{F}: i(M) \geq k\}
$$
and set
\begin{equation}\label{ck sub near 0}
    c_k:=\inf_{M\in \mathcal{F}_k}\sup_{u\in M}\Phi(u).
\end{equation}
If there exists a $k_0 \geq1$ such that $-\infty<c_k <0$ for all $k \geq k_0$, then $c_{k_0}\leq c_{k_0+1}\leq\cdots  \to0$ is a sequence of critical values of $\Phi$.
\end{proposition}
We need this preliminary
 \begin{lemma}\label{Lemma K estimate}
Consider $ p_{\rm rad}<\beta< 2^*_{s,\alpha}$. Then, for each
$\lambda\geq\lambda_1$ there exists $c>0$ such that
$$
\frac{1}{\beta}\int_{\mathbb{R}^N}|u|^{\beta}dx\geq c,\quad\forall u\in\tilde{\Psi}^{\lambda}.
$$
 \end{lemma}
\begin{proof}
 We fix $\lambda\geq\lambda_1$. 
 For any $u\in\tilde{\Psi}^{\lambda}\subset\mathcal{M}$ we know that
 \begin{equation}\label{K estimate}
\int_{\mathbb{R}^N}|u|^{2^*_{s,\alpha}}dx= \frac{2^*_{s,\alpha}}{\tilde\Psi(u)}\geq \frac{2^*_{s,\alpha}}{\lambda}\quad\mbox{and}\quad \|(-\Delta)^su\|_2^2\leq 2I(u)=2.
\end{equation}
Recalling that  $\beta<2^*_{s,\alpha}<2^*_{s}$, let $\sigma\in(0,1)$ be such that $\frac{1}{2^*_{s,\alpha}}=\frac{\sigma}{\beta}+\frac{1-\sigma}{2^*_{s}}$. So, by the interpolation inequality, we have
$$
\|u\|_{L^{2^*_{s,\alpha}}}\leq \|u\|_{L^{2^*_{s}}}^{1-\sigma}\|u\|_{L^\beta}^\sigma.
$$
Using the embedding $\dot{H}^s\hookrightarrow L^{2^*_{s}} $ together with the second inequality in \eqref{K estimate} we obtain
$$
\|u\|_{L^{2^*_{s,\alpha}}}\leq C \|(-\Delta)^su\|_{L^2}^{1-\sigma}\|u\|_{L^\beta}^\sigma
\leq C 2^{(1-\sigma)/2}\|u\|_{L^\beta}^\sigma, \quad\forall u\in\tilde{\Psi}^{\lambda},
$$
which by \eqref{K estimate} is enough to conclude the proof.
\end{proof}

As $\Phi$ satisfies the Palais-Smale condition at negative levels, we are in the position to deal with the following\\

\begin{proof-th} \textbf{\textit{of Theorem \ref{th sub near 0}.}}\addcontentsline{toc}{subsection}{Proof of Theorem~\ref{th sub near 0}}
  As anticipated, our argument is based on Proposition \ref{critical point result}. Let us define $h:\mathbb{R}\to\mathbb{R}$ as $h(t)=|t|^{\beta-2}t+\tilde{\beta}\tilde{C}|t|^{\tilde\beta-2}t$ and its associated potential $\tilde{h}:E\to E^*$ as
  $$
  \tilde{h}(u)v=\int_{\mathbb{R}^N}h(u)vdx.
  $$   
  We denote by $H$ the primitive of $h$ vanishing at zero, and set
\begin{eqnarray*}
   \tilde{\Phi}(u)&=&\frac{1}{2}\int_{\mathbb{R}^N}|(-\Delta)^\frac{s}{2}u|^2dx+\frac{C_\alpha}{4}\iint_{\mathbb{R}^N\times \mathbb{R}^N}\frac{u^2(x)u^2(y)}{|x-y|^{N-\alpha}}dxdy -\int_{\mathbb{R}^N}H(u)dx.
\end{eqnarray*}
  Since $ p_{\rm rad}<\beta,\tilde{\beta}< 2^*_{s,\alpha}$, by Lemma \ref{o(t)}\eqref{o(t)(i)} we get
  $$
  \tilde{h}(u_t)v_t=o(t^{4s+\alpha-N})\|v\|\quad\mbox{as}\,\,t\to\infty,
  $$
uniformly in $u$ on bounded sets of $E$, for all $v\in E$. Then, Lemma 2.15 in \cite{Mercuri-Perera} ensures that 
$\tilde{\Phi}$ is a coercive functional on $E$. By \eqref{G sub near 0} we have ${\Phi}(u)\geq\tilde{\Phi}(u)$ for all $u\in E$ which implies that $\Phi$ is also coercive. The boundedness of ${\Phi}$ on bounded sets now implies that ${\Phi}$ is bounded from below in $E$. Hence, the levels $c_k$ in \eqref{ck sub near 0}satisfy
$$
c_k\geq -C, \quad\mbox{for all}\,\, k\geq1,
$$
for some positive constant $C$. Let us show that $c_k<0$ for all $k\geq 1$. By the definition of $\Phi$, see \eqref{Phi-sub near 0}, we have
\begin{eqnarray*}
   \Phi(u_t)&=&I(u_t) - \frac{1}{\beta}\int_{\mathbb{R}^N}|u_t|^{\beta}dx
   - \int_{\mathbb{R}^N}G(|x|,u_t)dx.
\end{eqnarray*}
For any $u\in\mathcal{M}$ we also note that $I(u_t)= t^{4s+\alpha-N}I(u)=t^{4s+\alpha-N}$ 
and
$$
\frac{1}{\beta}\int_{\mathbb{R}^N}|u_t|^{\beta}dx= \frac{t^{\theta\beta-N}}{\beta}\int_{\mathbb{R}^N}|u|^{\beta}dx.
$$
As far as the last term, by \eqref{g sub near 0} we get
\begin{eqnarray*}
|\tilde{G}(u_t)|&\leq&\int_{\mathbb{R}^N}|G(|x|,u_t)|dx\,\leq\, a_7\int_{\mathbb{R}^N}|u_t|^{q_7}dx+  a_8\int_{\mathbb{R}^N}|u_t|^{q_8}dx
\\
&\leq&a_7t^{\theta q_7-N}\|u\|_{L^{q_7}}^{q_7}+  a_8t^{\theta q_8-N}\|u\|_{L^{q_8}}^{q_8}
\leq \tilde{a}_7t^{\theta q_7-N}+  \tilde{a}_8t^{\theta q_8-N},
\end{eqnarray*}
for all $u\in\mathcal{M}$ and $t>0$. 
Since
\begin{equation}\label{same sign}
    \theta q-N>0  \quad\mbox{if}\quad 4s+\alpha>N\quad\mbox{and}\quad q>p_{\rm rad},
    \end{equation}
we have $\theta q_i-N>\theta\beta-N>0$, for $i=7,8$, which yields
$$
\frac{|\tilde{G}(u_t)|}{t^{\theta\beta-N}}\leq \tilde{a}_7t^{\theta(q_7-\beta)}+  \tilde{a}_8t^{\theta(q_8-\beta)}\,\,\to 0\quad\mbox{as}\,\,t\to 0^+.
$$
Thus, as $4s+\alpha-N>\theta\beta-N$, we obtain the asymptotics
\begin{eqnarray}\label{Phi <0 near origin}
   \Phi(u_t)&=&t^{4s+\alpha-N} - \frac{t^{\theta\beta-N}}{\beta}\int_{\mathbb{R}^N}|u|^{\beta}dx
   - t^{\theta\beta-N}o(1)= - t^{\theta\beta-N}\left(\frac{1}{\beta}\int_{\mathbb{R}^N}|u|^{\beta}dx+o(1)\right)
\end{eqnarray}
with $o(1)\to 0$ as $t\to 0^+$ uniformly in  $u\in\mathcal{M}$. Now, for $k\geq1$, in order to estimate the minimax level $c_k$ from above, let us set
$$
M_t=\{u_t : u\in\tilde{\Psi}^{\lambda_k}\},\quad\mbox{for}\quad t>0.
$$
We have $i(M_t)=i(\tilde{\Psi}^{\lambda_k})$ as $u\mapsto u_t$ with $u\in \tilde{\Psi}^{\lambda_k}$, is an odd homeomorphism.  Let $m\geq1$ be such that $\lambda_k=\cdots=\lambda_{k+m-1}<\lambda_{k+m}$.
By Theorem \ref{lambdak}\ref{lambdak(iii)} we have $i(\tilde{\Psi}^{\lambda_k})=k+m-1$. This implies that $i(M_t)=k+m-1\geq k,$ hence $M_t\in\mathcal{F}_k$. By \eqref{Phi <0 near origin} and Lemma \ref{Lemma K estimate} we obtain
\begin{eqnarray*}
   \Phi(u_t)\leq  - t^{\theta\beta-N}\left(c+o(1)\right) <0\quad\mbox{for all}\quad u\in \tilde{\Psi}^{\lambda_k}
\end{eqnarray*}
for $t>0$ sufficiently small, 
and therefore 
\begin{equation*}
    c_k=\inf_{M\in \mathcal{F}_k}\sup_{u\in M}\Phi(u)\leq \sup_{u\in M_t}\Phi(u)=\sup_{u\in \tilde{\Psi}^{\lambda_k}}\Phi(u_t)<0,
\end{equation*}
for all $k\geq1$. By Proposition \ref{critical point result}, this concludes the proof.
\end{proof-th}

In the particular case $g\equiv0,$ we have the following

\begin{proof-th} \textbf{\textit{of Corollary \ref{cor1 sub near 0}.}}
    As shown in the proof of Theorem \ref{th sub near 0}, the functional $\Phi$ associated with this problem is coercive. In particular, any $(PS)$ sequence for $\Phi$ is bounded. Since we are dealing with a subcritical nonlinearity, $\Phi$ satisfies the $(PS)_c$ condition for any $c\in\mathbb{R}$. Thus, Theorem \ref{th sub near 0} can be applied, and this is enough to conclude the proof.
\end{proof-th}

\begin{proof-th} \textbf{\textit{of Corollary \ref{cor2 sub near 0}.}}
  We recall that $g(t)=\lambda |t|^{2^*_{s,\alpha}-2}t-|t|^{q-2}t$, which is an odd Caratheodory function
 satisfying \eqref{g sub near 0}-\eqref{q7 q8 sub near 0}, and for any $p_{\rm rad}<\tilde\beta<2^*_{s,\alpha}$ we can find $\tilde{C}>0$ such that
$$
 G(t)=\frac{\lambda}{2^*_{s,\alpha}} |t|^{2^*_{s,\alpha}}-\frac{1}{q}|t|^{q}\leq \tilde{C} |t|^{\tilde\beta},\quad\forall t\in\mathbb{R},
 $$
namely \eqref{G sub near 0} also holds. Then, by the proof of Theorem \ref{th sub near 0}  the action functional 
\begin{eqnarray*}\label{Phi(un) sub near 0}
  \Phi(u)&=&\frac{1}{2}\int_{\mathbb{R}^N}|(-\Delta)^\frac{s}{2}u|^2dx+\frac{C_\alpha}{4}\iint_{\mathbb{R}^N\times \mathbb{R}^N}\frac{u^2(x)u^2(y)}{|x-y|^{N-\alpha}}dxdy - \frac{1}{\beta}\int_{\mathbb{R}^N}|u|^{\beta}dx\\
  && -\frac{\lambda}{2^*_{s,\alpha}}\int_{\mathbb{R}^N}|u|^{2^*_{s,\alpha}}dx + \frac{1}{q}\int_{\mathbb{R}^N}|u|^{q}dx,
\end{eqnarray*}
 is coercive. As a consequence any $(PS)$ sequence for $\Phi$ is bounded. Hence, for any $q <2^*_s$, $\Phi$ satisfies the $(PS)_c$ in correspondence of every energy level $c\in\mathbb{R}$. As far as the critical case $q=2^*_s$ is concerned, let us pick  a $(PS)_c$ sequence $(u_n)$ sequence for $\Phi$. Since $(u_n)$ is bounded, we may write $\Phi'(u_n)u_n=o(1)$, namely
 \begin{eqnarray}\label{Phi'(un)un sub near 0}
   o(1)&=&\int_{\mathbb{R}^N}|(-\Delta)^\frac{s}{2}u_n|^2dx+{C_\alpha}\iint_{\mathbb{R}^N\times \mathbb{R}^N}\frac{u_n^2(x)u_n^2(y)}{|x-y|^{N-\alpha}}dxdy - \int_{\mathbb{R}^N}|u_n|^{\beta}dx\nonumber\\
   &&-\lambda \int_{\mathbb{R}^N}|u_n|^{2^*_{s,\alpha}}dx+\int_{\mathbb{R}^N}|u_n|^{2^*_s}dx.
\end{eqnarray}
 Again by the boundedness of $(u_n)$ is bounded, we may assume, up to a subsequence, that $u_n\rightharpoonup u$ in $E$ and in $L^{2^*_s}(\mathbb{R}^N)$, $u_n\rightarrow u$ in $L^r(\mathbb{R}^N)$ for any $ p_{\rm rad}<r< 2^*_s$, and $u_n(x)\rightarrow u(x)$ almost everywhere. As in the proof of Lemma \ref{verify H7}, the weak convergence in $E$ ensures that
$$
  \mathcal{A}(u_n)v= \int_{\mathbb{R}^N}(-\Delta)^\frac{s}{2}u_n(-\Delta)^\frac{s}{2}vdx +C_\alpha\iint_{\mathbb{R}^N\times\mathbb{R}^N}\frac{u_n^2(x)u_n(y)v(y)}{|x-y|^{N-\alpha}}dxdy\to \mathcal{A}(u)v,
   $$
 for all $v\in E$, and for any $p_{\rm rad}<r\leq 2^*_s$,  
 $$
 \int_{\mathbb{R}^N}|u_n|^{r-2}u_nvdx\to \int_{\mathbb{R}^N}|u|^{r-2}uvdx.
 $$
Thus, $\Phi'(u)v=\lim_{n\to\infty}\Phi'(u_n)v=0$  for all $v\in E$. Testing in particular with $v=u,$ we get
\begin{eqnarray}\label{Phi'(u)u sub near 0}
   0=\int_{\mathbb{R}^N}|(-\Delta)^\frac{s}{2}u|^2dx+{C_\alpha}\iint_{\mathbb{R}^N\times \mathbb{R}^N}\frac{u^2(x)u^2(y)}{|x-y|^{N-\alpha}}dxdy - \int_{\mathbb{R}^N}|u|^{\beta}dx -\lambda\int_{\mathbb{R}^N}|u|^{2^*_{s,\alpha}}dx
   +\int_{\mathbb{R}^N}|u|^{2^*_s}dx.
\end{eqnarray}
Finally, subtracting \eqref{Phi'(u)u sub near 0} by \eqref{Phi'(un)un sub near 0}, using the classical Brezis-Lieb Lemma and its nonlocal variant (see \cite[Proposition 4.1]{Mercuri-Moroz-VS-2016}), we obtain
 \begin{eqnarray*}
   o(1)\geq\int_{\mathbb{R}^N}|(-\Delta)^\frac{s}{2}(u_n-u)|^2dx+{C_\alpha}\iint_{\mathbb{R}^N\times \mathbb{R}^N}\frac{(u_n-u)^2(x)(u_n-u)^2(y)}{|x-y|^{N-\alpha}}dxdy 
   +\int_{\mathbb{R}^N}|u_n-u|^{2^*_s}dx,
\end{eqnarray*}
which gives us the strong convergence $u_n\to u$ in $E,$ namely $\Phi$ satisfies the $(PS)$ condition for $q=2^*_s.$ By Theorem \ref{th sub near 0} this concludes the proof.
\end{proof-th}

In the next proof we apply Proposition \ref{critical point result} in a context where $\Phi$ is not bounded from below.

 \begin{proof-th} \textbf{\textit{of Theorem \ref{subscaled-NE}.}}\addcontentsline{toc}{subsection}{Proof of Theorem~\ref{subscaled-NE}}
 We recall that our action functional is in this case
 \begin{eqnarray*}\label{Phi sub near 0 NE}
  \Phi(u)&=&\frac{1}{2}\int_{\mathbb{R}^N}|(-\Delta)^\frac{s}{2}u|^2dx+\frac{C_\alpha}{4}\iint_{\mathbb{R}^N\times \mathbb{R}^N}\frac{u^2(x)u^2(y)}{|x-y|^{N-\alpha}}dxdy - \frac{1}{\beta}\int_{\mathbb{R}^N}|u|^{\beta}dx-\frac{\lambda}{2^*_{s,\alpha}}\int_{\mathbb{R}^N}|u|^{2^*_{s,\alpha}}dx\\
  &=&\Phi_\lambda(u)- \frac{1}{\beta}\int_{\mathbb{R}^N}|u|^{\beta}dx.
\end{eqnarray*}
For $p_{\rm rad}<\beta< 2^*_{s,\alpha}$ we define $h:\mathbb{R}\to\mathbb{R}$ as $h(t)=|t|^{\beta-2}t$ and by
  $$
  \tilde{h}(u)v=\int_{\mathbb{R}^N}h(u)vdx
  $$   
its associated potential. By Lemma \ref{o(t)}\eqref{o(t)(iii)} we get
  $$
  \tilde{h}(u_t)v_t=o(t^{4s+\alpha-N})\|v\|\quad\mbox{as}\,\,t\to\infty,
  $$
uniformly in $u$ on bounded sets of $E$, for all $v\in E$. We observe that Lemma 2.26 of \cite{Mercuri-Perera} ensures that 
$\Phi$ satisfies the $(PS)$ condition when $\lambda$ is not an eigenvalue of \eqref{NE}. 
Let us consider $c_k$ as in \eqref{ck sub near 0}. Note that, as $\lambda\notin \sigma(\mathcal{A},\mathcal{B})$, we may have $\lambda<\lambda_1,$, as well as $\lambda_{k_0-1}<\lambda<\lambda_{k_0}$ for some $k_0\geq2$.
Let us estimate $c_k$ for $k\geq k_0\geq 1$ and show that $c_k<0$. 
By \eqref{same sign}, for all $u\in\mathcal{M}$ we have
\begin{eqnarray}\label{Phi <0 near origin 2}
   \Phi(u_t)&=&t^{4s+\alpha-N}\left(1-\frac{\lambda}{\tilde{\Psi}(u)}\right) - \frac{t^{\theta\beta-N}}{\beta}\int_{\mathbb{R}^N}|u|^{\beta}dx
   = - t^{\theta\beta-N}\left(\frac{1}{\beta}\int_{\mathbb{R}^N}|u|^{\beta}dx+o(1)\right)
\end{eqnarray}
with $o(1)\to 0$ as $t\to 0^+$ uniformly in  $u\in\mathcal{M}$. As in the proof of Theorem \ref{th sub near 0},  setting
$
M_t=\{u_t : u\in\tilde{\Psi}^{\lambda_k}\}
$
we have $M_t\in\mathcal{F}_k$ and so, by \eqref{Phi <0 near origin 2} and Lemma \ref{Lemma K estimate}, we get
\begin{equation*}
    c_k=\inf_{M\in \mathcal{F}_k}\sup_{u\in M}\Phi(u)\leq \sup_{u\in M_t}\Phi(u)=\sup_{u\in \tilde{\Psi}^{\lambda_k}}\Phi(u_t)<0,
\end{equation*}
for $t>0$ sufficiently small.  On the other hand we may estimate $c_k$ from below, as follows. By the definition of $\mathcal{F}_k$ in Proposition \ref{critical point result},  for each $M\in\mathcal{F}_k$ we have
\begin{equation}\label{M in Fk}
    i(M)\geq k\geq k_0,
\end{equation}
for all $k\geq k_0\geq 1$. Setting $Y=\{u_t : u\in\tilde{\Psi}_{\lambda_{k_0}},\,t\geq0\},$ we claim that 
\begin{equation}\label{M Y}
    M\cap Y\neq\emptyset \quad\mbox{for all}\,\, M\in\mathcal{F}_k\,\,\mbox{and}\,\,k\geq k_0.
\end{equation}
In fact, if this were not true, we would have $M\cap Y=\emptyset$ for some $M\in\mathcal{F}_k$ and $k\geq k_0$. So $\pi_{|_M}:M\to \mathcal{M}\backslash \tilde{\Psi}_{\lambda_{k_0}}$
would be an odd continuous map, and hence we would have
$
i(M)\leq i(\mathcal{M}\backslash \tilde{\Psi}_{\lambda_{k_0}}).
$
On the other hand, if $\lambda<\lambda_1$, by Theorem \ref{lambdak}\ref{lambdak(i)} we have $\tilde{\Psi}_{\lambda_{1}}=\mathcal{M}$ so that
$
i(\mathcal{M}\backslash \tilde{\Psi}_{\lambda_{1}})=i(\emptyset)=0,
$
whereas if $\lambda_{k_0-1}<\lambda<\lambda_{k_0}$ for some $k_0\geq2$, the same Theorem \ref{lambdak}\ref{lambdak(iii)} yields
$
i(\mathcal{M}\backslash \tilde{\Psi}_{\lambda_{k_0}})=k_0-1.
$
In both cases, we would conclude that
$$
i(M)\leq i(\mathcal{M}\backslash \tilde{\Psi}_{\lambda_{k_0}})\leq k_0-1,
$$
which is a contradiction by  \eqref{M in Fk}. Thus, \eqref{M Y} holds. Therefore, for each ${M\in \mathcal{F}_k}$
$$
\sup_{u\in M}\Phi(u)\geq \sup_{u\in M\cap Y}\Phi(u)\geq\inf_{u\in M\cap Y}\Phi(u)
\geq \inf_{u\in  Y}\Phi(u). 
$$
Since $\tilde{\Psi}(u)\geq\lambda_{k_0}>\lambda$ for all $u\in\tilde{\Psi}_{\lambda_{k_0}}$ and $\mathcal{M}$ is bounded in $L^\beta(\mathbb{R}^N)$ and in $L^{2^*_{s,\alpha}}(\mathbb{R}^N)$, \eqref{Phi <0 near origin 2} yields $\inf_{u\in  Y}\Phi(u):=-C>-\infty$. 
Hence
$$ 
c_k=\inf_{M\in \mathcal{F}_k}\sup_{u\in M}\Phi(u)\geq \inf_{u\in  Y}\Phi(u)\geq -C,\quad\forall k\geq k_0,
$$
and by Proposition \ref{critical point result}  we may conclude the proof.
 \end{proof-th}

\begin{proof-th} \textbf{\textit{of Theorem \ref{th sub near 0 critical}.}}\addcontentsline{toc}{subsection}{Proof of Theorem~\ref{th sub near 0 critical}}
The proof of this theorem is also based on Proposition \ref{critical point result}. 
As in the proof of Theorem 1.37 in \cite{Mercuri-Perera}, we are going to use a truncation of the functional $\Phi_\mu$ associated with this problem, which is 
 \begin{eqnarray*}
  \Phi_\mu(u)=\frac{1}{2}\int_{\mathbb{R}^N}|(-\Delta)^\frac{s}{2}u|^2dx+\frac{C_\alpha}{4}\iint_{\mathbb{R}^N\times \mathbb{R}^N}\frac{u^2(x)u^2(y)}{|x-y|^{N-\alpha}}dxdy - \frac{\mu}{\beta}\int_{\mathbb{R}^N}|u|^{\beta}dx-\frac{1}{2^*_{s}}\int_{\mathbb{R}^N}|u|^{2^*_{s}}dx.
\end{eqnarray*}
Arguing as in \cite[Lemma 3.1]{Ianni-Ruiz-2012} (see also \cite[Lemma 2.1]{Feng-Su-2024}) we verify that for any $p\in(p_{\rm rad}, 2^*_{s}]$ there exists $c>0$ such that
\begin{equation}\label{Lp by I(u)}
\int_{\mathbb{R}^N}|u|^pdx\leq c[I(u)]^\frac{p\theta-N}{4s+\alpha-N},\quad\forall u\in E.
\end{equation}
Thus, denoting $\gamma_\mu(t)=t- \tilde{c}_1\mu t^\frac{\beta\theta-N}{4s+\alpha-N}-\tilde{c}_2t^\frac{2^*_s\theta-N}{4s+\alpha-N}$, for $t\geq0$, we have
\begin{eqnarray*}
  \Phi_\mu(u)&\geq& I(u)- \tilde{c}_1\mu[I(u)]^\frac{\beta\theta-N}{4s+\alpha-N}-\tilde{c}_2[I(u)]^\frac{2^*_s\theta-N}{4s+\alpha-N}=\gamma_\mu(I(u)),\quad\forall u\in E.
\end{eqnarray*}
 Recalling that $2^*_s>2^*_{s,\alpha}>\beta$ implies that $\frac{\beta\theta-N}{4s+\alpha-N}\in(0,1)$ and $\frac{2^*_s\theta-N}{4s+\alpha-N}>1$, we note that is possible to find $\mu^*>0$ such that for each $\mu\in(0,\mu^*)$ there exist $R_1(\mu),R_2(\mu)>0$ such that
\[
\gamma_\mu(t)<0,\,\,\forall t\in[0,R_1(\mu))\cup(R_2(\mu),\infty)\qquad\mbox{and}\qquad \gamma_\mu(t)\geq0,\,\,\forall t\in[R_1(\mu)),(R_2(\mu)].
\]
Observe that if $\mu\in(0,\mu^*)$ it holds that $\gamma_\mu(t)\geq \gamma_{\mu^*}(t)$ for all $t\in[R_1(\mu^*)),(R_2(\mu^*)],$. hence, $R_1(\mu)\leq R_1(\mu^*)$ and $R_2(\mu^*)\leq R_2(\mu)$.
Pick now $\xi_\mu:[0,\infty)\to[0,1]$ smooth and satisfying 
$$
\xi_\mu\equiv1\,\,\mbox{in}\,\, [0,R_1(\mu)]\quad\mbox{and}\quad \xi_\mu\equiv0\,\,\mbox{in}\,\,[R_2(\mu),\infty),
$$
and consider the truncated functional $\tilde{\Phi}_\mu:E\to \mathbb{R}$ defined as
$$
\tilde{\Phi}_\mu(u)=\xi_\mu(I(u)){\Phi}_\mu(u).
$$
Note that if $\tilde{\Phi}_\mu(u)<0$ then $I(u)\in(0,R_1(\mu)),$ and, by the continuity of $I$, $I(v)\in(0,R_1(\mu))$ for $v$ in some neighborhood of $u$ in $E$. So
\begin{equation}\label{truncation<0}
    \tilde{\Phi}_\mu(u)={\Phi}_\mu(u)<0\quad\mbox{and}\quad \tilde{\Phi}'_\mu(u)={\Phi}'_\mu(u),
\end{equation}
hence critical points for $\tilde{\Phi}_\mu$ at negative levels are actual critical points for ${\Phi}_\mu$ corresponding to the same energy levels.
With this observation and Proposition \ref{critical point result} at hand, we focus on finding a sequence of critical points for $\tilde{\Phi}_\mu$ whose energy is negative. To this aim we start verifying that  $\tilde{\Phi}_\mu$ satisfies the $(PS)_c$ condition for $c<0$, provided $\mu\in(0,\mu^*)$ for some $\mu^*$ sufficiently small. Let us consider then a sequence $(u_n)$ in $E$ with the property 
$$
\tilde{\Phi}_\mu(u_n)\to c<0\quad{and}\quad \tilde{\Phi}'_\mu(u_n)\to 0.
$$
Since $\tilde{\Phi}_\mu(u_n)<0$ for large $n$, we have $I(u_n)\in(0,R_1(\mu^*))$, which forces $(u_n)$ to be uniformly bounded for $\mu\in(0,\mu^*)$. Passing if necessary to a subsequence, we may assume $u_n\rightharpoonup u$ in $E$.  Since  $I(u_n)\in(0,R_1(\mu^*))$   we get the uniform estimate $I(u)\leq R_1(\mu^*)$, for $\mu\in(0,\mu^*),$ which by \eqref{Lp by I(u)} implies, in turn, the uniform boundedness of $\|u\|_{L^\beta}$.
Note that \eqref{truncation<0} also implies that $(u_n)$ is a $(PS)_c$ sequence for ${\Phi}_\mu$. 
Arguing as in Lemma \ref{PSlocal}, if there is no subsequence of $(u_n)$ converging to $u$ in $E$, we see that
\begin{eqnarray*}
 c & \geq&\frac{s}{N}\mathbb{S}^\frac{N}{2s}
   -\mu \tilde{d}\int_{\mathbb{R}^N}|u|^{\beta}dx
  +\frac{s}{N}\int_{\mathbb{R}^N}|u|^{2^*_{s}}dx
\end{eqnarray*}
for some positive constant $\tilde{d}$, which does not depend neither on $u$ nor on $\mu$ (recall \eqref{PS level estimate} and that here $p_{\rm rad}<\beta<2^*_{s,\alpha}$).

Hence, the uniform boundedness of $\|u\|_{L^\beta}$  and the fact that $c<0$ gives us
\begin{equation*}
 \mu d'\geq \mu \tilde{d}\int_{\mathbb{R}^N}|u|^{\beta}dx \geq\frac{s}{N}\mathbb{S}^\frac{N}{2s}
\end{equation*}
for some $d'>0,$ which is violated for some possibly smaller $\mu^*>0$ and $\mu\in(0,\mu^*)$. We may finally conclude that $\tilde{\Phi}_\mu$ satisfies the $(PS)_c$ condition for any $c<0$ and $\mu\in(0,\mu^*)$. \\In order to apply Proposition \ref{critical point result}, let us pick $\mu\in(0,\mu^*)$ and estimate the levels $c_k$ defined in \eqref{ck sub near 0}. Since $\tilde{\Phi}_\mu$ is bounded, we have $c_k>-\infty$ for all $k\in\mathbb{N}$. On the other hand, 
recalling that $2^*_{s}\theta-N>4s+\alpha-N>\theta\beta-N>0$, for all $u\in\mathcal{M}$ we get the asymptotics
\begin{eqnarray*}\label{Phi <0 near origin critical}
   \Phi_\mu(u_t)&=&t^{4s+\alpha-N}- \frac{\mu t^{\theta\beta-N}}{\beta}\int_{\mathbb{R}^N}|u|^{\beta}dx
   -\frac{t^{\theta 2^*_{s}-N}}{2^*_{s}}\int_{\mathbb{R}^N}|u|^{2^*_{s}}dx
   =-t^{\theta\beta-N}\left(\frac{\mu}{\beta}\int_{\mathbb{R}^N}|u|^{\beta}dx+o(1)\right)
\end{eqnarray*}
with $o(1)\to 0$ as $t\to 0^+$ uniformly in  $u\in\mathcal{M}$, as well as  the bound $I(u_t)<R_1(\mu)$ for all $u\in\mathcal{M}$, if $t>0$ is small.  Arguing as in the proof of Theorem \ref{th sub near 0},  setting 
$
M_t=\{u_t : u\in\tilde{\Psi}^{\lambda_k}\}
$
we note that $M_t\in\mathcal{F}_k,$ and so, by Lemma \ref{Lemma K estimate}, we obtain $\sup_{u\in \tilde{\Psi}^{\lambda_k}}\Phi(u_t)<0$ for $t>0$ sufficiently small,
hence $\tilde\Phi_\mu(u)=\Phi_\mu(u)$ for all $u\in M_t$ and
\begin{equation*}
    c_k=\inf_{M\in \mathcal{F}_k}\sup_{u\in M}\tilde\Phi_\mu(u)\leq \sup_{u\in M_t}\tilde\Phi_\mu(u)=\sup_{u\in \tilde{\Psi}^{\lambda_k}}\Phi_\mu(u_t)<0,
\end{equation*}
and by Proposition \ref{critical point result}, this concludes the proof. 
\end{proof-th}

The next proof, based on Corollary \ref{Corollary 2.34 MP}, with $c^*=\infty,$ deals with an asymptotically scaled problem at infinity, involving a negative term which is subscaled at the origin.

\begin{proof-th} \textbf{\textit{of Theorem \ref{sub near 0, asympt, subcritical}.}}\addcontentsline{toc}{subsection}{Proof of Theorem~\ref{sub near 0, asympt, subcritical}}
 Since $\lambda\notin\sigma(\mathcal{A},\mathcal{B})$, by \cite[Lemma 2.26]{Mercuri-Perera} it follows that the action functional
 \begin{eqnarray*}
  \Phi(u)=\frac{1}{2}\int_{\mathbb{R}^N}|(-\Delta)^\frac{s}{2}u|^2dx+\frac{C_\alpha}{4}\iint_{\mathbb{R}^N\times \mathbb{R}^N}\frac{u^2(x)u^2(y)}{|x-y|^{N-\alpha}}dxdy + \frac{\mu}{\beta}\int_{\mathbb{R}^N}|u|^{\beta}dx-\frac{\lambda}{2^*_{s,\alpha}}\int_{\mathbb{R}^N}|u|^{2^*_{s,\alpha}}dx,
\end{eqnarray*}
satisfies the $(PS)$ condition. 
For any $\lambda\in(\lambda_k,\infty)\backslash\sigma(\mathcal{A},\mathcal{B})$ we may find $m\in\mathbb{N}$ such that $\lambda_{k+m-1}<\lambda<\lambda_{k+m}$. Hence, by Theorem \ref{lambdak}\ref{lambdak(iii)} 
we have
$$
i(\mathcal{M}\backslash\tilde\Psi_{\lambda})=k+m-1\geq k.
$$
Pick $A_0\subset \mathcal{M}\backslash\tilde\Psi_{\lambda}$ a compact symmetric subset of index $k+m-1$ (see the proof of Proposition 3.1 in \cite{De-Lan-2007}). For some $R>\rho>0$ to be chosen later, we set
$$
A=\{u_R:u\in A_0\}\quad\mbox{and}\quad X=\{u_t:u\in A_0,\,t\geq0\}.
$$
For all $u\in\mathcal{M}$ we have
\begin{eqnarray}\label{Phi asympt  subcritical}
   \Phi(u_t)=t^{4s+\alpha-N}\left(1-\frac{\lambda}{\tilde\Psi(u)}\right)+ \frac{t^{\theta\beta-N}}{\beta}\int_{\mathbb{R}^N}|u|^{\beta}dx.
\end{eqnarray}
Hence, by Theorem \ref{lambdak}\ref{lambdak(i)} and Lemma \ref{Lemma K estimate} we obtain
\begin{eqnarray*}
   \Phi(u_t)\geq t^{4s+\alpha-N}\left(1-\frac{\lambda}{\lambda_{1}}\right)+ ct^{\theta\beta-N}\quad \forall u\in\tilde\Psi^{\lambda_{k+m}}, \forall t\geq0,
   \end{eqnarray*}
for some $c>0$. Moreover, note that
\begin{eqnarray*}
   \Phi(u_t)\geq t^{4s+\alpha-N}\left(1-\frac{\lambda}{\lambda_{k+m}}\right)\quad \forall u\in\mathcal{M}\backslash \tilde\Psi^{\lambda_{k+m}}, \forall t\geq0.
   \end{eqnarray*}
Since $4s+\alpha-N>\theta\beta-N>0$ we find that
$$
\inf_{u\in\mathcal{M}}\Phi(u_\rho)>0
$$
provided $\rho>0$ is small enough. 
On the other hand, as $A_0\subset \mathcal{M}\backslash\tilde\Psi_{\lambda}$ is compact, there exists $\tilde{c}<{\lambda}$ such that $\tilde\Psi(u)\leq \tilde{c}$ for any $u\in A_0$. This yields, by the boundedness of $\|u\|_{L^\beta}$ for $u\in\mathcal{M}$, and \eqref{Phi asympt  subcritical}, the estimate
\begin{eqnarray*}
   \Phi(u_t)\leq t^{|4s+\alpha-N|}\left(1-\frac{\lambda}{\tilde{c}}\right)+ c't^{|\theta\beta-N|}\quad \forall u\in A_0, \forall t\geq0.
   \end{eqnarray*}
Finally, picking $R>\rho$ sufficiently large, we obtain
$$
\sup_{u\in A}\Phi(u)=\sup_{u\in A_0}\Phi(u_R)\leq0
$$
and
$$
\sup_{u\in X}\Phi(u)=\sup_{u\in A_0,t\in[0,R]}\Phi(u_t)<\infty,
$$
which by \ref{Corollary 2.34 MP} is enough to ensure the existence of $k+m-1(\geq k)$ pairs of critical points for $\Phi$ at positive energy levels. This concludes the proof.
\end{proof-th}

We now deal with a proof, again based on Corollary \ref{Corollary 2.34 MP}, for a critical growth problem.

\begin{proof-th} \textbf{\textit{of Theorem \ref{sub near 0, asympt, critical}.}}\addcontentsline{toc}{subsection}{Proof of Theorem~\ref{sub near 0, asympt, critical}}
Recall that the action functional is
 \begin{eqnarray*}
  \Phi(u)&=&\frac{1}{2}\int_{\mathbb{R}^N}|(-\Delta)^\frac{s}{2}u|^2dx+\frac{C_\alpha}{4}\iint_{\mathbb{R}^N\times \mathbb{R}^N}\frac{u^2(x)u^2(y)}{|x-y|^{N-\alpha}}dxdy-\frac{\lambda}{2^*_{s,\alpha}}\int_{\mathbb{R}^N}|u|^{2^*_{s,\alpha}}dx\\
  &&+ \frac{\mu}{\beta}\int_{\mathbb{R}^N}|u|^{\beta}dx-\frac{1}{2^*_{s}}\int_{\mathbb{R}^N}|u|^{2^*_{s}}dx\\
  &=& \Phi_\lambda(u)+ \frac{\mu}{\beta}\int_{\mathbb{R}^N}|u|^{\beta}dx-\frac{1}{2^*_{s}}\int_{\mathbb{R}^N}|u|^{2^*_{s}}dx.
\end{eqnarray*}
As in the preceding proof, we first show that $\Phi$ satisfies the $(PS)_c$ condition for all $c\in(0,s\mathbb{S}^\frac{N}{2s}/N)$ and $\mu>0$. Arguing as in Lemma \ref{PSlocal}, let us consider a  $(PS)_c$ sequence $(u_n)$, and assume it  is unbounded in $E.$ Setting $\tilde{u}_n=(u_n)_{t_{u_n}}\in\mathcal{M}$, and $\tilde{t}_n=t_{u_n}^{-1}=[I(u_n)]^\frac{1}{4s+\alpha-N}\to\infty$, by the fact that $\Phi(u_n)\to c,$ we get (see equation \eqref{eqPhiLn3})
\begin{eqnarray}\label{Phi(un) critical sub near 0}
(\tilde{t}_n)^{4s+\alpha-N}\Phi_\lambda(\tilde u_n)&
=&-\frac{\mu(\tilde{t}_n)^{\theta \beta-N}}{\beta}\int_{\mathbb{R}^N}|\tilde u_n|^{\beta}dx
+ \frac{(\tilde{t}_n)^{\theta 2^*_{s}-N}}{2^*_{s}}\int_{\mathbb{R}^N}|\tilde u_n|^{2^*_{s}}dx+c+o(1).
\end{eqnarray}
Note that, if $\beta\in(p_{\rm rad}, 2^*_{s,\alpha})$, it holds that
$
\theta2^*_{s} -N>4s+\alpha-N>\theta\beta -N>0. 
$
Hence, by the boundedness of $(\Phi_\lambda(\tilde{u}_n))$ and $(\|\tilde{u}_n\|_{L^\beta})$ we get
$$
\frac{(\tilde{t}_n)^{\theta (2^*_{s}-2^*_{s,\alpha})}}{2^*_{s}}\int_{\mathbb{R}^N}|\tilde u_n|^{2^*_{s}}dx
=\Phi_\lambda(\tilde{u}_n)+\frac{\mu}{\beta(\tilde{t}_n)^{\theta( 2^*_{s,\alpha}-\beta)}}\int_{\mathbb{R}^N}|\tilde u_n|^{\beta}dx+o(1)\leq c_1,
$$
for some $c_1>0$. As a consequence, $\tilde{u}_n\to0$ in $L^{2^*_{s}}(\mathbb{R}^N),$ and by the boundedness in $L^{\beta}(\mathbb{R}^N)$ and interpolation, we get $\tilde{u}_n\to0$ in $L^{2^*_{s,\alpha}}(\mathbb{R}^N)$.
Arguing as for \eqref{eqPhi'Lncri}, we obtain
\begin{eqnarray}\label{eqPhi'Lncri sub}
(\tilde{t}_n)^{(4s+\alpha-N)}\Phi'_\lambda(\tilde u_n)\tilde{u}_n=
-\mu(\tilde{t}_n)^{\theta \beta-N}\int_{\mathbb{R}^N}|\tilde u_n|^{\beta}dx+(\tilde{t}_n)^{\theta 2^*_s-N}\int_{\mathbb{R}^N}|\tilde u_n|^{2^*_s}dx+o((\tilde{t}_n)^{(4s+\alpha-N)/2}).
\end{eqnarray}
Multiplying now \eqref{Phi(un) critical sub near 0} by $2^*_s$, subtracting by \eqref{eqPhi'Lncri sub} and dividing by $(\tilde{t}_n)^{(4s+\alpha-N)}$ we finally get
\begin{eqnarray*}
  & &\left(\frac{2^*_s}{2}-1\right)\int_{\mathbb{R}^N}|(-\Delta)^\frac{s}{2}\tilde u_n|^2dx+C_\alpha\left(\frac{2^*_s}{4}-1\right) \iint_{\mathbb{R}^N\times\mathbb{R}^N}\frac{\tilde u_n^2(x)\tilde u_n^2(y)}{|x-y|^{N-\alpha}}dxdy \\
  &&\qquad\qquad=\lambda\left(\frac{2^*_s}{2^*_{s,\alpha}} -1\right)\int_{\mathbb{R}^N}|\tilde u_n|^{2^*_{s,\alpha}}dx
  -\left(\frac{2^*_s}{\beta} -1\right)\frac{\mu}{(\tilde{t}_n)^{4s+\alpha-\theta \beta}}\int_{\mathbb{R}^N}|\tilde u_n|^{\beta}dx+o(1)\\
  && \qquad\qquad=\,\, o(1),\quad\mbox{as}\,\,n\to\infty,
\end{eqnarray*}
and hence, as $2^*_s>4$ for $s>N/4,$ that $\tilde u_n\to0,$ a contradiction with the fact that $\tilde u_n\in\mathcal{M}.$
We may conclude then that $(u_n)$ is bounded in $E$. Passing if necessary to a subsequence, we can assume that $u_n\rightharpoonup u$ in $E$. If $(u_n)$ does not have a convergent subsequence, 
as in the proof of Lemma \ref{PSlocal}, we  get that
\begin{eqnarray*}
 c & \geq&\frac{s}{N}\mathbb{S}^\frac{N}{2s}
   +\mu d\int_{\mathbb{R}^N}|u|^{\beta}dx
  +\frac{s}{N}\int_{\mathbb{R}^N}|u|^{2^*_{s}}dx\geq\frac{s}{N}\mathbb{S}^\frac{N}{2s}.
\end{eqnarray*}
for some positive constant $d$ (see \eqref{PS level estimate}). Hence, if $c<\frac{s}{N}\mathbb{S}^\frac{N}{2s}$ any $(PS)_c$ sequence for $\Phi$ in $E$ possesses a convergent subsequence, namely $\Phi$ satisfies the $(PS)_c$ condition for any $c<s\mathbb{S}^\frac{N}{2s}/N$, for $\mu>0$.

Arguing now as in Theorem \ref{sub near 0, asympt, subcritical}, we shall apply Corollary \ref{Corollary 2.34 MP} with $c^*=s\mathbb{S}^\frac{N}{2s}/N$.

For any $\lambda\in(\lambda_k,\infty)$ we may find $m\in\mathbb{N}$ such that $\lambda_{k+m-1}<\lambda\leq\lambda_{k+m}$. We fix $\tilde{\lambda}\in(\lambda_{k+m-1},\lambda)$ and  $\lambda'\in(\lambda,\infty)$. By Theorem \ref{lambdak}\ref{lambdak(iii)} we have
$$
i(\mathcal{M}\backslash\tilde\Psi_{\tilde\lambda})=k+m-1\geq k.
$$
Pick $A_0\subset \mathcal{M}\backslash\tilde\Psi_{\tilde\lambda}$ a compact symmetric subset of index $k+m-1$ (see the proof of Proposition 3.1 in \cite{De-Lan-2007}). For any $R>\rho>0$ to be chosen later, we set
$$
A=\{u_R:u\in A_0\}\quad\mbox{and}\quad X=\{u_t:u\in A_0,\,t\geq0\}.
$$
For all $u\in\mathcal{M}$ we have
\begin{eqnarray}\label{Phi sub-superscaled critical}
   \Phi(u_t)=t^{(4s+\alpha-N)}\left(1-\frac{\lambda}{\tilde\Psi(u)}\right)+ \frac{\mu t^{(\theta\beta-N)}}{\beta}\int_{\mathbb{R}^N}|u|^{\beta}dx-\frac{t^{(\theta 2^*_{s}-N)}}{2^*_{s}}\int_{\mathbb{R}^N}|u|^{2^*_{s}}dx.
\end{eqnarray}
Note that as $\lambda'>\lambda>\lambda_1$, the boundedness of $\|u\|_{L^{2^*_s}}$ for $u\in\mathcal{M}$, by Theorem \ref{lambdak}\ref{lambdak(i)} and Lemma \ref{Lemma K estimate} give us the estimate
\begin{eqnarray*}
   \Phi(u_t)\geq -t^{(4s+\alpha-N)}\left(\frac{\lambda}{\lambda_{1}}-1\right)+ c_1\mu t^{(\theta\beta-N)}-c_2t^{(\theta 2^*_s-N)}\quad \forall u\in\tilde\Psi^{\lambda'}, \forall t\geq0,
   \end{eqnarray*}
for some $c_1,c_2>0$ which do not depend on $\mu$. Also observe that 
\begin{eqnarray*}
   \Phi(u_t)\geq t^{(4s+\alpha-N)}\left(1-\frac{\lambda}{\lambda'}\right)-c_2t^{(\theta 2^*_s-N)}\quad \forall u\in\mathcal{M}\backslash \tilde\Psi^{\lambda'}, \forall t\geq0.
   \end{eqnarray*}
Recalling that $\theta 2^*_{s}-N>4s+\alpha-N>\theta\beta-N$ we conclude that
$$
\inf_{u\in\mathcal{M}}\Phi(u_\rho)>0
$$
provided $\rho>0$ is small enough. 
On the other hand, since $A_0\subset \mathcal{M}\backslash\tilde\Psi_{\tilde\lambda}$  we have $\tilde\Psi(u)< \tilde{\lambda}$ for any $u\in A_0$. Moreover, $\inf_{u\in A_0}\|u\|_{L^{2^*_s}}>0$, as $A_0$ is compact.
Hence, using the boundedness of $\|u\|_{L^\beta}$ for $u\in\mathcal{M}$, by \eqref{Phi sub-superscaled critical} we get
\begin{eqnarray*}
   \Phi(u_t)\leq -t^{(4s+\alpha-N)}\left(\frac{\lambda}{\tilde{\lambda}}-1\right)+ c'\mu t^{(\theta\beta-N)}-c_3t^{(\theta 2^*_s-N)}\quad \forall u\in A_0, \forall t\geq0,
   \end{eqnarray*}
for some $c',c_3>0$. Picking now $R>\rho$ sufficiently large, we obtain
$$
\sup_{u\in A}\Phi(u)=\sup_{u\in A_0}\Phi(u_R)\leq0.
$$
Note also that, for $u\in A_0$ and $t\geq0$, we have
\begin{eqnarray*}
   \Phi(u_t)\leq c'\mu t^{(\theta\beta-N)}-c_3t^{(\theta 2^*_s-N)}\leq \theta(2^*_s-\beta)\left[\frac{\theta\beta-N}{c_3}\right]^\frac{\theta\beta-N}{\theta(2^*_s-\beta)}
   \left[\frac{c'\mu}{\theta2^*_s-N}\right]^\frac{\theta2^*_s-N}{\theta(2^*_s-\beta)}.
   \end{eqnarray*}
As a consequence, for $\mu\in(0,\mu^*)$ with $\mu^*$ sufficiently small, it holds that
$$
\sup_{u\in X}\Phi(u)=\sup_{u\in A_0,t\in[0,R]}\Phi(u_t)<\frac{s}{N}\mathbb{S}^\frac{N}{2s},
$$
and therefore, by Corollary \ref{Corollary 2.34 MP} there exist $k+m-1(\geq k)$ pairs of critical points for $\Phi$, at positive energy levels. This concludes the proof.
\end{proof-th}

\appendix

\section{Appendix: Density of $C^\infty(\mathbb R^N)$ in the fractional Coulomb-Sobolev spaces}

Our proof is based on the following

\begin{lemma}\label{suaves}
    Let $s>0$, $q\geq 1$, $0<\alpha<N$. Then $L^\infty(\mathbb R^N)\cap C^\infty(\mathbb R^N)\cap \mathcal E^{s,\alpha,q}(\mathbb R^N)$ is dense in $\mathcal E^{s,\alpha,q}(\mathbb R^N)$.
\end{lemma}
\begin{proof} Pick $\eta\in C^\infty_c(\mathbb{R}^N)$ such that $0\leq\eta(x)\leq1$, $supp(\eta)\subset B_1(0)$ and $\int_{\mathbb{R}^N}\eta(x)dx=1$. For each $n\in\mathbb{N}$, set $\eta_k(x)=k^N\eta(kx)$.  For  $u\in \mathcal E^{s,\alpha,q}(\mathbb R^N)$ we denote $v_k= \eta_k*u$. Since $u\in L_{loc}^1(\mathbb{R}^N),$ we have that $v_k\in C^\infty(\mathbb{R}^N)$. As in \cite[Proposition 2.6]{Mercuri-Moroz-VS-2016} we see that 
$$
\iint_{\mathbb{R}^N\times\mathbb{R}^N}\frac{|(v_k-u)(x)|^q|(v_k-u)(y)|^q}{|x-y|^{N-\alpha}}dxdy \to 0\quad{as}\,\,k\to\infty,
$$
and, for each fixed $k$, it holds
$$
\lim_{|x|\to\infty}|(\eta_k*u)(x)|=0.
$$
In particular, $\eta_k*u\in L^\infty(\mathbb{R}^N)$. On the other hand, as tempered distributions, we have $\widehat{v_k}=\widehat{\eta_k}\widehat{u}$ with $\widehat{\eta_k}(\xi)=\widehat{\eta}(\xi/k)$. Since $\eta_k\in C^\infty_c(\mathbb{R}^N)\subset \mathcal{S}(\mathbb{R}^N)$ (the Schwartz space), we have $\widehat{\eta_k}\in \mathcal{S}(\mathbb{R}^N)$ and
$$
\widehat{\eta_k}(\xi)=\int_{\mathbb{R}^N}e^{-2\pi i x\cdot\xi}\eta_k(x)dx \quad\Rightarrow 
\quad|\widehat{\eta_k}(\xi)|\leq1,\quad\forall \xi\in\mathbb R^N.
$$
Moreover, since $\widehat{\eta_k}$ is a continuous function, we see that
$$
\widehat{\eta_k}(\xi)= \widehat{\eta}(\xi/k) \to\widehat{\eta}(0)=1\quad\mbox{as}\,\,k\to\infty,
$$
for all $\xi\in\mathbb{R}^N$. Hence $\widehat{v_k}\to \widehat{u}$, a.e. on $\mathbb{R}^N$. Since $|\xi|^{2s}|\widehat u(\xi)|^2\in L^1(\mathbb{R}^N$)
 and 
 $$
 |\xi|^{2s}|\widehat{(v_k-u)}(\xi)|^2=|\xi|^{2s}|(\widehat{\eta_k}(\xi)-1)\widehat{u}(\xi)|^2\leq 2|\xi|^{2s}|\widehat{u}(\xi)|^2
 $$
for all $\xi\in\mathbb{R}^N$ and all $k\in\mathbb N$, we may conclude by the dominated convergence theorem that
 $$
 \|v_k-u\|^2_{\dot{H}^s(\mathbb{R}^N)}=\int_{\mathbb{R}^N}|\xi|^{2s}|\widehat{(v_k-u)}(\xi)|^2d\xi\to0\quad{as}\quad k\to\infty,
 $$
and this concludes the proof.
\end{proof}

\vspace{2mm}

We now assume $0<s<1,$ as this allows us to use Gagliardo seminorms.

\begin{proposition}\label{densityE} 
For all $s\in (0,1)$, $q\geq1$ and $\alpha\in(0,N),$  $C^\infty_c(\mathbb{R}^N)$ is dense in $\mathcal E^{s,\alpha,q}(\mathbb R^N)$. 
\end{proposition}
\begin{proof}
For any arbitrary $u\in\mathcal E^{s,\alpha,q}(\mathbb R^N)$ and $\varepsilon>0$, by Lemma \ref{suaves} we can pick $k_0$ such that $\|\eta_{k_0}*u-u\|<\varepsilon/2$.
Fix a function $\varphi\in C^\infty_c(\mathbb{R}^N)$ such that $0\leq\varphi(x)\leq1$, $\varphi\equiv1$ in $B_1(0)$ and $\varphi\equiv0$ in $\mathbb{R}^N\setminus B_2(0)$. For each $k\in\mathbb{N}$ we set $\varphi_k(x)=\varphi(x/k)$.
It is a standard fact that for $ \dot{H}^s$
    if $s\in(0,1)$ (see e.g. \cite[Proposition 1.37]{Bahouri-Chemin-Danchin}) it holds that $\|v\|_{\dot{H}^s}^2=C_{N,s}[v]^2$, where
  \begin{equation*}
     [v]:=\left(\iint_{\mathbb{R}^N\times\mathbb{R}^N}\frac{|v(x)-v(y)|^2}{|x-y|^{N+2s}}dxdy\right)^\frac{1}{2},
  \end{equation*}
  for $v\in \dot{H}^s$.  By \cite[Lemma 1.4.8]{book-Ambrosio}, for each $u\in\mathcal E^{s,\alpha,q}(\mathbb R^N)$
we have that $[\varphi_kv-v]\to0$ as $k\to\infty$. Thus  $\|\varphi_kv-v\|_{\dot{H}^s}\to0$. On the other hand, since $\varphi_kv(x)\to v(x)$ a.e. and $|\varphi_kv(x)-v(x)|\leq|v(x)|$ by the dominated convergence theorem we have 
$\|\varphi_kv-v\|_{Q^{\alpha,q}(\mathbb R^N)}\to0,$ hence $\varphi_kv\to v$ in $\mathcal E^{s,\alpha,q}(\mathbb R^N)$.
In particular, setting $v=\eta_{k_0}*u$, we can pick $k_1$ such that $\|\varphi_{k_1}(\eta_{k_0}*u)-\eta_{k_0}*u\|<\varepsilon/2,$ and therefore $\|\varphi_{k_1}(\eta_{k_0}*u)-u\|<\varepsilon.$ Since $\varphi_{k_1}(\eta_{k_0}*u)\in C^\infty_c(\mathbb{R}^N),$ this concludes the proof.
\end{proof}

\section{Appendix: Regularity results}\label{regularityappendix}

The regularity results we develop here in this fractional Coulomb-Sobolev framework are built based on the classical work of Silvestre \cite{Silvestre}, which we recall for reader's convenience (see Propositions \ref{regularity1} and \ref{regularity2} below). We first prove, in the following lemma, a useful $L^\infty$ estimate in the spirit of the classical work of Trudinger \cite{Trudinger},  which applies to critical growth nonlinearities like $f(x,u)=|u|^{2^*_s-2}u$. See also Proposition 3.2.14 \cite{book-Ambrosio}.

\begin{lemma}\label{L infty 4s+a>N}
Let $s\in(0,1)$, $N\geq2$ and  $\alpha\in(0,N)$ such that $4s+\alpha>N$.
Suppose that $u\in \mathcal{E}^{s,\alpha,2}(\mathbb{R}^N)$ is a solution to
\begin{equation*}\label{eq regularity}
(-\Delta)^su+\left(I_\alpha*u^2\right)u= f(x,u), \quad\textrm{in}\quad\mathbb{R}^N
\end{equation*}
where  $f:\mathbb{R}^N\times\mathbb R\rightarrow\mathbb R$ is a Caratheodory function satisfying
$$
|f(x,t)|\leq c(|t|^{p-1}+|t|^{q-1}),\quad\forall (x,t)\in \mathbb{R}^N\times\mathbb R
$$
for $p,q\in[2^*_{s,\alpha},2^*_s].$ Then $u\in L^r(\mathbb{R}^N)$ for any  $r\in[2^*_{s,\alpha},\infty]$.
\end{lemma}
\begin{proof}
For any $t\in\mathbb R$, set
  \[
    t_{L} =
      \begin{cases}
        t,          & |t|<L,\\
        \operatorname{sgn}(t)\,L, & |t|\ge L,
      \end{cases}
    \qquad
    \gamma(t)=t\,|t_{L}|^{2\beta},
  \]
  where $\beta>0$ and $L>0.$ Define
  \[
    \gamma^{+}(t):=t^{+}\,|t_{L}|^{2\beta},
    \qquad
    \gamma^{-}(t):=t^{-}\,|t_{L}|^{2\beta},
    \qquad
    t^{+}:=\max\{t,0\},
    \quad
    t^{-}:=\min\{t,0\}.
  \]
  Note that $\gamma^{+}$ and $\gamma^{-}$ are non-decreasing.
 Let us define $\Gamma^{\pm}\colon\mathbb R\to\mathbb R$ as
  \[
    \Gamma^{\pm}(t)
      =\pm\int_{0}^{t}
           \bigl[\,(\gamma^{\pm})'(\tau)\bigr]^{1/2}\,d\tau,
  \]
 Observe that
  \[
      \Gamma^{+}(t)=0,\;\; \forall t\le 0;\qquad
      \Gamma^{-}(t)=0,\;\; \forall t\ge 0;\qquad
      \Gamma^{\pm}\ge 0.
  \]
  Moreover, 
  \begin{equation}\label{Gammaporcima}
      \Gamma^{\pm}(t)
        \;\ge\;
        \frac{1}{\beta+1}\,
        |t^{\pm}|\,|t_{L}|^{\beta},
        \qquad\forall t\in\mathbb R,
\end{equation}
and for all \(a,b\in\mathbb R\) it holds that
  \begin{equation}\label{desiGamma}
      (a-b)\,\bigl(\gamma^{\pm}(a)-\gamma^{\pm}(b)\bigr)
      \;\ge\;
      \bigl|\Gamma^{\pm}(a)-\Gamma^{\pm}(b)\bigr|^{2}.
  \end{equation}
  
With these preliminaries in place, we now estimate the positive and negative parts of $u$ separately, in the spirit of \cite[Proposition 3.2.14]{book-Ambrosio}. We deal only with $u^{+},$ as the same argument applies to $u^{-}.$

\medskip

Testing the PDE with $\gamma^+(u)\in \mathcal{E}^{s,\alpha,2}(\mathbb{R}^N),$ we obtain
$$
C_{N,s}\iint_{\mathbb R^N\times\mathbb R^N}\frac{(u(x)-u(y))(\gamma^+(u(x))-\gamma^+(u(y))}{|x-y|^{N+2s}}dxdy+\int_{\mathbb R^N}(I_\alpha*u^2)u\gamma^+(u)dx=\int_{\mathbb R^N}f(x,u)\gamma^+(u)dx.
$$
Since 
$(I_\alpha*u^2)u\gamma^+(u)\geq0,$ by classical Sobolev's embedding of $\dot H^s(\mathbb R^N)$ and by   \eqref{desiGamma}, we obtain
\begin{align}
\nonumber\|\Gamma^+(u)\|^2_{L^{2^*_s}}\leq&\  C\iint_{\mathbb R^N\times\mathbb R^N}\frac{|\Gamma^+(u(x))-\Gamma^+(u(y))|^2}{|x-y|^{N+2s}}dxdy\\\label{vaivoltar}
\leq&\  C\iint_{\mathbb R^N\times\mathbb R^N}\frac{(u(x)-u(y))(\gamma^+(u(x))-\gamma^+(u(y))}{|x-y|^{N+2s}}dxdy\\
\nonumber\leq&\ C\int_{\mathbb R^N}|f(x,u)|\gamma^+(u)dx\leq C\int_{\mathbb R^N}(|u|^{p-2}+|u|^{q-2})(u^+)^2|u_L|^{2\beta}dx,
\end{align}
where $u_L(x)=(u(x))_L$. Hereafter $\beta>0$ is such that $(u^+)^{\beta+1}\in L^2(\mathbb R^N)$. We set $h=2|u^+|^{2^*_s-2}\chi_{\{u\geq1\}}$. Since $p,q\leq 2^*_s$,  by \eqref{Gammaporcima} and \eqref{vaivoltar}, we have
\begin{align}\label{bdd basic ineq}
\|u^+|u_L|^\beta\|_{L^{2^*_s}}^2\leq&\  C(\beta+1)^2\int_{\mathbb R^N}(2+h)(u^+|u_L|^\beta)^2dx,\quad\forall L>0.
\end{align}
Then, using $h\in L^{N/2s}(\mathbb R^N)$, we obtain
\begin{align*}
\|u^+|u_L|^\beta\|_{L^{2^*_s}}^2
\leq&\ C(\beta+1)^2\left((2+K)\int_{\mathbb R^N}(u^+|u_L|^\beta)^2dx+\int_{\{h\geq K\}}h(u^+|u_L|^\beta)^2dx\right)\\
\leq&\ C(\beta+1)^2\left((2+K)\int_{\mathbb R^N}(u^+)^{2\beta+2}dx+\left(\int_{\{h\geq K\}}|h|^{N/2s}dx\right)^{\frac{2s}{N}}\|u^+|u_L|^\beta\|^2_{L^{2^*_s}}\right),
\end{align*}
for any $K>0$. Hence, picking $K=K(\beta)$ such that $C(\beta+1)^2\left(\int_{\{h\geq K\}}|h|^{N/2s}dx\right)^{\frac{2s}{N}}<\frac{1}{2},$ we get
\begin{equation}\label{interation1}
\|u^+|u_L|^\beta\|_{L^{2^*_s}}^2\leq C(\beta)\int_{\mathbb R^N}(u^+)^{2(\beta+1)}dx.
\end{equation}
Therefore, with $\beta=\beta_1$ such that $2(\beta_1+1)=2^*_s$ and letting $L\rightarrow\infty$ in \eqref{interation1}, we conclude that $u^+\in L^{2_s^*(\beta_1+1)}(\mathbb R^N)$. \newline
Going through the same argument again we obtain \eqref{interation1} with $\beta=\beta_2$ such that $2(\beta_2+1)=2_s^*(\beta_1+1)$ and $\beta_2-\beta_1>2s/(N-2s).$ By iteration we may therefore construct a sequence $\beta_n\rightarrow\infty$ such that $u^+\in L^{2_s^*(\beta_n+1)}(\mathbb R^N),$ namely $u^+\in L^r(\mathbb R^N)$ for all $r\in[2^*_{s,\alpha},\infty)$. \newline It is observed at this stage of the proof that at each iteration step $K$ depends on $\beta_n.$ To obtain an $L^\infty$ estimate we now use that $(u^+)^r$ is summable for any $r\geq 2_s^*,$ hence $h\in L^{N/s}(\mathbb R^N),$ and letting $L\rightarrow\infty$ in \eqref{bdd basic ineq}, we get
\begin{eqnarray*}\label{basic ineq for beta}
\|(u^+)^{\beta+1}\|^2_{L^{2^*_s}}
\leq C(\beta+1)^2\int_{\mathbb R^N}
\left(2+h\right)(u^+)^{2(\beta+1)}dx.
\end{eqnarray*}
Holder's inequality now yields
\begin{eqnarray*}
\|(u^+)^{\beta+1}\|^2_{L^{2^*_s}}
\leq C(\beta+1)^2\left[2\|(u^+)^{\beta+1}\|_{L^{2}}^2+
\|h\|_{L^{N/s}}\|(u^+)^{\beta+1}\|_{L^{2}}\|(u^+)^{\beta+1}\|_{L^{2^*_s}}\right].
\end{eqnarray*}
We stress that $C$ in  \eqref{bdd basic ineq} does not depend on $\beta$. By the Cauchy-Schwarz inequality we get, for all $\varepsilon>0,$ that
\begin{eqnarray*}
\|(u^+)^{\beta+1}\|_{L^{2}}\|(u^+)^{\beta+1}\|_{L^{2^*_s}}\leq \frac{1}{4\varepsilon}\|(u^+)^{\beta+1}\|_{L^{2}}^2+\varepsilon\|(u^+)^{\beta+1}\|_{L^{2^*_s}}^2.
\end{eqnarray*}
Picking $\varepsilon>0$ such that
$$
\varepsilon C(\beta+1)^2
\|h\|_{L^{N/s}}=\frac{1}{2},
$$
we obtain
\begin{eqnarray*}
\|(u^+)^{\beta+1}\|^2_{L^{2^*_s}}
\leq C^2(\beta+1)^4
\left(1+\|h\|_{L^{N/s}}^2\right)\|(u^+)^{\beta+1}\|_{L^{2}}^2.
\end{eqnarray*}
Setting $\tilde{C}^2=C^2(1+\|h\|_{L^{N/s}}^2)$, as 
$$
\tilde{C}(1 +\beta)^2 \leq C_1e^{\sqrt{\beta+1}},\quad\forall \beta>0,
$$
for some $C_1>1$ independent of $\beta$, we have
\begin{eqnarray}\label{ind on beta}
\|u^+\|_{L^{2^*_s(\beta+1)}}
\leq [C_1]^{\frac{1}{\beta+1}}e^{\frac{1}{\sqrt{\beta+1}}}\|u^+\|_{L^{2(\beta+1)}}.
\end{eqnarray}
By means of this estimate we can now conclude that $u^+\in L^\infty(\mathbb{R}^N)$ with a standard argument following e.g. \cite[Proposition 3.2.14]{book-Ambrosio} and \cite[Lemma 5.4]{Ub-Pra-Ca}). We pick $\beta=\beta_1=(2^*_s/2)-1$ and construct a sequence $\beta=\beta_k$ satisfying
$$
2(\beta_{k+1}+1)=2^*_s(\beta_k+1)
$$
and by \eqref{ind on beta}
\begin{eqnarray*}
\|u^+\|_{L^{2^*_s(\beta_k+1)}}
\leq C_1^{\sum_{j=1}^k\frac{1}{\beta_j+1}} e^{\sum_{j=1}^k\frac{1}{\sqrt{\beta_j+1}}} \|u^+\|_{L^{2^*_s}}.
\end{eqnarray*}
Since $\beta_{k}+1=(2^*_s/2)^k$ we have
$$
\sum_{j=1}^\infty\frac{1}{\beta_j+1}<\infty\quad\mbox{and}\quad \sum_{j=1}^\infty\frac{1}{\sqrt{\beta_j+1}}<\infty,
$$
hence
$$
\|u^+\|_{L^\infty}=\lim_{k\to\infty}\|u^+\|_{L^{2^*_s(\beta_k+1)}}
\leq C_2\|u^+\|_{L^{2^*_s}},
$$
and this concludes the proof.
\end{proof}

\begin{remark}
When working with radial functions, by the same argument we may relax the growth condition as $|f(|x|,t)|\leq c(|t|^{p-1}+|t|^{q-1})$ with $p,q\in(p_{\rm{rad}},2^*_s],$ provided $\alpha\in(1,N).$ 
\end{remark}

We now deal with the boundedness of the Riesz potential $I_{\alpha}*u^{2}$ when $u$ lies in some Coulomb-Sobolev space. 

\begin{lemma}\label{Phi_u in L infinite}
Assume that $u\in L^{r}(\mathbb R^N)$, for some $r>2N/\alpha$, and one of the following cases occur:
\begin{description}
    \item[i)] $\alpha\in(1,N)$ and $u\in \mathcal{E}^{s,\alpha,2}_{rad}(\mathbb{R}^N)$;
    \item[ii)] $\alpha\in(0,N)$, $u\in \mathcal{E}^{s,\alpha,2}(\mathbb{R}^N)$ with $s\in(0,1)$ such that $4s+\alpha<N$;
    \item[iii)] $\alpha\in(0,N)$ , $u\in \mathcal{E}^{s,\alpha,2}(\mathbb{R}^N)$ for $s\in(0,1)$ such that $4s+\alpha>N$ and $1>\frac{2s(2\alpha-N)}{\alpha(N-\alpha)}$;
\end{description}
   Then, $I_\alpha*u^2\in L^\infty(\mathbb R^N)$.
\end{lemma}
\begin{proof}
Observe that,  if we can find  $q<2N/\alpha$ such that $u\in L^{q}(\mathbb R^N)$ then 
\begin{eqnarray*}
    I_\alpha*u^2(x)&=&\int_{B_1(x)}\frac{u^2(y)}{|x-y|^{N-\alpha}}dy+\int_{B_1^c(x)}\frac{u^2(y)}{|x-y|^{N-\alpha}}dy\\
    &\leq& \|u\|^2_{L^r}\left(\int_{B_1(0)}\frac{1}{|z|^{(N-\alpha)(\frac{r}{2})'}}dz\right)^{(\frac{r}{2})'}+\|u\|^2_{L^q}\left(\int_{B^c_1(0)}\frac{1}{|z|^{(N-\alpha)(\frac{q}{2})'}}dz\right)^{(\frac{q}{2})'}\leq C
\end{eqnarray*}
 almost everywhere on $\mathbb{R}^N$. 
When $u\in \mathcal{E}^{s,\alpha,2}_{rad}(\mathbb{R}^N)$, with $\alpha\in(1,N)$, this is always possible for $4s+\alpha\neq N$, since  $p_{\rm  rad}<2N/\alpha$ and $u\in L^{p}(\mathbb R^N)$ for all $p_{\rm rad}<p\leq r$. Now, assume that (ii) holds. Since $4s+\alpha<N$ implies $2^*_{s,\alpha}<2N/\alpha$, we can use $q=2^*_{s,\alpha}$ to conclude that $I_\alpha*u^2\in L^\infty(\mathbb R^N)$. Finally, in the case (iii), we also have $2^*_{s,\alpha}<2N/\alpha$ and this concludes the proof. 
\end{proof}

As anticipated, we recall in the next two proposition, some regularity results from Silvestre  \cite{Silvestre}, based on which we will state and prove the main regularity results of this section.

\begin{proposition}[{\cite[Proposition 2.8]{Silvestre}}]\label{regularity1}
Let \(w=(-\Delta)^{s}u\).  
Assume \(w\in C^{0,\tau}(\mathbb{R}^{N})\) and \(u\in L^{\infty}(\mathbb{R}^{N})\) for
\(\tau\in(0,1]\) and $0<s<1$.

\begin{itemize}
\item If \(\tau+2s\le 1\), then \(u\in C^{0,\tau+2s}(\mathbb{R}^{N})\).
      Moreover,
      \[
         \|u\|_{C^{0,\tau+2s}}
           \le C\bigl(\|u\|_{L^{\infty}}
                     +\|w\|_{C^{0,\tau}}\bigr),
      \]
      where \(C\) depends only on \(N,\tau,s\).

\item If \(\tau+2s>1\), then \(u\in C^{1,\tau+2s-1}(\mathbb{R}^{N})\).
      Moreover,
      \[
         \|u\|_{C^{1,\tau+2s-1}}
           \le C\bigl(\|u\|_{L^{\infty}}
                    +\|w\|_{C^{0,\tau}}\bigr),
      \]
      for a constant \(C\) depending only on \(N,\tau,s\).
\end{itemize}
\end{proposition}

\begin{proposition}[{\cite[Proposition 2.9]{Silvestre}}]\label{regularity2}
Let \(w=(-\Delta)^{s}u\).  
Assume \(w\in L^{\infty}(\mathbb{R}^{N})\) and \(u\in L^{\infty}(\mathbb{R}^{N})\) for $0<s<1$.

\begin{itemize}
\item If \(2s\le 1\), then \(u\in C^{0,\tau}(\mathbb{R}^{N})\) for every \(\tau<2s\).
      Moreover,
      \[
         \|u\|_{C^{0,\tau}}
           \le C\bigl(\|u\|_{L^{\infty}(\mathbb{R}^{N})}
                    +\|w\|_{L^{\infty}}\bigr),
      \]
      where \(C\) depends only on \(N,\tau,s\).

\item If \(2s>1\), then \(u\in C^{1,\tau}(\mathbb{R}^{N})\) for every \(\tau<2s-1\).
      Moreover,
      \[
         \|u\|_{C^{1,\tau}}
           \le C\bigl(\|u\|_{L^{\infty}}
                  +\|w\|_{L^{\infty}}\bigr),
      \]
      with \(C\) depending only on \(N,\tau,s\).
\end{itemize}
\end{proposition}

We are now in the position to state and prove our main regularity result.

\begin{proposition}\label{regularitytheorem}
Suppose $u\in \mathcal{E}^{s,\alpha,2}(\R^N)\cap L^{\infty}(\mathbb R^N)$ is a weak solution to
\begin{equation}\label{eq daregularity}
(-\Delta)^su+\left(I_\alpha*u^2\right)u= f(u), \quad\textrm{in}\quad\R^N
\end{equation} 
where $s\in(0,1)$,  $\alpha\in(0,N)$ and $f\in C^{1,\tau}_{\rm{loc}}(\mathbb R)$ for all $0<\tau<1$. Suppose also that $I_{\alpha}*u^2\in L^\infty(\Omega)$. Then, $u\in C^{2,\tau}(\mathbb R^N)$ for all $0<\tau<2s$ if $s\leq 1/2$ and for all $0<\tau<2s-1$ if $s> 1/2$.
\end{proposition}

\begin{proof}
 We have $(-\Delta)^su=f(u)-\left(I_\alpha*u^2\right)u:=\tilde f(x)$, with $\tilde f\in L^\infty(\mathbb R^N)$ as $u\in L^\infty(\mathbb R^N)$. Using Proposition \ref{regularity2}, we distinguish two cases: (1) if $s>1/2$ then $u\in C^{1,\tau}(\mathbb R^N)$ for every $0<\tau<2s-1$; (2) if $s\leq 1/2,$ then  $u\in C^{0,\tau}(\mathbb R^N)$ for every $0<\tau<2s$. 

\medskip

Let us assume $s\leq 1/2,$ and use the fact that $I_\alpha*u^2$ weakly solves the fractional Laplacian equation $(-\Delta)^{\alpha/2}v=u^2.$  Since $I_\alpha*u^2\in L^\infty(\mathbb R^N)$ and $u^2\in C^{0,\tau}(\mathbb R^N)$, we may apply Proposition \ref{regularity1}, to get $I_\alpha*u^2\in C^{0,\tau+\alpha}(\mathbb R^N)$ in the case $\tau+\alpha\leq 1$ or 
$I_\alpha*u^2\in C^{1,\tau+\alpha-1}(\mathbb R^N)$ in the case $\tau+\alpha> 1.$ When $\alpha/2>1$ then write $(-\Delta)^{\alpha/2}$ as $(-\Delta)^k(-\Delta)^\eta$ with $\alpha/2=k+\eta$, $k\in\mathbb N$, $0<\eta<1$ and combine standard elliptic regularity with Proposition \ref{regularity1}. 

\medskip

Since in all cases $I_\alpha*u^2\in C^{0,\tilde \tau}(\mathbb R^N)$ for all $\tilde \tau<\min\{2s+\alpha,1\}$ (in fact, $\tilde\tau$ can be 1 if $2s+\alpha>1$), it follows that $\tilde f(x):= f(x,u)-\left(I_\alpha*u^2\right)\in C^{0,\tau}(\mathbb R^N)$ for all $0<\tau<2s$ (since $\tilde\tau>\tau$). This implies, again by Proposition \ref{regularity1} that either $u\in C^{0,\tau+2s}(\mathbb R^N)$ if $\tau+2s\leq1,$ or $u\in C^{1,\tau+2s-1}(\mathbb R^N)$ if $\tau+2s>1$. If $\tau+2s\leq1$ still occurs,  we can repeat the process to infer that $I_\alpha*u^2\in C^{0,\tilde \tau}(\mathbb R^N)$ for all $\tilde \tau<\min\{\tau+2s+\alpha,1\}$ and then either $u\in C^{0,\tau+4s}(\mathbb R^N)$ if $\tau+4s\leq1,$ or $u\in C^{1,\tau+4s-1}(\mathbb R^N)$ if $\tau+4s>1$. After a finite number of steps, say $k,$ we obtain $\tau+2ks>1$ and hence $u\in C^{1,\tau}(\mathbb R^N),$ for any $\tau$ satisfying $0<\tau<2s$ if $s\leq 1/2,$ or $0<\tau<2s-1$ if $s\geq 1/2$. Moreover, we have $I_\alpha*u^2\in C^{1,\tilde \tau}(\mathbb R^N)$ for every $0<\tilde\tau<\alpha$  if $\alpha\leq1,$ or $0<\tilde\tau<\alpha-1$ if $\alpha>1$.

\medskip

To conclude, we note that $\partial u/\partial x_i$ is a solution to $(-\Delta)^s(\partial u/\partial x_i)=\partial[f(u)-\left(I_\alpha*u^2\right)u]/\partial x_i:=\bar f(x)$. Hence, by Proposition \ref{regularity1} and Proposition $\ref{regularity2}$  $\partial u/\partial x_i$ and $\bar f$ are essentially bounded and so carrying out the same arguments above we get $\partial u/\partial x_i\in C^{1,\tau}(\mathbb R^N)$ every $0<\tau<2s$ or $2s-1$ depending on $s$ as in the statement, and this concludes the proof.
\end{proof}

\section{Appendix: Pohozaev Identity}

This section is devoted to proving a Pohozaev-type identity related to PDEs studied in this paper. The proof  employes the celebrated Caffareli-Silvestre extension and the arguments are in the spirit of \cite{book-Ambrosio,Moroz-Van-2013}. 

\begin{lemma}(Pohozaev identity)\label{pohozaev}
Suppose that $u\in \mathcal{E}^{s,\alpha,2}(\mathbb{R}^N)\cap C^{2,\gamma}_{loc}(\mathbb R^N)$, for some $0<\gamma<1$, is a solution to
\begin{equation}\label{eq poh}
(-\Delta)^su+\left(I_\alpha*u^2\right)u= f(u), \quad\textrm{in}\quad\mathbb{R}^N
\end{equation}
{where $s\in(0,1)$, $N\geq2$,  $\alpha\in(0,N)$, $4s+\alpha\neq N$ and $f:\mathbb R\rightarrow\mathbb R$ is continuous, such that $f(0)=0$ and $F(u)\in L^1(\mathbb R^N)$}. Then $u$ satisfies
\begin{equation}\label{Poh general}
\frac{N-2s}{2}\|(-\Delta)^{\frac{s}{2}}u\|_2^2+\frac{C_\alpha(N+\alpha)}{4}\iint_{\mathbb{R}^N\times\mathbb{R}^N}\frac{u^2(x)u^2(y)}{|x-y|^{N-\alpha}}dxdy
=N\int_{\mathbb{R}^N} F(u)dx.
\end{equation}
\end{lemma}
\begin{proof}
As it is customary for proving this kind of identities, we integrate \eqref{eq poh} after multiplying by $x\cdot \nabla u$. To this aim we use the characterization of $(-\Delta)^su$ by means of the Caffarelli-Silvestre extension (see \cite{Caffarelli-Silvestre}). \newline

We recall that on $\mathbb R^N\times\mathbb R$, the Poisson kernel $P_s(x, y)$ is defined as
\[
P_s(x, y) = p_{N,s} \frac{y^{2s}}{(|x|^2 + y^2)^{\frac{N + 2s}{2}}}, \quad 
p_{N,s} = \pi^{-\frac{N}{2}} \frac{\Gamma\left(\frac{N + 2s}{2}\right)}{\Gamma(s)},
\]

and that the $s$‑harmonic extension \(v(x,y)=P_s(x,y)*u(x)\) of \(u\) solves
\begin{equation}\label{vextu}
\begin{cases}
\displaystyle\operatorname{div}\bigl(y^{1-2s}\nabla v\bigr)=0 & \text{in }\mathbb{R}^{N+1}_{+},\\[4pt]
v(\cdot,0)=u & \text{on }\mathbb{R}^{N}\simeq \partial\mathbb R^{N+1}_+,\\[4pt]
\displaystyle\frac{\partial v}{\partial\nu^{1-2s}}=\kappa_s\,(-\Delta)^su & \text{on }\mathbb{R}^{N},
\end{cases}
\end{equation}
where \(\kappa_s = \Gamma(1 - s)2^{2s - 1}/ \Gamma(s)\) and 
\[
\frac{\partial v}{\partial \nu^{1-2s}}(x, 0) = - \lim_{y \to 0} y^{1 - 2s} \frac{\partial v}{\partial y}(x, y).
\]
Moreover, it holds that
\begin{equation}\label{normaextension}
\iint_{\mathbb{R}_+^{N+1}} y^{1-2s} |\nabla v|^2 \, dx \, dy = \kappa_s \int_{\mathbb{R}^N} \left| (-\Delta)^{\frac{s}{2}} u \right|^2 dx,
\end{equation}
$v(\cdot,\delta)\rightarrow u$ in $\dot{H}^s(\mathbb R^N)$ and $-\delta^{1-2s}\displaystyle\frac{\partial v}{\partial y}(\cdot,\delta)\rightarrow k_s(-\Delta)^su$ in $\dot{H}^{-s}(\mathbb R^N)$ as $\delta\rightarrow 0$. 

\medskip

Now, since $u\in C^{2,\gamma}_{loc}(\mathbb R^N)$, regularity theory gives $v\in C^2\left(\overline{\mathbb R^{N+1}_+}\right)$. So, we can perform point-wise calculations on the first equation of \eqref{vextu}. Following e.g. \cite[Theorem 3.5.1]{book-Ambrosio}, we multiply this equation by $(x,y)\cdot\nabla v$, integrating in half-balls $D_{R,\delta}=\{|(x,y)|\leq R\}\cap\{y\geq\delta\}$, and take the limit letting $\delta\rightarrow0$ first, and then $R\rightarrow\infty.$ The delicate issue after taking the first limit is to handle an integral in $B_R(0)$ involving the term $(x\cdot\nabla u)I_\alpha*u^2$. In the spirit of \cite{Moroz-Van-2013}, we use a smooth cut-off function  $\varphi\in C^\infty_c(\mathbb R^{N+1})$ such that $\varphi\equiv 1$ in $B_1(0)$, $\varphi\equiv 0$ in $\mathbb R^{N+1}\backslash B_2(0)$ and for any $\varepsilon>0,$ set $\varphi_\varepsilon(x,y)=\varphi(\varepsilon(x,y))$. Multiply the first equation in \eqref{vextu} by $((x,y)\cdot\nabla v)\varphi_\varepsilon$. Define
$$
P:=y^{1-2s}\left[((x,y)\cdot\nabla v)\nabla v-\frac{|\nabla v|^2}{2}(x,y)\right],
$$
$\Omega_\delta=\mathbb R^N\times[\delta,+\infty)$
and observe that
\begin{align}\nonumber
0&=\iint_{\Omega_\delta}\operatorname{div}\bigl(y^{1-2s}\nabla v\bigr)((x,y)\cdot\nabla v)\varphi_\varepsilon dxdy=\iint_{\Omega_\delta}\varphi_\varepsilon\operatorname{div}P+\varphi_\varepsilon \frac{N-2s}{2}y^{1-2s}|\nabla v|^2 dxdy\\\label{I1+I2+I3=0}
&=\iint_{\Omega_\delta}\operatorname{div}(\varphi_\varepsilon P) dxdy+\frac{N-2s}{2}\iint_{\Omega_\delta}\varphi_\varepsilon y^{1-2s}|\nabla v|^2 dxdy-\iint_{\Omega_\delta} \nabla\varphi_\varepsilon \cdot Pdxdy\\\nonumber
&=I_{1,\delta,\varepsilon}+\frac{N-2s}{2}I_{2,\delta,\varepsilon}-I_{3,\delta,\varepsilon}.
\end{align}
Let us now handle these three integrals above separately, and for a fixed $\varepsilon>0.$  In order to perform $\delta\rightarrow 0,$ set $\tilde{\varphi}_\varepsilon(x)=\varphi_\varepsilon(x,0)$ in $\mathbb R^N=\partial\mathbb R_+^{N+1}$ and note that since

\begin{align*}
I_{1,\delta,\varepsilon}&=\int_{\partial\Omega_\delta}\varphi_\varepsilon(P\cdot\nu) dS\\ 
&= \int_{\{y=\delta\}}\varphi_\varepsilon\delta^{1-2s}\left(-\frac{\partial v}{\partial y}(x,\delta)\right)((x,y)\cdot\nabla v)dS+\frac{\delta}{2}\int_{\{y=\delta\}}\varphi_\varepsilon y^{1-2s}|\nabla v|^2dS\\
&=\int_{\{y=\delta\}}\varphi_\varepsilon\delta^{1-2s}\left(-\frac{\partial v}{\partial y}(x,\delta)\right)\left(x\cdot\nabla_x v(x,\delta)+\delta\frac{\partial v}{\partial y}(x,\delta)\right)dS+o_\delta(1)\\
&=\int_{\{y=\delta\}}\varphi_\varepsilon\delta^{1-2s}\left(-\frac{\partial v}{\partial y}(x,\delta)\right)\left(x\cdot\nabla_x v(x,\delta)\right)dS+o_\delta(1),
\end{align*}
by the properties of the extension function $v$, it holds that
\begin{equation}\label{I1epsilon}
\lim_{\delta\rightarrow0}I_{1,\delta,\varepsilon}=k_s\int_{\mathbb R^N}\tilde{\varphi}_\varepsilon(-\Delta)^su(x\cdot\nabla u)dx.
\end{equation} 
On the other hand, by \eqref{normaextension}, we get
\begin{equation}\label{I2epsilon}
\lim_{\delta\rightarrow0}I_{2,\delta,\varepsilon}=\iint_{\mathbb R_+^{N+1}}\varphi_\varepsilon y^{1-2s}|\nabla v|^2 dxdy.
\end{equation}
Hence, by \eqref{I1+I2+I3=0} for any $\varepsilon>0,$ the $\lim_{\delta\rightarrow0}I_{3,\delta,\varepsilon}$ exists. Moreover, 
\begin{align*}
|I_{3,\delta,\varepsilon}|&\leq\iint_{R^{N+1}_+}|P||\nabla\varphi_\varepsilon|dxdy\leq \varepsilon\|\nabla\varphi\|_{L^\infty}\frac{3}{2}\iint_{\left\{\frac{1}{\varepsilon}\leq|(x,y)|\leq\frac{2}{\varepsilon}\right\}}y^{1-2s}|\nabla v|^2|(x,y)|dxdy\\
&\leq3\|\nabla\varphi\|_{L^\infty}\iint_{\left\{|(x,y)|\geq\frac{1}{\varepsilon}\right\}}y^{1-2s}|\nabla v|^2dxdy.
\end{align*}
As a consequence, letting $\delta\rightarrow 0$ first and then $\varepsilon\rightarrow 0,$ $I_{3,\delta,\varepsilon}$ goes to zero. Furthermore, as far as $I_{2,\delta,\varepsilon},$ by \eqref{normaextension} and \eqref{I2epsilon} we have
$$
\lim_{\epsilon\rightarrow 0}\lim_{\delta\rightarrow0}I_{2,\delta,\varepsilon}=\kappa_s \int_{\mathbb{R}^N} \left| (-\Delta)^{\frac{s}{2}} u \right|^2 dx.
$$
Putting all these together and letting $\delta\rightarrow 0$  first and then $\varepsilon\rightarrow 0$ in \eqref{I1+I2+I3=0}, by \eqref{I1epsilon} we get
\begin{equation}\label{pohofirstpart}
\lim_{\varepsilon\rightarrow 0}\int_{\mathbb R^N}\tilde{\varphi}_\varepsilon(-\Delta)^su(x\cdot\nabla u)dx=
\frac{2s-N}{2}\int_{\mathbb{R}^N} \left| (-\Delta)^{\frac{s}{2}} u \right|^2 dx.
\end{equation}
We stress that so far we have used only that  $u\in C^{2,\gamma}_{loc}(\mathbb R^N)\cap \dot{H}^s(\mathbb R^N).$ To conclude, we finally take into account that $u\in \mathcal{E}^{s,\alpha,2}(\mathbb{R}^N)$ and satisfies \eqref{eq poh}, by which we obtain
\begin{align*}
\int_{\mathbb R^N}\tilde{\varphi}_\varepsilon(-\Delta)^su(x\cdot\nabla u)dx=&\int_{\mathbb R^N}f(u)\tilde{\varphi}_\varepsilon(x\cdot\nabla u)dx-\int_{\mathbb R^N}(I_\alpha*u^2)u\tilde{\varphi}_\varepsilon(x\cdot\nabla u)dx\\
=&I_{4,\varepsilon}-I_{5,\varepsilon}.
\end{align*}
We also note that arguing as in \cite[Proposition 3.1]{Moroz-Van-2013} we have
\begin{equation*}
\lim_{\varepsilon\rightarrow 0}I_{5,\varepsilon}=-\frac{N+\alpha}{4}\int_{\mathbb R^N}(I_\alpha*u^2)u^2dx.
\end{equation*}
Finally, since
\begin{align*}
    I_{4,\varepsilon}=&\int_{B_{2/\varepsilon}(0)}\big[\operatorname{div}(\tilde{\varphi}_\varepsilon F(u)x)-N\tilde\varphi_\varepsilon F(u)-(\nabla\tilde\varphi_\varepsilon,x)F(u)\big]dx\\
    =&-N\int_{\mathbb R^N}\tilde\varphi_\varepsilon F(u)dx-\int_{\{\frac{1}{\varepsilon}\leq|x|\leq\frac{2}{\varepsilon}\}}(\nabla\tilde\varphi_\varepsilon,x)F(u)dx,
\end{align*}
and that by the fact $F(u)\in L^1{(\mathbb R^N)}$, the last integral above goes to zero as $\varepsilon\rightarrow 0.$ 
Hence
$$
I_{4,\varepsilon}\rightarrow -N\int_{\mathbb R^N} F(u)dx,
$$
and therefore
$$
\lim_{\varepsilon\rightarrow 0}\int_{\mathbb R^N}\tilde{\varphi}_\varepsilon(-\Delta)^su(x\cdot\nabla u)dx=
-N\int_{\mathbb R^N} F(u)dx+\frac{N+\alpha}{4}\int_{\mathbb R^N}(I_\alpha*u^2)u^2dx,
$$
which combined with \eqref{pohofirstpart} yields the conclusion. 
\end{proof}

\begin{remark}
Conditions which guarantee the required regularity for this Pohozaev identity in the case $4s + \alpha > N$ are provided in Appendix \ref{regularityappendix}. The case $4s + \alpha < N$ seems an interesting open question to be addressed in future. 
\end{remark}

\section*{Acknowledgements}
Elisandra Gloss and Bruno Ribeiro would like to thank Conselho Nacional de Desenvolvimento Científico e Tecnol\'ogico (CNPq) for the support; in particular, for the grants 201454/2024-6 and 443594/2023-6 (E.G.),  grants 314111/2023-9, 443594/2023-6, 201452/2024-3 (B.R.). Carlo Mercuri is a member of the group GNAMPA of Istituto Nazionale di Alta Matematica (INdAM).

\section*{Conflict of interest statement}

The authors have no relevant financial or non-financial interests to disclose.
The authors have no conflicts of interest to declare that are relevant to the content of this article.
All authors certify that they have no affiliations with or involvement in any organization or entity with any financial interest or non-financial interest in the subject matter or materials discussed in this manuscript.
The authors have no financial or proprietary interests in any material discussed in this article.

\section*{Data availability statement}
This manuscript has no associated data.

\section*{Human and animal studies}
The content of this manuscript did not involve studies on humans nor animals.

\end{document}